\documentclass[aos,reqno]{imsart}

\RequirePackage{amsthm,amsmath,amsfonts,amssymb}
\RequirePackage[authoryear,round]{natbib}
\RequirePackage[colorlinks,citecolor=blue,urlcolor=blue]{hyperref}
\RequirePackage{graphicx}

\usepackage[dvipsinames]{xcolor}
\definecolor{bordeau}{rgb}{0.3515625,0,0.234375}
\usepackage[absolute,overlay]{textpos}

\usepackage{subcaption}
\usepackage{multicol}
\usepackage{afterpage}
\usepackage{diagbox}
\usepackage{pgffor}
\usepackage{url}            % simple URL typesetting
\usepackage{booktabs}       % professional-quality tables
\usepackage{nicefrac}       % compact symbols for 1/2, etc.
\usepackage{microtype}      % microtypography
\usepackage{dsfont}            % for indicator function
\usepackage{import}
\usepackage[toc,page]{appendix}
\usepackage[tight]{minitoc}
\usepackage{etoc}
\usepackage{mathrsfs}
\usepackage{tikz}
\usepackage{cleveref}
\usepackage{autonum}
\usepackage{algorithm}
\usepackage{algpseudocode}
\usepackage{enumitem}

\startlocaldefs
\theoremstyle{plain}
\newtheorem{theorem}{Theorem}
\newtheorem{prop}{Proposition}
\newtheorem{lemma}{Lemma}

\theoremstyle{remark}

\newtheorem{assumption}{Assumption}

\newtheorem{corollary}{Corollary}
\newtheorem{remark}{Remark}
\allowdisplaybreaks

\newcommand*\circled[1]{\tikz[baseline=(char.base)]{
            \node[shape=circle,draw,inner sep=2pt] (char) {#1};}}
            
\definecolor{darkmidnightblue}{HTML}{003366}    
\definecolor{midnightblue}{HTML}{0059b3}
\definecolor{chromered}{HTML}{f14233}

\definecolor{darkpowderblue}{rgb}{0.0, 0.2, 0.6}
\definecolor{dukeblue}{rgb}{0.0, 0.0, 0.61}

\hypersetup{
    colorlinks = true,
    citecolor= darkmidnightblue,
    urlcolor=darkmidnightblue,
    breaklinks=true,
    linkcolor = darkmidnightblue,
    linkbordercolor = {white},
}

\newcommand{\dis}{\displaystyle}

\newcommand\prob{\mathbb{P}}

\newcommand\calX{\mathcal{X}}
\newcommand\calY{\mathcal{Y}}
\newcommand\calL{\mathcal{L}}
\newcommand\calN{\mathcal{N}}
\newcommand\calG{\mathcal{G}}
\newcommand\calC{\mathcal{C}}

\newcommand\calF{\mathcal{F}}
\newcommand\calM{\mathcal{M}}
\newcommand\calA{\mathcal{A}}
\newcommand\calT{\mathcal{T}}
\newcommand\calP{\mathcal{P}}
\newcommand\calQ{\mathcal{Q}}

\newcommand{\calW}{\mathcal{W}}
\newcommand{\ZZ}{\mathcal{Z}}
\newcommand{\KL}{D_{\sf KL}}

\newcommand\bThetaEWA{\widehat\bTheta{}^{\sf EWA}}
\newcommand\bThetaLS{\widehat\bTheta{}^{\sf LS}}
\newcommand\WLS{\widehat W{}^{\sf LS}}
\newcommand\WEWA{\widehat W{}^{\sf EWA}}
\newcommand\bs{\boldsymbol}
\newcommand{\Img}{\mathop{\rm Im}}

\def\Cr{C^{\textrm{prod}}}

\def\Zl{Z^{\textrm{cust}}}
\def\Zr{Z^{\textrm{prod}}}
\def\hZl{\widehat Z^{\textrm{cust}}}

\newcommand\bfE{\mathbf E}
\newcommand\bfP{\mathbf P}
\newcommand\bfQ{\mathbf Q}

\newcommand\bfH{\mathbf H}

\newcommand\bfK{\mathbf K}
\newcommand\bfL{\mathbf L}
\newcommand\bfM{\mathbf M}
\newcommand\bfZ{\mathbf Z}

\newcommand\bfA{\mathbf A}
\newcommand\bfa{\mathbf a}
\newcommand\bfb{\mathbf b}
\newcommand\bfB{\mathbf B}

\def\bfZleft{\mathbf Z^{\textrm{cust}}}
\def\bfZright{\mathbf Z^{\textrm{prod}}}
\def\hatbfZleft{\widehat{\mathbf Z}^{\textrm{cust}}}
\def\hatbfZright{\widehat{\mathbf Z}^{\textrm{prod}}}

\newcommand\bB{\boldsymbol B}

\newcommand\bZ{\boldsymbol Z}

\newcommand\bQ{\boldsymbol Q}

\newcommand\bH{\boldsymbol H}
\newcommand\bK{\boldsymbol K}
\newcommand\bU{\boldsymbol U}
\newcommand\bV{\boldsymbol V}

\newcommand\bXi{\boldsymbol\Xi}
\newcommand\bTheta{\boldsymbol\Theta}

\newcommand\bxi{\boldsymbol\xi}

\newcommand{\balpha}{\boldsymbol\alpha}

\newcommand\bx{\boldsymbol x}

\newcommand\bu{\boldsymbol u}
\newcommand\bv{\boldsymbol v}
\newcommand\bw{\boldsymbol w}

\newcommand\var{\mathbf{Var}}
\newcommand\Ex{\mathbb E}

\DeclareMathOperator{\tr}{tr}

\newcommand\bzeta{\bs \zeta}
\newcommand\bchi{\bs\chi}

\def\tilde{\widetilde}
\def\hat{\widehat}

\def\bfE{\mathbf E}

\def\bfs{\boldsymbol s}
\def\bft{\boldsymbol t}

\crefname{cor}{Corollary}{Corollaries}
\crefname{condition}{Condition}{C}

\endlocaldefs

\begin{document}

\begin{frontmatter}

\title{Graphon Estimation in bipartite graphs with 
observable edge labels and unobservable node labels}
\runtitle{Graphon Estimation in bipartite graphs}

\begin{aug}
    \author[A]{\fnms{Etienne}~\snm{Donier-Meroz}\ead[label=e1]{Etienne.Donier-Meroz@ensae.fr }},
    \author[B]{\fnms{Arnak S.}~\snm{Dalalyan}\ead[label=e2]{arnak.dalalyan@ensae.fr}},
    \author[B]{\fnms{Francis}~\snm{Kramarz}\ead[label=e3]{francis.kramarz@ensae.fr}},
    \author[B]{\fnms{Philippe}~\snm{Chon\'e}\ead[label=e4]{philippe.chone@ensae.fr}}
    \and
    \author[B]{\fnms{Xavier}~\snm{D'Haultfoeille}\ead[label=e5]{xavier.dhaltfoeuille@ensae.fr}}
    
    \address[A]{CREST, CNRS, Institut Polytechnique de Paris\printead[presep={,\ }]{e1}}
    
    \address[B]{CREST, ENSAE, Institut Polytechnique de Paris\printead[presep={,\ }]{e2,e3,e4,e5}}
    \end{aug}

\begin{abstract}
    Many real-world data sets can be presented in the form of a matrix whose entries correspond to the interaction between two entities of different natures (the choice of a product by a customer, the number of times a web user visits a web page, a student's grade in a subject, a patient's rating of a doctor, etc.). We assume in this paper that the mentioned interaction is determined by unobservable latent variables describing each entity. Our objective is to estimate the conditional expectation of the data matrix given the unobservable variables. This is presented as a problem of estimation of a bivariate function referred to
    as graphon. We study the cases of piecewise constant and H\"older-continuous graphons. We
    establish finite sample risk bounds for the least squares estimator and the exponentially weighted aggregate. These bounds highlight the dependence of the estimation error on the size of the data set, the maximum intensity of the interactions, and the level of noise. As the analyzed least-squares estimator is intractable, we propose an adaptation of Lloyd's alternating minimization algorithm to compute an approximation of the least-squares estimator. Finally, we present numerical experiments 
    in order to illustrate the empirical performance of the graphon estimator on synthetic data sets. 
\end{abstract}



\end{frontmatter}

\dosecttoc
\faketableofcontents % Run a fake tableofcontents command for the partocs

\section{Introduction}

In this paper, we consider the problem of estimating
the conditional mean of a random matrix generated by 
a bivariate graphon and (unobserved) latent variables. 
More precisely, let $n$ and $m$ be two positive 
integers assumed to be large, and $\bfH$ be an 
$n \times m$ random matrix with real entries 
$H_{i,j}$.  We assume that the distribution of 
this matrix $\bfH$ satisfies the following 
condition.

\begin{assumption}\label{ass:1}
There is a function $W^*:[0,1]^2\to \mathbb R$,
called the graphon, and two random vectors $\bU 
= (U_1,\ldots, U_n)$ and $\bV = (V_1,\ldots, V_m)$
such that
\vspace{-3pt}
\begin{itemize}\itemsep=0pt
    \item[{\bf A 1.1}] the random variables 
    $U_1,\ldots,U_n,V_1,\ldots,V_m$ 
    are independent and drawn from the uniform 
    distribution $\mathcal U([0,1])$. 
    \item[{\bf A 1.2}] conditionally to $(\bU,\bV)$, 
    the entries $H_{i,j}$ are independent
    and %satisfy 
    $\Ex[H_{i,j}|\bU,\bV] = W^*(U_i,V_j)$.
\end{itemize}
\end{assumption}

The aforementioned setting corresponds to the practical 
situation in which there are $n$ customers and $m$ products. 
Each customer has an unobserved latent feature $U$ and each 
product has an unobserved latent feature $V$. We observe 
the label $H$ that characterizes the interaction between 
the customer and the product (typically, that the customer chose the product). The function $W^*$ corresponds to 
the mean value of the interaction for given values of
the latent features.

The scope of applicability for both the aforementioned framework and the results derived in this paper encompasses the following main examples:
\vspace{-4pt}
\begin{enumerate}
    \itemsep=0pt   
    \item The entries $H_{i,j}$ take the values $0$ or $1$
    and correspond to the presence of an edge in the
    bipartite graph. In the example of customers and
    products, one might set $H_{i,j}=1$ if and only if
    customer $i$ has already bought the product $j$. It is
    important in this setting to take into consideration
    the case of large and sparse graphs, in which the
    probabilities of having an edge between two nodes are
    small for all pairs of nodes. 
    
    \item The provided data is composed of an aggregation of
    $N$ instances of the customer-product graph outlined in the preceding paragraph. Each entry $H_{i,j}$ signifies the observed frequency of customer $i$ purchasing product $j$. For instance, in a scenario where the
    customer visited the store $N = 10$ times and $H_{i,j}
    = 3/10$, it indicates that the customer acquired
    product $i$ on three out of ten occasions. In this
    context, the range of values for $H_{i,j}$ spans from 0
    to 1. Assuming independence across the aforementioned
    $N$ trials, the variance of $H_{i,j}$ scales with the
    order of $1/N$. As a result, this framework exhibits
    the distinctive trait of having low-noise variance.
    This particular instance can be encompassed within a
    broader context of a sub-Gaussian distribution
    characterized by a small variance parameter $\sigma^2$,
    which corresponds to $1/N$.
    
    \item The provided data corresponds to observations collected within a time window of duration $T>0$. Each $H_{i,j}$ represents the average occurrences of customer $i$ purchasing product $j$ over the time interval with a length of $T>0$. This scenario can be expressed using the formula $H_{i,j} = N_{i,j}(T)/T$, where $N_{i,j}(T)$ denotes a Poisson random variable characterized by an intensity proportional to $T$.  
    The parameter $T$ may be large, thereby facilitating accurate estimation. It is important to distinguish this example from the prior one in which the values of $H_{i,j}$ were confined to the interval $[0,1]$. In the current context, the values of $H_{i,j}$ are not subject to such constraints and can vary arbitrarily in magnitude. 
\end{enumerate}

\begin{figure}
    \centering
    \includegraphics[width = \textwidth]{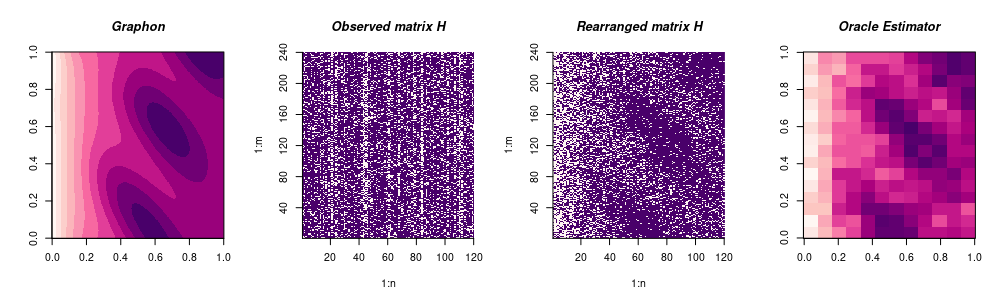}
    \vglue-10pt
    \caption{An illustration of the graphon 
    problem. The leftmost graph represents 
    the unknown graphon $W^*$. The second 
    leftmost graph is the adjacency matrix 
    observed in the graph where the links 
    are made according to the Bernoulli model. 
    The third graph is the adjacency matrix 
    that would be obtained after a rearrangement 
    of the rows and columns if we had access to 
    the latent variables. The rightmost graph represents the histogram estimator obtained 
    from the rearranged adjacency matrix. Our
    goal is to design an estimator which is 
    nearly as good as the oracle, without having
    access to the latent variables.
    }
    \label{fig:illust}
\end{figure}

Our goal is to study the minimax risk of estimating 
$W^*$ and to highlight its dependence on the 
important parameters of the problem. The sizes  
$n$ and $m$ of the matrix are among these parameters, 
but we will also be interested in the dependence on 
the smoothness of $W^*$, on the ``sparsity'' of 
interactions  (denoted by $\rho$) and on the noise 
level (denoted by $\sigma$). These parameters $\rho$ 
and $\sigma$ are positive real numbers such that
\begin{align}
    \|W^*\|_\infty = \sup_{u,v\in[0,1]} 
    |W^*(u,v)|\leqslant \rho \quad 
    \text{and}\quad 
    \var[H_{i,j}|U_i,V_j] \leqslant \sigma^2  
    \ \text{a.s.},
    \quad \forall i\in[n], \forall j\in[m].
\end{align}
The parameters $\rho$ and $\sigma$ may depend 
on $n$ and $m$, but we choose to write $\rho$
instead of $\rho_{n,m}$ and $\sigma_{n,m}$ 
for the sake of simplicity. 

During the process of estimating the graphon $W^*$, 
an important intermediate step involves the
estimation of the matrix $\bTheta^* = W^*(U_i,V_j)$. 
This matrix estimation holds significant value in 
its own right. We perform this task by solving the
least squares problem over the set of 
constant-by-block matrices, with blocks generated
by partitions of the sets of rows and the columns
of the matrix $\bfH$. It will be further shown that
the method of aggregation by exponential weights 
can be used to ensure adaptivity to the number of
blocks. Under the condition that the graphon is 
piecewise constant or $\alpha$-smooth in the sense
of H\"older smoothness, we establish risk bounds 
for the graphon estimator derived from the estimator 
of $\bTheta^*$. These risk bounds are nonasymptotic, 
and shown to be rate optimal in the minimax sense
for a broad range of regimes. 

\subsection{Measuring the quality of an estimator} 
\Cref{fig:illust} provides an illustration of the
graphon estimation problem, in the case where
$\bfH$ is the adjacency matrix of the bipartite
graph, that is the entries of $\bfH$ are either 0
or 1. We see in this figure that the absence of 
knowledge of the latent variables has a strong 
impact on the recovery of the graphon. Indeed, the
adjacency matrix $\bfH$ depicted in the second
leftmost plot carries little information on $W^*$,
as compared with the rearranged adjacency matrix
displayed in the third plot. In fact, when $\bU$ 
and $\bV$ are unknown the graphon $W^*$ is 
unidentifiable. Let us say that two graphons $W$ 
and $W'$ are equivalent, if there exist two 
bijections that preserve the Lebesgue measure 
$\tau_1:[0,1]\to[0,1]$ and $\tau_2:[0,1]\to[0,1]$ 
such that\footnote{We use notation $\tau_1\otimes 
\tau_2$ for the function from 
$[0,1]^2$ to $[0,1]^2$ defined by $(\tau_1 \otimes 
\tau_2)(u,v) = (\tau_1(u),\tau_2(v))$.}  
$W = W'\circ(\tau_1\otimes\tau_2)$. One can check
that two matrices $\bfH$ generated by equivalent
graphons $W^*$ and $\bar W^*$ have the same 
distribution. This implies that one can at best 
estimate the equivalence class containing $W^*$. 
This is the reason underlying the (pseudo)-distance
we use in this work for measuring the quality of
an estimator $\hat W$ of $W^*$, namely
\begin{align}
    \delta(\hat W , W^*) &=\inf_{\tau_1, \tau_2 
    \in \calM} \bigg(\iint_{[0,1]^2} \big| 
    \hat W (\tau_1 (u), \tau_2 (v)) - 
    W^*(u,v) \big|^2 du\,dv\bigg)^{\nicefrac12} \\
    &= \inf_ {\tau_1, \tau_2 \in \calM} \| \hat W 
    \circ (\tau_1\otimes \tau_2) - W^* \|_{\mathbb L^2}
\end{align}
where $\calM$ is the set of all automorphisms 
$\tau : [0,1] \to [0,1]$ such that $\tau$ and 
$\tau^{-1}$ are measurable,  and  $\tau$ 
preserves the Lebesgue measure in the sense 
that $\lambda(\tau^{-1}(B)) = \lambda (B)$ for 
every Borel-set $B \subset [0,1]$. Two graphons 
$W_1$ and $W_2$ are called weakly isomorphic if 
$\delta(W_1,W_2)=0$.

\subsection{Our contributions} 
The main contributions of the present paper are the following:
\vspace{-8pt}
\begin{itemize}\itemsep=0pt
    \item We present a nonparametric framework
    based on bivariate graphon functions and unobservable
    latent variables that offer a flexible way of
    modeling random matrices and, in particular, adjacency matrices of random bipartite graphs.
    
    \item We establish finite sample risk bounds for
    the estimator minimizing the squared error over
    piecewise constant matrices with a given number
    of clusters, as well as for the exponentially
    weighted aggregate that combines the mentioned
    least-squares estimators. These results apply
    to adjacency matrices of bipartite graphs, whose
    entries are drawn from the Bernoulli distribution,
    but they are also valid for the binomial distribution,
    the scaled Poisson distribution and sub-Gaussian
    distributions. 

    \item We present an adaptation of Lloyd's
    algorithm of alternating minimization (including 
    a step of convex relaxation) to our setting, 
    which allows us to obtain a computationally
    tractable approximation of the least squares
    estimator. 
    
    \item In the case of matrices with entries whose
    conditional distribution given latent variables 
    is the Bernoulli one, we prove lower bounds on the 
    worst-case risk, over the set of piecewise constant 
    graphons, for any graphon estimator. These lower 
    bounds,
    in the vast majority of cases, are of the same
    order as the upper bounds obtained for the 
    least-squares estimator.     
\end{itemize}

\subsection{Prior work}
Statistical analysis of matrix and network data, 
based on block models similar to those considered 
in this work, is an active area of research since 
at least two decades. The stochastic block model, 
introduced by \cite{HOLLAND1983109}, is perhaps 
one of the most studied latent structures for
network data, see also \citep{Nowicki2001} and 
\citep{GOVAERT2003} for early references.  
Community detection, which is the problem of 
detecting the underlying block structure, has 
been the focus of much research effort, as evidenced 
by studies such as \citep{Zhao2012, chin2015,
Lei2015, Lei2016, Zhang2016, Wang2017, 
Chen2018,Xu2020} and literature reviews such 
as \citep{Chao2017,Abbe2018}. It should 
be noted that the majority of these studies focused 
on unipartite graphs with binary or discrete 
edge-labels, and their optimality was mostly 
related to identifying the smallest separation 
rate between the parameters of the communities 
that enables their consistent recovery. 
Similar problems for bipartite block models have been
investigated in \citep{FeldmanPV15,FlorescuP16,
Neumann18,ZhouA19,ZhouA20,Cai2021,NdaoudST22}.

In contrast with the aforementioned papers, the 
focus here is on optimality in terms of the 
estimation error for a model that encompasses 
bipartite graphs with real-valued edge labels. 
Consistency of graphon estimators has been
studied in \citep{AiroldiCC13,wolfe2013nonparametric,
olhede2014network}. In the same problem, 
minimax-rate-optimality of the least squares estimator
has been established in \citep{Gao15,GaoLMZ16,Klopp2017,
klopp2019optimal}, see also the survey 
article \citep{Gao21}. To better present our 
contributions within the current state-of-the-art, 
it would be beneficial to provide a brief overview 
of the contents of these papers. \cite{Gao15} 
considered the case of binary observations $H_{i,j}$,
focusing on the dense case $\rho\asymp 1$, and obtained
minimax rates of estimation over the classes
of piecewise constant and H\"older continuous 
graphons. \cite{GaoLMZ16} extended these results
to matrices $\bfH$ with sub-Gaussian entries, some 
of which might be missing completely at random. 
However, as discussed in \Cref{ssec:3.1}, their 
results are sub-optimal in some cases w.r.t. the 
noise variance, for instance when edge-labels are 
drawn from the binomial distribution. In the case 
of a unipartite graph with binary edge-labels, 
\citep{Klopp2017} established the minimax-optimal 
rates of estimation for sparsely connected graphs, 
where $\rho\ll 1$. While \citep{Gao15,GaoLMZ16} 
measured the estimation error using the normalized 
Frobenius norm of the difference between the estimated 
matrix and the true one, \citep{Klopp2017} additionally 
considered the $\mathbb L_2$-distance between the 
equivalence classes of graphons. In 
\citep{klopp2019optimal}, minimax optimal rates of 
graphon estimation in the cut distance have been 
established for unipartite graphs with binary 
observations. 

The problems studied in this work have connections 
with some recent work in econometrics. Indeed, as 
a consequence of \citep[Theorem 1.4]{AHR}, if the
matrix $\bfH$ is row and column exchangeable, then
there is a function $g^*:[0,1]^4\to \mathbb R$ and 
independent random variables $\alpha$, $\{U_i\}$, 
$\{V_j\}$, $\{\xi_{ij}\}$ uniformly distributed in
$[0,1]$ such that the random matrices $\bfH$ and 
$\big(g^*(\alpha, U_i,V_j,\xi_{ij});i\in[n],j\in[m]
\big)$ have the same distribution. The problem under
consideration in this paper is equivalent to 
estimating the random function $W^*(u,v) = \int_0^1 
g^* (\alpha,u,v,z)\,dz$ from the observations 
$g^*(\alpha, U_i,V_j,\xi_{ij});i\in[n],j\in[m]$, without
assuming any parametric form of the function $g^*$. 
If in addition to $H_{i,j}$, we are
also given a feature vector $\bs X_{i,j}$, for every pair
of nodes $(i,j)$ then the extended model defined
by $g^*(\bs X_{i,j},\alpha,U_i,V_j,\xi_{i,j})$ can be
considered. 
This approach is adopted, for instance,  in 
\citep{Graham17}, where the specific parametric form 
$g^*(\bs x,\alpha,u,v,z) = \mathds 1\big(\bs x^\top 
\bs\beta + u + v + \log(z/(1-z))\big)$ is considered. 
In such a parametric context, the parameter of interest
is the vector $\bs\beta$. In \citep{Graham20}, it is 
assumed that the regression function $\int_{[0,1]^3} 
g^*(\bx,\alpha, u,v,z)\,du\,dv\,dz$ has a parametric 
form $\exp\{a + 
\bx^\top\bs\beta\}/(1 + \exp\{a + \bx^\top\bs\beta\})$
and the problem of estimating the vector $(a,\bs\beta)$
is studied. Asymptotic results (law of large numbers and
central limit theorem) for exchangeable arrays have been
proved in \cite{Davezies}.

\subsection{Notation} For an integer $n\geqslant 
1$, we set $[n] = \{1,\ldots,n\}$. In mathematical
formulae, we use bold capitals for matrices and 
bold italic letters for vectors. The integer 
part of a real number $x$ is denoted by $\lfloor 
x\rfloor$, whereas the minimum and the maximum of 
two real number $x,y$ are denoted by $x\wedge y$ 
and $x\vee y$, respectively. For two $n\times m$ 
matrices $\bfB$ and $\bar\bfB$, the inner product is 
defined as
\begin{align}
    \langle \bfB, \bar\bfB \rangle= \tr(\bfB
    \bar\bfB^\top) 
    = \sum_{i=1}^n\sum_{j=1}^m B_{ij} \bar B_{ij} ,
\end{align}
and we denote by $\|\bfB\|_{\sf F}= \sqrt{\langle
\bfB, \bfB \rangle} $ the Frobenius norm of the 
matrix $\bfB$. The sup-norm of $\bfB$ denoted by 
$\|\bfB\|_\infty$ is defined as the largest in 
absolute value entry of $\bfB$. We write $\bB_{i, 
\bullet}$ and $\bB_{\bullet,j}$ for the $i$th 
row and the $j$th column of $\bfB$, respectively. 
The length of an interval $I\subset\mathbb R$ is 
denoted by $|I|$. For $n\in\mathbb N$, $\mathbf 1_n$
is the $n$-vector with all its entries equal to one.

\section{Estimators of the mean matrix and the 
graphon} \label{sec:2}

In this section, we define the estimators of the
mean matrix $\bTheta^* = \bfE[\bfH|\bU,\bV]$ 
and of the graphon $W^*$ that are investigated
in this paper. We focus here on mathematical 
definitions only; computational and algorithmic
properties of these estimators and their 
tractable approximations are deferred to 
\Cref{sec:4}. 

\subsection[Least squares estimator of the mean 
matrix]{Least squares estimator of $\bTheta^*$}

Let us start by introducing some notation. For 
positive integers $n_0,n,K$ satisfying $Kn_0 
\leqslant n$ and $K \geqslant 2$, we define the set  
\begin{align}\label{cluster_matrix}
    \ZZ({n,K,n_0})= \Big\{ \bfZ \in \{0,1\}^{n 
    \times K}: \bfZ\mathbf 1_K = 
    \mathbf 1_n \text{ and } \min_{k\in[K]} \mathbf 
    1_n^\top\bZ_{\bullet,k} \geqslant n_0   \Big\}.
\end{align}
The elements of this set can be seen as assignment 
matrices: each one of the $n$ users is assigned to 
one (and only one) of the $K$ ``communities'', and 
we have the condition that each community has at 
least $n_0$ ``members''. Similarly, we will repeatedly 
use the set $\ZZ({m,L,m_0})$ of the assignment 
matrices corresponding to the items. Since the $n$ 
rows of $\bfH$ correspond to the users and the $m$ 
columns of $\bfH$ correspond to items, the elements 
of $\ZZ({n,K,n_0})$ will be denoted by $\bfZleft$ 
whereas the elements of $\ZZ({m,L,m_0})$ will be 
denoted by $\bfZright$. Matrices $\bfZleft$ and 
$\bfZright$ correspond to a biclustering: the 
clusters of users are specified by the matrix 
$\bfZleft$; in the same way, $\bfZright$ encodes the 
clusters of items.

Given the observed adjacency matrix $\bfH$, the least 
squares estimator is defined by
\begin{align}\label{LSE1}
    (\widehat \bfQ, \hatbfZleft, \hatbfZright)^{\sf LS} 
    \in \arg \underset{\begin{subarray}{c} \bfQ \in 
    \mathbb R^{K \times L} \\ \bfZleft \in \ZZ(n,K,n_0)\\
    \bfZright \in \ZZ(m,L,m_0)  \end{subarray}}{ \min}
    \big\| \bfH - \bfZleft\bfQ(\bfZright)^\top \big\|_{
    \sf F}^2.
\end{align}
Here, $\bfZleft\bfQ(\bfZright)^\top$ is a $n \times m$
constant-by-block matrix. The idea is thus to find the
constant-by-block matrix that is the closest to $\bfH$ 
in the metric induced by the Frobenius norm, where the 
blocs are given by the matrices $\bfZleft$ and $\bfZright$,
and the sizes of blocks are at least $n_0 \times m_0$. 

These estimators computed by \eqref{LSE1} lead to the
constant-by-block least squares estimator of $\bTheta^*$
defined by $\bThetaLS = \hatbfZleft \widehat\bfQ 
(\hatbfZright)^\top$. One can write $\bThetaLS$ 
in the following alternative way. Let us consider the 
class of constant-by-block matrices 
\begin{align}\label{TT}
    \calT = \calT_{n_0,m_0}^{K,L} = 
    \left\{ \bTheta =  \bfZleft\bfQ(\bfZright)^\top 
    \!\!\!:\ (\bfQ, 
    \bfZleft, \bfZright) \in \mathbb R^{K \times L} 
    \!\times\! \ZZ_{n,K,n_0}\! \times\! \ZZ_{m,L,m_0}
    \!\right\}.
\end{align}
The least squares estimator $\bThetaLS$ is a solution to
\begin{align}\label{hat-Theta}
    \bThetaLS = \hat\bTheta{}^{\sf LS}_{n_0,m_0}[K,L]
    \in \mathop{\rm arg}\min_{\bTheta\in \calT_{n_0,m_0}^{K,L}} 
    \|\mathbf H - \bTheta\|_{\textsf{F}}^2.
\end{align}
Our first results, reported in the next section, provide 
non asymptotic upper bounds on the risk of the estimator 
$\bThetaLS$.

\subsection{Aggregation by exponential weights}

The least squares estimator $\bThetaLS$ is constant
on $KL$ blocks. The number of these blocks, chosen
beforehand, is a hyperparameter of the method. If the
true matrix $\bTheta^*$ is far from being 
blockwise constant on $KL$ blocks, then the quality
of estimation by $\bThetaLS_{n_0,m_0}[K,L]$ might
be poor because of the presence of a large bias. 
One can reduce this bias by computing the least squares
estimator for several values of  $K,L$ (but also $n_0$ and $m_0$) and then by aggregating these estimators.  

To this end, we consider here the extended framework in
which two independent copies $\bfH$ and $\bfH'$ are
observed, both satisfying Assumption A 1.2 (with exactly
the same $\bU$ and $\bV$). The matrix $\bfH$ is used to
construct estimators, whereas $\bfH'$ is used to define 
``weights'' which are used for computing the exponentially
weighted aggregate (EWA). More precisely, we denote
by $\hat\bTheta{}_1^{\sf LS}, \ldots, \hat\bTheta{}_s^{
\sf LS}$ least -squares estimators computed by solving
\eqref{hat-Theta} for $s$ different values of 
$(K,L,n_0,m_0)$. We define
\begin{align}\label{def:EWA}
    \bThetaEWA  = \sum_{\ell = 1}^s w_\ell 
    \hat\bTheta{}^{\sf LS}_{\ell},\quad\text{with}
    \quad    w_\ell = \frac{\exp\{\|\bfH' - 
    \hat\bTheta{}^{\sf LS}_{\ell}\|_{\sf F}^2/
    \beta\}}{\sum_{r=1}^s\exp\{\|
    \bfH'-\hat\bTheta{}^{\sf LS}_r\|_{\sf F}^2/\beta\}}\ 
\end{align}
where $\beta>0$ is a parameter often referred to as 
the temperature. Since $\bfH'$ and $\hat\bTheta{
}^{\sf LS}_{\ell}$ have been computed on two independent
data matrices, they are independent; this will play an
important role in the proofs. The choice of $\beta$ 
depends on the nature of the observations and, more
precisely, on the distribution of the noise $\bxi$, see 
the next section for more details. 

\subsection{Adaptations in the case of missing values}

The estimators of $\bTheta^*$ presented in previous 
paragraphs use all the entries of the matrix $\bfH$. 
However, these estimators, as well as the mathematical
results stated in the next sections, are easy to adapt to
the case of missing observations. More precisely, assume
that we observe some iid random variables $M_{i,j}$ taking
values $0$ and $1$ such that the value $H_{i,j}$ is 
revealed to the statistician if and only if $M_{i,j} = 1$. 
Denoting by $\mathsf p=\mathbf P(M_{i,j}=1)$ and assuming that
$M_{i,j}$ is independent of $(H_{i,j},U_i,V_j)$ (this 
case is commonly referred to as missing completely at 
random), we can define the adjusted observation matrix 
$\tilde \bfH$ by its entries $\tilde H_{i,j} = H_{i,j} 
M_{i,j}/\mathsf p$, for $i\in[n]$ and $j\in[m]$. 

Then, to define the least-squares estimator of $\bTheta^*$, it 
suffices to replace $\bfH$ with $\tilde\bfH$ in 
\eqref{LSE1}. Similar modifications can be made for
defining the exponentially weighted aggregate. Note that
this strategy has been already successfully applied 
in \citep{GaoLMZ16}. The entries of matrix $\tilde\bfH$ 
are all observable, the conditional on $(\bU,\bV)$ 
expectation of $\tilde\bfH$ is still $\bTheta^*$, and
$\var[\tilde H_{i,j}|U_i,V_j] \leqslant \big(\sigma^2 + 
\rho^2(1 - \mathsf p)\big)/\mathsf p$, where $\sigma^2$ 
is an upper bound on the conditional variance 
$\var[ H_{i,j}|U_i,V_j]$ and $\rho^2$ is an upper bound 
on $(\Theta_{i,j}^*)^2$. 

\subsection{Estimating the graphon}

Having estimated the matrix $\bTheta^*$, the focus shifts 
to developing an estimator for the graphon $W^*$. 
To this end, for any $n\times m$ matrix $\bTheta$, we
define its associated graphon $W_{\bTheta}:[0,1]^2\to [0,1] 
$ as a constant function on each rectangle $I_i\times J_j =
[(i-1)/n,i/n[ \times [(j-1)/m,j/m[$, for $(i,j)\in[n]\times
[m]$, given by 
\begin{align}\label{est:graphon}
    W_{\bTheta}(u,v) = \Theta_{i,j},\quad \text{for all}
    \quad (u,v)\in I_i \times J_j.
\end{align}
The rationale behind this definition is the following: 
when $n$ and $m$ are large, the order statistics 
$U_{(i)}$ and $V_{(j)}$ lie with high probability 
in the intervals $I_i$ and $J_j$. Therefore, the matrix
$\tilde\bTheta$ defined by $\tilde\Theta_{i,j} = W_{
\bTheta}(U_i,V_j)$ coincides, up to a permutation of
rows and a permutation of columns, with $\bTheta^*$. 
This means that the matrices generated by $W^*$ and
$W_{\bTheta^*}$ are equivalent. Hence, one can expect
that these two graphons are close. 

In addition, for an estimated graphon defined by
\eqref{est:graphon}, the estimation error can be
easily related to the error, measured by the Frobenius
norm, in estimating matrix $\bTheta^*$.  Indeed,
one easily checks that $\|W_{\hat\bTheta} - W_{\bTheta^*} 
\|_{\mathbb L_2} = \|\hat\bTheta-\bTheta^*\|_{\sf F}/ 
\sqrt{nm}$, which leads to 
\begin{align}\label{eq5:0}
    \delta(W_{\hat\bTheta},W^*) \leqslant  
    \|W_{\hat\bTheta} - W_{\bTheta^*}\|_{\mathbb L_2} +
    \delta(W_{\bTheta^*},W^*)
    \leqslant \frac{\|\hat\bTheta - \bTheta^*\|_{\sf F}}{
    \sqrt{n m}} + \delta(W_{\bTheta^*},W^*). 
\end{align}
We will use this inequality both for $\hat\bTheta = 
\bThetaLS$ and $\hat\bTheta = \bThetaEWA$. To ease 
notation, we often write $\hat W^{\sf LS}$ and $\hat 
W^{\sf EWA}$ instead of $W_{\bThetaLS}$ and 
$W_{\bThetaEWA}$, respectively. The decomposition
provided by \eqref{eq5:0} splits the graphon estimation
error into two components: the error of estimating the 
conditional mean matrix $\bTheta^*$ and the bias of 
approximating $W^*$ by the piecewise constant function
$W_{\bTheta^*}$. The former is the only term that depends
on the estimation routine and on the probabilistic 
assumptions on the noise; it will be analyzed in the next
section under various such assumptions. The latter depends only
on the ``smoothness properties'' of the graphon. The next
result allows us to evaluate this term.

\begin{prop}\label{prop:2}
Let $\Theta_{ij}^* = W^*(U_i,V_j)$ for $i\in[n]$ and
$j\in[m]$, where $W^*:[0,1]^2\to [A,B]$ for some 
$A,B$ such that $-\infty\leqslant A<B\leqslant +\infty$. 
\vspace{-8pt}
\begin{enumerate}\itemsep = 0pt
    \item (Piecewise constant graphon) More precisely, for
    $0=a_0<\ldots<a_K = 1$ and $0 = b_0<\ldots <b_L = 1$, 
    the function $W^*$ is constant on each rectangle $[a_k, 
    a_{k+1}[\times [b_\ell,b_{\ell+1}[$. If we define 
    $W_{\bTheta^*} : [0,1]^2\to [A,B]$ by $W_{\bTheta^*}(u,v) 
    = \Theta^*_{i,j}$ for all $(u,v)\in [(i-1)/n,i/n[ \times 
    [(j-1)/n,j/n[$, then 
    \begin{align}
        \Ex[\delta(W_{\bTheta^*}, W^*)^2]^{\nicefrac12} 
        \leqslant \frac{(B-A)}{\sqrt{2}} \bigg(\sqrt{
        \frac{K}{n}} + \sqrt{\frac{L}{m}}\bigg)^{\nicefrac12}.
    \end{align}
    \item (H\"older continuous graphon)
    If the graphon $W^*$ is $\alpha$-smooth, that is $W^*$ 
    is in the H\"older class \footnote{${\mathbb H}_{\alpha, 
    \calL}$ is the set of functions $W:[0,1]^2\to\mathbb R$ 
    satisfying $|W(x,y) - W(x',y')| \leqslant \calL 
    \big( (x-x')^{2} + (y-y')^{2} \big)^{\alpha/2}$ for 
    all $x,y,x',y' \in [0,1]$.} 
    ${\mathbb H}_{\alpha,\calL}$, 
    for some $\alpha\in(0,1]$ and $\calL >0$, then
    \begin{align}
        \Ex[ \delta(W_{\bTheta^*} , W^*)^2]^{\nicefrac12} 
        \leqslant \frac{2 \calL}{n^{\alpha/2}} + \frac{2 \calL}{
        m^{\alpha/2}}.
    \end{align}
\end{enumerate} 
\end{prop}
The proof of this result, postponed to \Cref{proof:prop2}, 
follows essentially the steps of \citep[Proposition 3.2]{Klopp2017} 
that deals with symmetric functions only. 
As a minor remark, the constants in our results are 
smaller than those available in the literature.

\section{Finite sample risk bounds}

We have introduced in \Cref{sec:2} the least-squares 
estimator and the exponentially weighted aggregate for
estimating the matrix $\bTheta^*$, as well as their 
associated graphon estimators. We provide in this section
upper bounds for the risks of these estimators. The main 
purpose of these bounds is to highlight the behavior 
of the estimators when $n,m$ are large and $\rho,\sigma$ 
are small ($\sigma$ denoting the noise magnitude).

\subsection[Risk bounds for the least-squares estimator
]{Risk bounds for the least-squares estimator $\bThetaLS$}
\label{ssec:3.1}

We start by stating risk bounds for estimators of 
$\bTheta^*$. To this end, without loss of generality, 
both in the statements and in the proofs, we treat 
$\bTheta^*$ as a deterministic matrix; this is why 
we require $H_{i,j}$ to be independent instead or 
requiring conditional independence given $\bU,\bV$. 

\begin{theorem}\label{prob_matrix}
Let $n, m, n_0, m_0, L, K $ be positive integers
such that $L\geqslant 2$, $K\geqslant 2$, $ 3 
\leqslant n_0\leqslant n$  and $3\leqslant m_0 
\leqslant m$. Let $\mathbf H$ be an $n \times m$
random matrix with independent entries satisfying
$\Ex[H_{ij}] \in[0, \rho]$ for every $i \in [n]$, 
$j \in [m]$ and some $\rho >0$. In addition, assume 
that the random variables $(H_{ij} - \Ex[H_{ij}])$ 
satisfy the $(\sigma^2, b)$-Bernstein condition. 
Then, the least squares estimator $\bThetaLS$ 
of the mean matrix $\bTheta^* = \Ex[\bfH]$, defined 
by \eqref{hat-Theta}, satisfies the exact oracle 
inequality
\begin{align}
     \frac{\Ex\big[\|\bThetaLS - \bTheta^*
     \|^2_{\sf F} \big]^{\nicefrac12}}{\sqrt{nm}} 
     & \leqslant  \inf_{\bTheta\in \calT_{n_0,
     m_0}^{K,L}} \frac{\| \bTheta - \bTheta^*  \|_{
     \sf F}}{\sqrt{nm}} + \big(25\sigma^2 + 4b\rho \big)^{\nicefrac12} r_{n,m}(K,L),
\end{align}
with  $r_{n,m}(K,L)$ given by
\begin{align}\label{r_nm}
    r_{n,m}(K,L) =  
    \Big(\frac{3KL}{nm} + \frac{\log K}{m} + 
    \frac{\log L}{n}\Big)^{\nicefrac12}
\end{align}
provided that $\psi_{n,m}(n_0,m_0) := 
\frac{3}{m_0}\,\log (en/n_0) + \frac{3}{n_0}\,
\log(em/m_0)\leqslant (\sigma/b)^2$.
\end{theorem}

Lower bounds on the minimax risk of 
all possible estimators, showing that the risk bound 
in \Cref{prob_matrix} is rate optimal under various regimes, 
will be stated in \Cref{sec:5}. Let us mention here 
the fact that in the particular case $b=0$ corresponding
to a sub-Gaussian distribution, the condition $\bfE[H_{i
,j}]\in[0,\rho]$ can be removed and the claim of the 
theorem remains true.  

\Cref{prob_matrix} being stated for general distributions, 
it is helpful to see its consequences in the cases of 
common distributions of $H_{i,j}$ mentioned in the 
introduction.  

\begin{corollary}
    We assume that the conditions on $n,m,n_0,m_0,L,K$ 
    required in \Cref{prob_matrix} hold. 
    \vspace{-7pt}
    \begin{enumerate}
        \item If $H_{ij}$'s are independent Bernoulli---or 
        any other distribution with support $[0,1]$---
        random variables with mean $\Theta^*_{i,j}
        \leqslant \rho$, then they satisfy the $(\rho, 
        1/3)$-Bernstein condition and, therefore, 
        \begin{align}\label{rb:bern}
             \frac{\Ex\big[\|\bThetaLS - \bTheta^*
             \|^2_{\sf F} \big]^{\nicefrac12}}{\sqrt{nm}} 
             & \leqslant  \inf_{\bTheta\in \calT_{n_0,
             m_0}^{K,L}} \frac{\| \bTheta - \bTheta^*  \|_{
             \sf F}}{\sqrt{nm}} + 9\sqrt{\rho}\,
            \Big(\frac{KL}{nm} + \frac{\log K}{m} + 
            \frac{\log L}{n}\Big)^{\nicefrac12}
        \end{align}
        provided that $ 
        \frac{1}{m_0}\,\log (en/n_0) + \frac{1}{n_0}\,
        \log(em/m_0)\leqslant 3\rho$.
        
        \item If for some $N\in \mathbb N$, $(NH_{ij})$'s 
        are independent binomial random variables with 
        parameters $(N,\Theta^*_{i,j})$ such that $\Theta^*_{i,j}
        \leqslant \rho$, then $H_{i,j}$'s satisfy the $(\rho/N, 
        1/3N)$-Bernstein condition and, therefore, 
        \begin{align}
             \frac{\Ex\big[\|\bThetaLS - \bTheta^*
             \|^2_{\sf F} \big]^{\nicefrac12}}{\sqrt{nm}} 
             & \leqslant  \inf_{\bTheta\in \calT_{n_0,
             m_0}^{K,L}} \frac{\| \bTheta - \bTheta^*  \|_{
             \sf F}}{\sqrt{nm}} + 9\sqrt{\rho}\,
            \Big(\frac{KL}{Nnm} + \frac{\log K}{Nm} + 
            \frac{\log L}{Nn}\Big)^{\nicefrac12}
        \end{align}
        provided that $ 
        \frac{1}{m_0}\,\log (en/n_0) + \frac{1}{n_0}\,
        \log(em/m_0)\leqslant 3N\rho$.
        
        \item If $H_{ij}$'s are independent sub-Gaussian
        random variables with means $\Theta^*_{i,j}$ and 
        variance proxies $\tau^2_{i,j} \leqslant \sigma^2$, 
        then they satisfy the $(\sigma^2, 0)$-Bernstein 
        condition and, therefore, 
        \begin{align}
             \frac{\Ex\big[\|\bThetaLS - \bTheta^*
             \|^2_{\sf F} \big]^{\nicefrac12}}{\sqrt{nm}} 
             & \leqslant  \inf_{\bTheta\in \calT_{n_0,
             m_0}^{K,L}} \frac{\| \bTheta - \bTheta^*  \|_{
             \sf F}}{\sqrt{nm}} + 5\sigma\,
            \Big(\frac{KL}{nm} + \frac{\log K}{m} + 
            \frac{\log L}{n}\Big)^{\nicefrac12}.
        \end{align}
        
        \item If for some $T>0$, $(TH_{ij})$'s are 
        independent Poisson random variables with 
        parameters $T\Theta^*_{i,j} \leqslant T\rho$, then 
        $H_{i,j}$ satisfy the $(\rho/T, 1/3T)$-Bernstein 
        condition and, therefore, 
        \begin{align}
             \frac{\Ex\big[\|\bThetaLS - \bTheta^*
             \|^2_{\sf F} \big]^{\nicefrac12}}{\sqrt{nm}} 
             & \leqslant  \inf_{\bTheta\in \calT_{n_0,
             m_0}^{K,L}} \frac{\| \bTheta - \bTheta^*  \|_{
             \sf F}}{\sqrt{nm}} + 9\sqrt{\rho}\,
            \Big(\frac{KL}{Tnm} + \frac{\log K}{Tm} + 
            \frac{\log L}{Tn}\Big)^{\nicefrac12}
        \end{align}
        provided that $ 
        \frac{1}{m_0}\,\log (en/n_0) + \frac{1}{n_0}\,
        \log(em/m_0)\leqslant 3T\rho$.
    \end{enumerate}
\end{corollary}
The expression of the remainder term appearing in these
risk bounds can be seen as 
\begin{align}
    \textsf{noise magnitude}\times \frac{\textsf{size of
    the parameter space}}{\textsf{sample size}}. 
\end{align}
Indeed, $KL + n\log K + m\log L$ is the order of magnitude
of the logarithm of the covering number of $\calT_{n_0, 
m_0}^{K,L}$, a common measure of the complexity of the 
parameter space. In addition to being instructive, this 
interpretation explains why this upper bound is optimal
up to a multiplicative constant under some mild 
conditions. 

Note also that the least-squares estimator for which the
risk bounds above are established does not require the
knowledge of $\rho$, $\sigma$, and $b$. This explains the
presence of a condition on $\rho$ requiring it to be not
too small. For small values of $\rho$, our proofs may still
be used to get a risk bound for the least-squares estimator. 
For instance, in the Bernoulli model, when $\rho\leqslant 
\frac{1}{m_0}\,\log (en/n_0) + \frac{1}{n_0}\,\log(em/m_0)$,
the remainder term is the same as in \eqref{rb:bern} with
$\rho$ replaced by $\frac{1}{m_0}\,\log (en/n_0) + \frac{1
}{n_0}\,\log(em/m_0)$. However, for such a small value
of $\rho$ smaller risk bounds can be obtained either for 
the estimator that outputs a matrix with all zero entries, 
or for the constrained least-squares estimator with the
constraint $\|\bTheta\|_\infty \leqslant \rho$ (see
\cite{GaoLMZ16,Klopp2017} for results of this flavor). 

\begin{remark}
As mentioned in the introduction, an upper bound 
similar to the one of \Cref{prob_matrix} has been 
established in \citep{GaoLMZ16}, under the condition 
that $H_{i,j} - \Ex[H_{i,j}]$ are $\sigma$-sub-Gaussian. 
The remainder term obtained therein is of the order 
$(\sigma + \rho)r_{n,m}(K,L)$. Since, $\sigma 
$-sub-Guassianness is equivalent to the $(\sigma^2,0)
$-Bernstein condition, our theorem applies to the same
setting and yields a smaller remainder term,  
$\sigma r_{n,m}(K,L)$ (which is independent
of $\rho$). Of particular interest is the case 
where the entries are scaled binomial random 
variables, as in the second claim of the last 
corollary. In this scenario, the risk bound of 
\citep{GaoLMZ16} includes a remainder term of 
order $r_{n,m}(K,L)$ since the sub-Gaussian norm 
of the averages of independent Bernoulli random 
variables is of constant order. Interestingly, 
our upper bound is substantially tighter since 
its remainder term includes a deflation factor of 
$\sqrt{\rho/N}$. In \Cref{sec:5}, we demonstrate 
that this upper bound is tight, at least when $n$ 
and $m$ are of the same order of magnitude.    
\end{remark}

\subsection[Risk bounds for the EWA]{Risk bounds for 
the EWA $\bThetaEWA$}
\label{ssec:3.2}

In the context of signal denoising, pioneering work by 
\cite{LeungBarron} established sharp bounds, initially 
limited to Gaussian noise. Subsequent progress saw 
extensions to encompass a wider array of noise distributions, 
as evidenced by works such as \citep{DT07,DT08,DT12b,
Dal_IHP,Dalalyan2022}. The findings presented in this 
section are derived from 
those expounded in \citep{Dalalyan2022}, which, to the 
best of our knowledge, remain the sole results in the 
literature applicable to models featuring asymmetric 
noise distributions---akin to scenarios found in the 
Bernoulli and binomial models.

\begin{theorem}\label{th:EWA}
    Let $\bfH$ an $n\times m$ matrix with independent 
    entries and let $\bTheta^* = \bfE[\bfH]$. Let 
    $\bfH'$ be an independent copy of $\bH$. Let 
    $\mathfrak P$ be a set of quadruplets $(K,L, 
    n_0,m_0)=:\mathfrak p$ and let $\bThetaLS[
    \mathfrak p]$ be the least squares estimator
    defined by \eqref{hat-Theta}. Let $\bThetaEWA[
    \mathfrak P]$ be the exponentially weighted 
    aggregate \eqref{def:EWA} applied to estimators 
    $\{\bThetaLS[\mathfrak p]: \mathfrak p\in 
    \mathfrak P\}$ with some temperature parameter 
    $\beta>0$. Let $r_{n,m}(K,L)$ be as in \eqref{r_nm}. 
    \begin{enumerate}
        \item (Bernoulli/binomial model) Assume that 
        for some $N\in\mathbb N$, $NH_{i,j} \sim{\sf 
        binomial}(N, \Theta_{i,j}^*)$ with 
        $\Theta^*_{i,j} \leqslant \rho $ for every 
        $i\in[n]$ and $j\in[m]$. For $\beta = 8/(3N)$,
        the estimator $\bThetaEWA = \bThetaEWA[
        \mathfrak P]$ satisfies
        \begin{align}
             \frac{\Ex\big[\|\bThetaEWA - \bTheta^*
             \|^2_{\sf F} \big]^{\nicefrac12}\!\!\!\!}{
             \sqrt{nm}} 
             & \leqslant  \min_{\mathfrak p\in
             \mathfrak P}\Big\{\inf_{\bTheta\in \calT
             [\mathfrak p]} \frac{ \|\bTheta - 
             \bTheta^*\|_{\sf F}\!\!}{\sqrt{nm}} + 
             9\sqrt{\frac{\rho}{N}}\, r_{n,m}(K,L)\Big\} + \Big\{
             \frac{8\log|\mathfrak P|}{3Nnm} \Big\}^{
             \nicefrac12}
        \end{align}
        provided that $ \max_{\mathfrak p\in\mathfrak P}
        \big(\frac{1}{m_0}\,\log (en/n_0) + \frac{1}{n_0}\,
        \log(em/m_0)\big)\leqslant 3N\rho$.
        \item (Gaussian model) Assume that $H_{i,j} \sim
        \calN(\Theta_{i,j}^*,\sigma^2_{i,j})$ with 
        $\sigma^2_{i,j} \leqslant \sigma^2 $ for every 
        $i\in[n]$ and $j\in[m]$. For $\beta = 4\sigma^2$,
        the estimator $\bThetaEWA = \bThetaEWA[\mathfrak 
        P]$ satisfies
        \begin{align}
             \frac{\Ex\big[\|\bThetaEWA - \bTheta^*
             \|^2_{\sf F} \big]^{\nicefrac12}\!\!\!\!}{\sqrt{nm}} 
             & \leqslant  \min_{\mathfrak p\in\mathfrak P}
             \Big\{\inf_{\bTheta\in \calT[\mathfrak p]} \frac{
             \|\bTheta - \bTheta^*\|_{\sf F}\!\!}{\sqrt{nm}} + 5\sigma r_{n,m}(K,L)\Big\} + \Big\{ \frac{\log|\mathfrak P| 
             }{nm} \Big\}^{\nicefrac12}.
        \end{align}
    \end{enumerate}    
\end{theorem}

To prove the first claim of this theorem, it suffices to
combine \citep[Cor.\ 4]{Dalalyan2022} with 
\Cref{prob_matrix}. Similarly, the second claim follows 
from \citep[Cor.\ 2]{Dalalyan2022} and \Cref{prob_matrix}.
Similar results can be obtained for arbitrary distribution
with bounded support and for the Laplace distribution, 
using Corollary 3 and Corollary 5 from \citep{Dalalyan2022}, 
respectively. Unfortunately,  we are currently unaware of 
any result that facilitate the extension of these bounds to 
encompass the Poisson distribution and the broader class of 
sub-Gaussian distributions. 

\begin{remark}\label{rem:1}
The upper bounds obtained in \Cref{th:EWA} show that
the extra error term due to aggregation is not large, 
when the sample size $nm$ is large. Note that the factor $\log 
|\mathfrak P|$ is usually not large. A reasonable 
choice for this set is the following: choose 
geometrically increasing sequences $K_i = \lfloor 
2^{1 + i/2}\rfloor $ and $L_j = \lfloor 2^{1 + j/2}
\rfloor $ for $0\leqslant i\leqslant 2\log_2(n/10)$ 
and  $0\leqslant j\leqslant 2\log_2(m/10)$. Then, for 
each $K_i$ and  $L_j$, choose $n_0\leqslant n/K_i$ 
and $m_0\leqslant m /L_j$ to be of the form $\lfloor 
2^{2+\ell/2}\rfloor$. This method of choosing 
$\mathfrak P$ ensures that $|\mathfrak P| \leqslant 
4 \log_2^2(n/7)\log_2^2(m/7)$. Therefore, the term
$\log |\mathfrak P|$ is, in almost all settings, 
of smaller order than $KL$; indeed, one can check 
that $\log |\mathfrak P|\geqslant 12$ implies that
$\min(n,m)\geqslant 12\times 10^4$.    
\end{remark}

\subsection[Risk bounds for the graphon estimators]{
Risk bounds for the graphon estimators $\hat W^{\sf 
LS}$ and $\hat W^{\sf EWA}$}

%\subsubsection{Piecewise constant graphons}

A suitable combination of the risk bounds established 
in \Cref{prob_matrix} on the error of estimators of the
mean matrix $\bTheta^* = \bfE[\bfH]$ and of inequality
\eqref{eq5:0},  allows us to get risk bounds for 
the graphon estimators $\hat W^{\sf LS}$ and $\hat 
W^{\sf EWA}$. We will focus on two classes, piecewise
constant and H\"older continuous graphons, for which 
the evaluation of the approximation error is provided 
in \Cref{prop:2}. Obviously, using this strategy makes 
the term $\inf_{\bTheta\in \calT} \|\bTheta -
\bTheta^*  \|_{\sf F}$---the oracle error---appear in 
the error bound, where $\calT$ is the set of
constant by-block-matrices defined in \eqref{TT}. 
In the case of the class of piecewise constant graphons, 
this oracle error vanishes, whereas in the case of
H\"older continuous graphons, it needs to be
evaluated, which is done in the next proposition. 
Note that, unlike in \Cref{ssec:3.1} 
and \Cref{ssec:3.2},  we now return to the original 
framework of random $\bTheta^*$ and all the expectations 
comprise integration with respect to the latent 
variables $\bU$ and $\bV$.

\begin{prop}\label{prop_aprox_error}
Let $W^*$ be $\alpha$-H\"older continuous, 
\textit{i.e.}, $W^* \in {\mathbb H}_{\alpha,\calL}$ for 
some $\alpha\in(0,1]$, $\calL >0$. Let $\Theta^*_{i,j} = 
W^*(U_i,V_j)$ for $i\in[n]$ and $j\in[m]$. Let 
$n_0 \geqslant 2$,  $m_0 \geqslant 2$, $K \leqslant 
n/n_0 $ and $L\leqslant  m/m_0 $ be four integers. 
Then, the $n\times m$ matrix $\bTheta^*$ with entries
$\Theta_{i,j}$ satisfies
\begin{align}\label{smooth_graphon_bound}
    \Ex\bigg[ \inf_{\bTheta\in \calT}\frac{\| 
    \bTheta - \bTheta^*  \|_{\sf F}}{\sqrt{nm}} 
    \bigg]^{\nicefrac12} \leqslant \frac{3\calL}{2} 
    \Big(\frac{1}{K^\alpha} + \frac{1}{L^\alpha}\Big),
\end{align}
where $\calT = \calT_{n_0,m_0}^{K,L}$ is the set 
of constant-by-block matrices defined in \eqref{TT}.
\end{prop}

We have now all the necessary ingredients to state the
main results of this paper, quantifying the error of 
estimating the graphon. We do this first for the least 
squares estimator, considering in particular that the 
parameters $K$ and $L$ of the set $\calT$ (on 
which the minimum in \Cref{LSE1} is computed) is fixed. 
We then state the result for the exponentially weighted 
aggregate. 

\begin{theorem}\label{thm:1}
Let $\bfH$ be a $n\times m$ random matrix satisfying 
\Cref{ass:1} with some graphon $W^*:[0,1]^2\to [0,\rho]$. 
Assume that for some constant $\sigma >0$,  
conditionally to $\bU,\bV$, the random variables $(H_{ij}-\Ex[H_{ij}|\bU,\bV])$ satisfy the 
$(\sigma^2, b)$-Bernstein condition.
\begin{enumerate}
    \item Assume that the graphon $W^*$ is $(K,L)
    $-piecewise constant, 
    meaning that for some integers $K,L\geqslant 2$ and 
    for $0=a_0<\ldots<a_K = 1$, $0 = b_1<\ldots <b_L = 1$, 
    such that  
    \begin{align}\label{eq:Delta}
        \Delta^{(K)} := \min_{k\in[K]} |a_k-a_{k-1}|
        \geqslant \frac{8\log(nK)}{n}, \quad 
        \Delta^{(L)} := \min_{\ell\in[L]} |b_\ell - 
        b_{\ell-1}| \geqslant \frac{8\log(mL)}{m},
    \end{align}
    the function $W^*$ is constant on each rectangle $[a_{
    k-1},a_{k}[\times [b_{\ell-1},b_{\ell}[$. Then, the 
    estimator $\WLS = W_{\bThetaLS}$ with $\bThetaLS =
    \hat\bTheta{}_{n_0,m_0}^{\sf LS}[K,L]$ defined by
    \eqref{hat-Theta} satisfies
    \begin{align}
        \Ex[\delta(\WLS, W^*)^2]^{\nicefrac12} 
        &\leqslant \big(27\sigma^2 + 4b\rho \big)^{
        \nicefrac12} 
        \Big(\frac{3KL}{nm} + \frac{\log K}{m} + 
        \frac{\log L}{n} \Big)^{\nicefrac12}\!\!\! +
        \rho \bigg(\sqrt{\frac{K}{n}}
        + \sqrt{\frac{L}{m}}\bigg)^{\nicefrac12}\!\!,
    \end{align}
    provided that $\psi_{n,m}(\Delta^{(K,L)}) = 
    \frac{6\log (2e/\Delta^{(K)})}{m\Delta^{(L)}} + 
    \frac{6\log(2e/\Delta^{(L)})}{n\Delta^{(K)}} 
    \leqslant (\sigma/b)^2$.

    \item Assume that the graphon $W^*$ is 
    $\alpha$-H\"older continuous, meaning that 
    $W^* \in {\mathbb H}_{\alpha, \calL}$ for some 
    $\alpha\in(0,1]$ and $\calL > 0$. Assume 
    that\footnote{The assumption $n\geqslant m$ 
    does not cause any loss of generality, since 
    $n$ and $m$ play symmetric roles in the framework 
    under consideration.} the number of nodes $n,m$ 
    satisfy $n\geqslant m$ and
    \begin{align}\label{th2:1}
        \frac{m^{2\alpha + 1}}{n\log^4 (2n)} 
        \geqslant \calL^2 \frac{(4b/\sigma)^{4 
        (\alpha+1)} \vee 3}{(25\sigma^2 + 4b\rho)}.
    \end{align}
    Let $\beta = \alpha/(2\alpha + 2)$. 
    Then, there is a choice of $K,L,n_0,m_0$ such 
    that the least squares estimator $\WLS = 
    W_{\bThetaLS}$ with $\bThetaLS = \hat\bTheta{}^{
    \sf LS}_{n_0,m_0} [K,L]$ satisfies
    \begin{align}\label{riskW:LS}
     \Ex[\delta(\WLS , W^*)^2]^{\nicefrac12} 
        \leqslant 6\calL^{1 - 2\beta}\bigg( \frac{25 
        \sigma^2 + 4b\rho}{3nm}\bigg)^{\beta} + 
        \Big(\frac{(50\sigma^2 + 8b\rho)\log m}{m} 
        \Big)^{\nicefrac12} + \frac{4\calL}{
        m^{\alpha/2}}.
    \end{align}
\end{enumerate}
\end{theorem}

In order to ease understanding of these results, let 
us make some comments. First, one can note that 
applying the first claim of the theorem to Bernoulli 
random variables\footnote{In fact, exactly the same
result holds if we replace Bernoulli by any 
distribution supported by $[0,1]$.} $H_{i,j}$ 
(the Bernstein condition is then fulfilled with 
$\sigma^2 = \rho$, $b=1/3$), we obtain
\begin{align}\label{upperW:bernoulli}
    \Ex[\delta(\WLS, W^*)^2]^{\nicefrac12} 
    &\leqslant 10 \Big(\frac{\rho KL}{nm} + 
    \frac{\rho\log K}{3m} + \frac{\rho\log L}{3n} 
    \Big)^{\nicefrac12} + \rho
    \bigg(\sqrt{\frac{K}{n}} + \sqrt{\frac{L}{m}}
    \bigg)^{\nicefrac12},
\end{align}
provided that $\frac{2}{3m \Delta^{(L)}} \log (2e/ 
\Delta^{(K)}) + \frac{2}{3n\Delta^{(K)}} \log(2e/ 
\Delta^{(L)})\leqslant \rho $. In the balanced 
setting $n=m$, $K = L$ and $\Delta^{(K)} = 
\Delta^{(L)} = 1/K$, the upper bound in 
\eqref{upperW:bernoulli} simplifies to 
\begin{align}
    10\sqrt{\rho}\,\frac{K}{n} + 9\Big(\frac{\rho
    \log K}{n}\Big)^{\nicefrac12} + \sqrt{2}\rho\Big(
    \frac{K}{n}\Big)^{\nicefrac14},
\end{align}
provided that $\rho\geqslant {2(K/n)\log(2e K)}$. 
The last expression is of the same order as the rate
established in \citep[Corollary 3.3.i)]{Klopp2017} for 
graphons of unipartite graphs. Furthermore, it holds 
under more general conditions on the observations 
and contains explicit values for the constants.

Second, one can have a closer look at the order of 
magnitude of the three terms appearing in \eqref{riskW:LS}
in the case $n=m$ tending to infinity and assuming $\calL, 
\sigma,b$ and $\rho$ to be of order one. Then the first
term is of order $(n^2)^{-\alpha/(2\alpha + 2)}$, which is
known to be the minimax optimal rate of estimating an
$\alpha$-H\"older continuous, $d=2$-variate regression 
function based on $n^2$ observations. The second term, of
order $(\log n/n)^{-1/2}$, is dominated by the first term
when $\alpha<1$, and has the optimal order up to a 
logarithmic factor when $\alpha = 1$. The third term 
being of order $n^{-\alpha/2}$, is of optimal order 
$n^{-1/2}$ when $\alpha=1$, and is the largest term of 
the sum for all $\alpha < 1$. 

To the best of our knowledge, the question of whether there 
are estimators of H\"older-continuous graphons that achieve 
a faster rate of convergence than $n^{-\alpha/2}$ remains 
open. The common belief is that this term is unavoidable 
and it is the price to pay for not observing the covariates
$\bU,\bV$. Note that the deterioration caused by this lack 
of information, measured by the ratio of the third and the
first terms of the risk bound in \eqref{riskW:LS} is of
order $n^{\alpha(1-\alpha)/(2\alpha+2)}\leqslant n^{0.086}$, 
when $n$ and $m$ are of the same order. From a practical
point of view, this deterioration is not significant, since
even for $n=10^{9}$, $n^{0.086}\leqslant 6$.

One can also draw the consequences of the second
claim of the theorem under various (conditional to
$\bU$, $\bV$) distributions of $H_{i,j}$. For 
$\alpha = 1$ ($\beta = 1/4$), that is 
Lipschitz-continuous graphons, conditions 
\eqref{th2:1} and inequality \eqref{riskW:LS} are 
reported in \Cref{tab:1}.

\begin{table}
    \centering
    \begin{tabular}{c|c|c|c}
        \toprule
        Distr.\ of $H_{ij}$ & Values $(\sigma^2,b)$ 
        & Condition \eqref{th2:1} & Risk Bound \eqref{riskW:LS}\\
        \midrule 
        Bernoulli$(\rho)$ & $(\rho,1/3)$  & $\rho^5 \geqslant 
        \dis\frac{\calL^2 n\log^4(2n)}{m^3}$ & 
        $\dis \frac{11 \sqrt{\calL} \rho^{1/4}}{(nm)^{1/4}} + 
                \frac{8\sqrt{\rho\log m}}{\sqrt{m}} + 
                \frac{4\calL}{\sqrt{m}}$\\
        \midrule 
        Binomial$(N,\rho)/N$ & $(\rho/N,1/3N)$  & $\rho^5
        \geqslant 
        \dis\frac{\calL^2 N n\log^4(2n)}{m^3}$ & 
        $\dis \frac{11\sqrt{\calL} \rho^{1/4}}{(Nnm)^{1/4}} + 
                \frac{8\sqrt{\rho\log m}}{\sqrt{Nm}} + 
                \frac{4\calL}{\sqrt{m}}$\\
        \midrule 
        sub-Gauss$(\sigma^2)$ & $(\sigma^2, 0)$  & $\sigma^2 
        \geqslant\dis\frac{3\calL^2 n\log^4(2n)}{25m^3}$ & 
        $\dis \frac{11\sqrt{\calL\sigma}}{(nm)^{1/4}} + 
                \frac{8\sigma\sqrt{\log m}}{\sqrt{m}} + 
                \frac{4\calL}{\sqrt{m}}$\\
        \midrule 
        Poisson$(T\rho)/T$ & $(\rho/T,1/3T)$  &  $\rho^5
        \geqslant \dis\frac{\calL^2 T n\log^4(2n)}{m^3}$ & 
        $\dis \frac{11\sqrt{\calL} \rho^{1/4}}{(Tnm)^{1/4}} + 
                \frac{8\sqrt{\rho\log m}}{\sqrt{Tm}} + 
                \frac{4\calL}{\sqrt{m}}$\\
        \bottomrule
    \end{tabular}
    \caption{Upper bound for Lipschitz-continuous
    graphons and various distributions.}
    \label{tab:1}
\end{table}

To close this section, we state the risk bounds that
can be obtained by combining \Cref{th:EWA} and 
\Cref{thm:1}. To keep the statement simple, only the
case of piecewise constant graphon is presented.

\begin{corollary}
    Let $n\leqslant m$ and let $\bfH$ be a $n\times m$ random matrix satisfying 
    \Cref{ass:1} with some $(K,L)$-piecewise constant
    graphon $W^*:[0,1]^2\to [0,\rho]$. This means that 
    for some integers $K,L\geqslant 2$ and for $0 = a_0 
    < \ldots <a_K = 1$, $0 = b_1<\ldots <b_L = 1$ such
    that  \eqref{eq:Delta} holds, the function $W^*$ is
    constant on each rectangle $[a_{k-1},a_{k}[\times
    [b_{\ell-1},b_{\ell}[$. Let $\mathfrak P$ be chosen 
    as in \Cref{rem:1} and let $\WEWA[\mathfrak P] =
    W_{\bThetaEWA[\mathfrak P]}$.
    \vspace{-8pt}

    \begin{enumerate}\itemsep=0pt
        \item (Bernoulli/binomial model) Assume that 
        for some $N\in\mathbb N$, conditionally to $(\bU, 
        \bV), $$NH_{i,j}$ is drawn from the ${\sf binomial}(N, \Theta_{i, 
        j}^*)$ distribution with $\Theta^*_{i,j} = W^*(U_i,V_j) 
        \leqslant \rho $ for every $i\in[n]$ and $j \in
        [m]$. For $\beta = 8/(3N)$, we have
        \vspace{-6pt}
    \end{enumerate}
    \begin{align}
            \Ex[\delta(\WEWA, W^*)^2]^{\nicefrac12} 
        &\leqslant 9
        \Big(\frac{2\rho KL + \log \log_2 m}{nm} + \frac{\rho\log K}{m} + 
        \frac{\rho\log L}{n} \Big)^{\nicefrac12}\!\!\! +
        {\rho} \bigg(\sqrt{\frac{K}{n}}
        + \sqrt{\frac{L}{m}}\bigg)^{\nicefrac12}
    \end{align}
    \begin{enumerate}
        \item[] provided that $ \frac{\log
        (3e/\Delta^{(K)})}{m\Delta^{(L)}} + \frac{\log(3e/
        \Delta^{(L)})}{n\Delta^{(K)}}\leqslant N\rho$.
    \end{enumerate}

    \begin{enumerate}
        \setcounter{enumi}{1} 
        \item (Gaussian model) Assume that, conditionally 
        to $(\bU,\bV)$, the entries $H_{i,j}$ are drawn from the Gaussian 
        $\calN (W^*(U_i,V_j), \sigma^2 (U_i,V_j))$ distribution 
        with $|W^*(u,v)| \leqslant \rho$ and  $\sigma^2(u,v) 
        \leqslant \sigma^2 $ for every $u,v\in[0,1]$. For 
        $\beta = 4\sigma^2$, 
        we have\vspace{-5pt}
    \end{enumerate}
    \begin{align}
            \Ex[\delta(\WEWA, W^*)^2]^{\nicefrac12} 
            &\leqslant 5\sigma
            \Big(\frac{2KL + \log \log_2 m}{nm} + \frac{\log K}{m} + \frac{\log L}{n}
            \Big)^{\nicefrac12}\!\!\! + {2\rho} \bigg(\sqrt{ 
            \frac{K}{n}} + \sqrt{\frac{L}{m}} \bigg)^{ 
            \nicefrac12}\!\!\!.
        \end{align}
\end{corollary}

\section{Tractable approximation of the least-squares 
estimator}\label{sec:4}

The least squares estimator introduced in \eqref{LSE1}
and studied in previous sections is a solution to
a combinatorial optimization problem that is 
computationally intractable. It is impossible to
compute this estimator in polynomial time.  The goal
of this section is to present a tractable algorithm 
that computes an approximation to $\bThetaLS$. 
Of course, there is no guarantee that the presented
algorithm provides an estimator that is always
close to $\bThetaLS$, but it is plausible that 
this is true in many cases.

The proposed approximation can be seen as a 
version of Lloyd's algorithm for $k$-means 
clustering \citep{Lloyd}. To describe it, let 
us recall  that the least square estimator is 
defined by
\begin{align}\label{LSE2}
    (\widehat \bfQ, \hatbfZleft, \hatbfZright
    )^{\sf LS} \in \arg \min_{\substack{
        \bfQ \in \mathbb R^{K \times L} \\ 
        \bfZleft \in \ZZ(n,K,n_0)\\
        \bfZright \in \ZZ(m,L,m_0)}}
    \big\| \bfH - \bfZleft\bfQ(\bfZright)^\top 
    \big\|_{\sf F}^2.
\end{align}
It turns out that when we fix two of the three 
arguments $\bfQ,\bfZleft,\bfZright$ of $\calL$, 
the minimization problem with respect to the third 
becomes tractable. We can therefore use the 
alternating minimization algorithm below, with 
the guarantee that the cost function $\calL( 
\bfZleft, \bfQ, \bfZright) = \| \bfH - \bfZleft 
\bfQ (\bfZright)^\top \|_{\sf F}^2$ decreases 
at each iteration. Different versions of this
algorithm have been studied in the literature 
on estimation and detection in the presence of 
a latent structure \citep{coja2010graph,Lei2015, 
lu2016statistical, giraud2019partial}. 

\begin{algorithm}
\caption{Lloyd's algorithm of alternating 
minimization for approximating the LSE \eqref{LSE2}} 
\label{alg:lloyd}
\begin{algorithmic}[ht!]
\Require $\bfZleft$, $\bfZright$ the left and 
right cluster matrices with entries in $\{0,1\}$, 
$\bf H$ the data matrix.
\Ensure  $(\bfZleft,\bfQ,\bfZright)$ local 
minimizer of $\calL( \cdot , \cdot ,\cdot )$.
\State Repeat:
\begin{enumerate}
\State Compute $\bfQ = (\bfZleft_{\sf{norm}}
)^\top \bfH \bfZright_{\sf{norm}} $ where 
$\bfZleft_{\sf{norm}}$ is the matrix $\bfZleft$ 
with normalized columns with respect to 
$\ell^1$-norm (the number of 1 in the column), 
and similarly for $\bfZright_{\sf{norm}}$. 

\State Update $\bfZleft$ that minimize $\bfZ 
\mapsto \calL(\bfZ, \bfQ, \bfZright)$ 
%\State (Optional) Update $\bfQ$ as in step 1.
\State Update $\bfZright$ that minimize $\bfZ 
\mapsto \calL(\bf \bfZleft, \bfQ, \bfZ)$ 
% \Comment{Similar to the left procedure}
\end{enumerate}
\end{algorithmic}
\end{algorithm}

The rest of this section provides more details on
each step of this algorithm, as well as on the 
initialization and on the stopping criterion.

\paragraph*{Minimization in the second argument \texorpdfstring{$\bfQ$}{}}

When the clusters are known, meaning
that we know matrices $\bfZleft \in \ZZ(n,K,n_0)$ and
$\bfZright \in \ZZ(m,L,m_0) $, see \eqref{cluster_matrix}, the solution to
\begin{align}
   \widehat \bfQ = \arg \underset{\bfQ \in \mathbb R^{K \times L}}{ \min} \, \calL(\bfZleft, \bfQ, \bfZright) 
\end{align} 
is easy to compute: each entry $\widehat Q_{kl}$ is
the average of the coefficients $H_{ij}$ belonging to
the $(k,l)$-block defined by $\bfZleft$ and $\bfZright$ (a coefficient $(i,j)$ is in the block $(k,l)$ if $\Zl_{ik}=1$ and $\Zr_{jl} = 1$). Formally, this is equivalent to
$\bfQ = (\bfZleft_{\sf{norm}})^\top \bfH \bfZright_{\sf{norm}}$
where $\bfZ_{\sf{norm}} = (\bfZ^\top\bfZ)^{-1}\bfZ$ 
for $\bfZ \in\{\bfZleft, \bfZright\}$.

\paragraph*{Minimization with respect to \texorpdfstring{$\bfZleft$}{}}
We focus now on the problem of minimizing the cost
function $\calL(\bfZleft, \bfQ, \bfZright)$ over
$\bfZleft \in \ZZ(n,K,n_0)$. 
Let us first consider the relatively
simple case $n_0=0$ when there is no constraint 
on the cardinality of left clusters. We aim to find $\hatbfZleft \in \mathbb R^{n \times K}$ that minimizes $\bfZ \mapsto \calL(\bfZ, \bfQ, \bfZright)$ under the constraints that ${\bfZ} \in \{0 ,1 \}^{n \times K}$ and ${\bfZ} {\mathbf 1}_K = {\mathbf 1}_n$
(\textit{i.e.}, each row of $\bfZ$ has only one 
entry equal to $1$). Let us define 
$\Cr_\ell = \{ j \in [m] , \Zr_{i\ell} 
=1 \}$ to be the $\ell$-th right cluster
and introduce notation
\begin{align}
    \quad \bar H^{\rm item}_{
    i\ell} = \frac{1}{|\Cr_\ell|} \sum_{j\in 
    \Cr_\ell} H_{ij}\quad\text{ and }\quad \mathbf D =
    \textsf{diag}(|\Cr_\ell|)_{\ell\in[L]}.
\end{align}
Simple algebra yields
\begin{align}
    \calL(\bfZ, \bfQ, \bfZright) 
    &= \sum_{i=1}^n \| H_{i, \bullet} - \bZ_{i, \bullet}^\top \bfQ (\bfZright)^\top \|_2^2,
\end{align}
where $\bZ_{i, \bullet}$ is $i$th row of $\bfZ$. 
Since $\bZ_{i, \bullet}$ is allowed to have only one nonzero entry, and it should be equal to one, $\bZ_{i, \bullet}^\top \bfQ (\bfZright)^\top$ is merely equal
to one row of $\bfQ (\bfZright)^\top$. 
This implies
that $\hZl_{i \, j} = \mathds 1(j= k_i)$, where
\begin{align}
    \hat k_i &= \arg \min_{k \in [K]} 
    \big\| H_{i, \bullet} - \bfQ_{k,\bullet}
    (\bfZright )^\top \big\|_2^2 
    = \arg \min_{k \in [K]}   
    \sum_{\ell=1}^L \sum_{j \in \Cr_\ell} 
    ( H_{i j} - Q_{kl} )^2 .
\end{align}
We can rewrite the above expression  of $\hat k_i$ as 
\begin{align}
    \hat k_i &= \arg\min_{k \in [K]}
    \, \| \mathbf D^{1/2}(\bar{\bH}_{i, \bullet
    }^{\rm item} - \bQ_{k, \bullet})^\top\|_2^2.
\end{align}
Thus, in order to determine $\bfZleft$, it suffices
to compute the matrix $\bar\bfH^{\rm item}\in\mathbb 
R^{n\times L}$ and then find for each row of 
$\bar\bfH^{\rm item}$ the closest row of $\bfQ$. Of
course, the same procedure is valid for minimizing
$\calL$ with respect to $\bfZright$ for known 
$\bfZleft$ and $\bfQ$.

Let us return to the general case $n_0\geqslant 0$. 
In this case, we show that the minimization of 
$\calL$ with respect to $\bfZleft$ can be done by
solving a linear program.  Indeed, let is define
\begin{align}
    \phi(\bfZ) &= -2\tr(\bfZ \bfQ(\bfZright)^\top
    \bfH^\top) + \sum_{k=1}^K \mathbf 1_n^\top
    \bZ_{\bullet,k} \bQ_{k,\bullet}\mathbf D
    \bQ_{k,\bullet}\\
    \widetilde \ZZ(n,K,n_0) &= \Big\{ \bfZ \in 
    [0,1]^{n \times K}: \bfZ\mathbf 1_K = 
    \mathbf 1_n \text{ and } \min_{k\in[K]} \mathbf 1_n^\top\bZ_{\bullet, k} \geqslant n_0 \Big\}.
\end{align}  
Note that $\phi$ is a linear function of $\bfZleft$, 
whereas $\widetilde \ZZ(n,K,n_0)$ is a convex
polytope containing $ \ZZ(n,K,n_0)$.

\begin{prop}\label{prop:polytope}
The following two claims hold true.
\vspace{-4pt}
\begin{enumerate}\itemsep=0pt
    \item The function $\calL(\bfZ,\bfQ,\bfZright) 
    - \phi(\bfZ) = \|\bfH\|_{\sf F}^2$ is independent 
    of $\bfZ$. 
    \item The set of extreme points of $\widetilde
    \ZZ(n,K,n_0)$ is $\ZZ(n,K,n_0)$. Equivalently, 
    an element of $\tilde\ZZ(n,K,n_0)$ is an extreme
    point if and only if all its entries are either
    0 or 1.
\end{enumerate}
\end{prop}

The set $\widetilde \ZZ(n,K,n_0)$ is a convex polytope 
because it is defined by linear constraints. A well 
known result \citep[p 65, Thm 2.7]{polytope} implies that if 
$\cal L(\cdot,\bfQ,\bfZright)$  has a minimizer in the 
polytope $\widetilde\ZZ(n,K,n_0)$, then it has at least 
one solution in the set of its  extreme points $\ZZ(n,K,n_0)$. 
There are many efficient solvers for finding such a solution.

\paragraph*{Initialization}

The initialization of \Cref{alg:lloyd} might have a strong 
impact on the final result. One possible strategy is to 
run in parallel $N$ instances of the algorithm with 
different initial values, chosen at random. The final 
estimator is the one that minimizes $\calL$ among the 
resulting ($N$) candidates. 

Another strategy, often used in conjunction with Lloyd's 
algorithm, is based on spectral initialization. When 
the graphon is piecewise constant, the problem 
under consideration is nothing else but the 
bi-stochastic block model for bipartite networks. 
Therefore, initial values $(\bfZleft,\bfZright)$ can
be obtained, for instance, by the spectral method from
\citep{biclustering}. It consists in computing the 
$K$-truncated singular value decomposition of a 
regularized version of matrix $\bfH$, and then, in 
applying $k$-means clustering to the $K$-truncated 
left singular vectors to obtain an initialization for $\bfZleft$.  
The procedure is similar for $\bfZright$.

\paragraph*{Stopping rule}

As mentioned, the cost function is non increasing 
over the iterations, and it takes its values in 
a finite set since there is only a finite number 
of configurations for $(\bfZleft,\bfZright)$. This 
is why, from a certain iteration onwards, the 
values of the cost function remain constant. 
However, the algorithm may require a large number 
of iterations to achieve consistency. This suggests to
stop iterating if either two consecutive
values of the cost function are equal or the maximum 
number of iterations is attained. 

To conclude this section, we stress once again that 
there is no guarantee that the computationally 
tractable algorithms we presented here provide the 
global minimum of the cost function $\calL$. However, 
as can be seen from the numerical examples in 
\Cref{sec:exp}, results are quite satisfactory.

\section{Lower bounds on the minimax risk}
\label{sec:5}

We show in this section, that the least squares 
estimator $\WLS$ is optimal, among all possible
estimators, in the sense of its rate of convergence 
in the worst case over the class
\begin{align}
    \calW_\rho[K,L] = \bigg\{W:[0,1]^2\to [0,\rho] : 
    \exists \ \{I_k\}_{k=1}^K, \{J_\ell\}_{\ell=1}^L
    \text{ s.t. } W = \sum_{k=0}^{K-1}
    \sum_{\ell = 0}^{L-1} W(a_k,b_\ell)\mathds 1
    _{I_k\times J_\ell}
    \bigg\},
\end{align}
where $I_k = [a_k,a_{k+1})$ and $J_\ell = [b_\ell, 
b_{\ell+1})$ form a partition of $[0,1)$ into intervals. 
The lower bound will be proven for the binomial model, 
but all the techniques used in the proof can be extended
to the other models presented in the introduction. 

\begin{figure}%[ht]
    \centering
    \includegraphics[width = 0.8\textwidth]{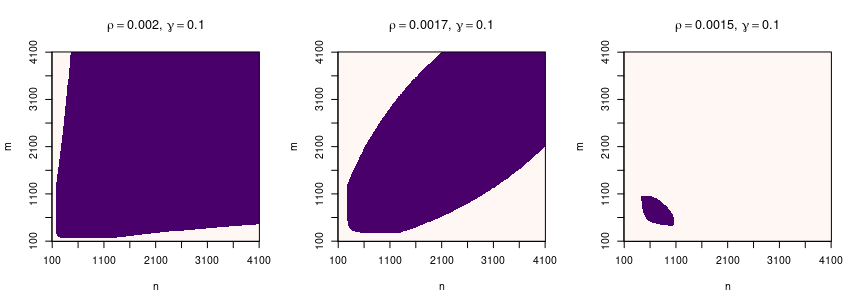}
    \vglue-15pt
    \includegraphics[width = 0.8\textwidth]{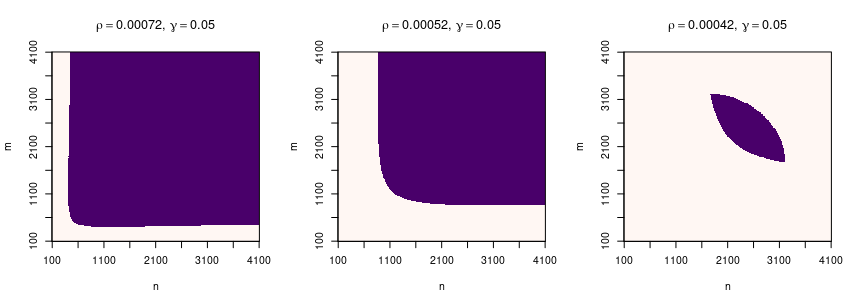}
    \vglue-5pt
    \caption{Illustration of the optimality
    of the least squares estimator for $N=1$. The purple 
    area corresponds to the values of $n$ and $m$,
    for some fixed values of $\rho$ and $\gamma = 
    K/n = L/m$, for which the lower bound is 
    within a constant factor of the upper bound. 
    More precisely, when $\rho\gamma^2\wedge \rho^2 + 
    \rho(nm)^{-1/2}\wedge\rho^2+2\rho^2\sqrt{\gamma}$
    is larger than half of $\rho\gamma^2 + 
    (\rho\log K)/(3m) + (\rho\log L)/(3n) + 
    2\rho^2\sqrt{\gamma}$. We observe that unless 
    $\rho$ is very small, the upper bound established
    for the least-squares estimator is within a 
    constant factor of the lower bound for all
    estimators for most values of $n$ and $m$.     
    }
    \label{fig:lower}
\end{figure}

\begin{theorem}\label{lower_bound}
Assume that conditionally to $(\bU,\bV)$, the 
entries $H_{i,j}$ of the observed $n\times m$ 
matrix $\bfH$ are independent and drawn from 
the Binomial distribution with parameter 
$(N,W^*(U_i,V_j))$. There exist universal constants 
$c,C>0$, such that for any $K,L>C$ satisfying
$KL\geqslant L\log^2 L + K\log^2 K$ and for any
$\rho > 0$,
\begin{align}\label{lowerW:bernoulli}
    \inf_{\widehat W} \sup_{ W^* } 
    \Ex \big[ \delta^2(\widehat W , W^*) 
    \big]^{\nicefrac{1}{2}} \geqslant c  
    \sqrt{\rho} \bigg(\frac{KL}{Nnm}\wedge \rho + \frac{1}{
    N\sqrt{nm}}\wedge\rho \bigg)^{\nicefrac12}\!\!\!\!
    + c\rho \bigg(
    \sqrt{\frac{K}{n}}  +\sqrt{\frac{L}{m}}
    \bigg)^{\nicefrac12}\!\!\!\!,
\end{align}
where the inf is over all possible estimators 
$\hat W$ and the sup is over all $W^*\in 
\calW_\rho[K,L]$.
\end{theorem}

The right-hand side of \eqref{lowerW:bernoulli} 
should be compared to \eqref{upperW:bernoulli}. One 
can observe that if the values of $n$,$m$, $K$, $L$ 
and $\rho$ are such that the dominating term in the
upper bound is one of the terms  $\rho (K/n)^{1/4}$ 
and $\rho (L/m)^{1/4}$, then the lower bound in \eqref{lowerW:bernoulli} is of the 
same order as the upper bound. Therefore, in this
case, the least squares estimator of the graphon is
minimax-rate-optimal. Similarly, if the dominating 
term is $(\rho KL/(Nnm))^{1/2}$ and $\rho\geqslant 
KL/(Nnm)$, then the LSE is minimax-rate-optimal. 
Note also that if $\rho$ is very small, that is 
smaller than both $KL/(Nnm)$ and $(N^2nm)^{-1/2}$, then
the lower bound in \eqref{lowerW:bernoulli} is of
order $\rho$, which might be much smaller than the
upper bound established for the LSE. This is not an
artifact of the proof, but reflects the fact that in
this situation the naive estimator $\hat K\equiv 0$
is better than the LSE. Furthermore, this naive
estimator turns out to be minimax-rate-optimal 
under the mentioned constraint on $\rho$. This
basically means that for such small values of 
$\rho$ the problem of estimating the graphon is not
meaningful from a statistical point of view. 
\Cref{fig:lower} depicts the regions of the values
of $n$ and $m$ where the lower and the upper bounds
are of the same order, illustrating thus the 
optimality of the least-squares estimator. Similarly, 
\Cref{fig:lower2} shows the regions of the values
of $\rho$ and $\gamma=K/n=L/m$, for fixed values of $n$ and
$m$, where the lower bound is larger than half
of the upper bound. We clearly see that even for
very unbalanced graphs ($m$ much larger than $n$), 
the purple region covers almost the whole square, 
which means that the least-squares estimator is 
minimax-rate-optimal in this region.

\begin{figure}%[ht]
    \centering
    \includegraphics[width = 0.95\textwidth]{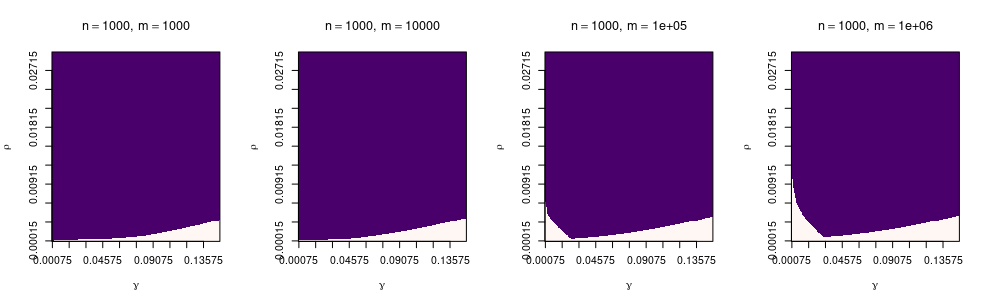}
    \vglue-5pt
    \caption{Illutration of the optimality
    of the least squares estimator. The purple 
    area corresponds to the values of $\rho$ and $\gamma
    = K/n = L/m$, for some fixed values of $n$ and $m$, 
    for which the lower bound is within a constant 
    factor of the upper bound. We observe that unless 
    $\rho$ is very small, the upper bound and the lower 
    bound for are of the same order.     
    }
    \label{fig:lower2}
\end{figure}

\section{Numerical experiments}\label{sec:exp}

In this section, we present the results of some 
numerical experiments illustrating the behavior of 
the error of the estimated graphon $\WLS$ and its
dependence on different parameters of the model. 
We first consider the case of piecewise constant graphons 
and study the estimation error of the matrix $\bTheta^*$. 
We explore the dependence of this error on $n$ for 
different values of $(\rho , K ,L)$ (assuming that 
$m = n/2$) as well as on the sparsity parameter $\rho$ 
for different values of $(n,m,K,L)$. We then show the 
results of  the estimation for a H\"older-continuous 
graphon, when parameters $K$ and $L$ are chosen as functions
of $n$ and $m$ respectively, as recommended by our
theoretical results.

\subsection{Estimation error of the piecewise constant 
matrix \texorpdfstring{$\bTheta^*$}{}}

We report the results of two different experimental
set-ups, referred to as rand-graphon and cos-graphon. 
The two set-ups differ in the choice of the 
graphon only. In both cases, the partitions on which 
the graphon is piecewise constant is the regular 
partition induced by the rectangles of the form 
$[(k-1)/K,k/K)\times [(\ell-1)/L,\ell/L)$  for 
$k\in[K]$ and $\ell\in[L]$. In the rand-graphon 
set-up, the values of $W^*$ are chosen at random
between $0$ and $\rho$, while in the cos-graphon
set-up, $W^*$ is defined as
\begin{align}
    W^*(u,v) = 
        \frac{2\rho}3 + \frac{\rho}3\cos \big( 
        3\pi \lfloor Ku \rfloor \lfloor Lv \rfloor 
        \big),\qquad 
        \forall u,v\in [0,1].
\end{align}
The results obtained in these two set-ups are depicted
in \Cref{fig:error_n} and \Cref{fig:error_n_cos,fig:error_rho}, 
respectively. In each experiment, we chose $m=n/2$ 
and computed the median of the squared error 
$\frac{1}{nm}\| \widehat \bTheta - \bTheta^* 
\|^2_{\sf F}$ for 50 independent repetitions. The 
estimator $\widehat \bTheta$ was computed by 
\Cref{alg:lloyd}. 
    
For better legibility, the errors in the plots are 
presented using a log-scale. To check the consistency of the 
numerical results with our theoretical results, we 
plotted (in green) the remainder term appearing in 
the upper-bound in \cref{prob_matrix}. We also 
displayed the oracle error (red curve) corresponding 
to the error of the best pseudo-estimator that is 
built with the knowledge of the true left and right 
cluster matrices. We only computed the block averages 
in this case. The labels ``spectral'' and ``random'' 
refer to the initialization process used for the 
algorithm. To display the uncertainty, we plotted
colored areas corresponding to the quantiles of 
order 0.1 and 0.9 respectively. (One may be surprised
by the fact that this area grows with $n$ in some cases; 
this is an artifact of the log-scale). In 
\Cref{alg:lloyd}, we chose $\gamma = 10^{-3}$ and 
the maximum number of iterations equal to $40$. 

In these experimental results, an evident trend emerges: 
the error of the ``spectral'' version consistently diminishes 
as $n$ increases. Additionally, as $\rho$ assumes larger 
values and when both $n/K$ and $m/L$ are elevated, the 
error of the estimator is close to the error of the oracle, 
a pattern that aligns with intuitive expectations. A larger 
$\rho$ implies a greater number of links, consequently yielding 
more accurate estimations. Similarly, heightened values of 
$n/K$ and $m/L$ correspond to improved accuracy in estimations.

Conversely, the ``random" version of the algorithm exhibits a 
more erratic behavior. In the majority of instances, its 
error surpasses that of the ``spectral" version, particularly 
when $n/K$ and $m/L$ are large enough.

\begin{figure}%[h]
    \centering
     \begin{subfigure}[b]{0.32\textwidth}
         \includegraphics[width=\textwidth]{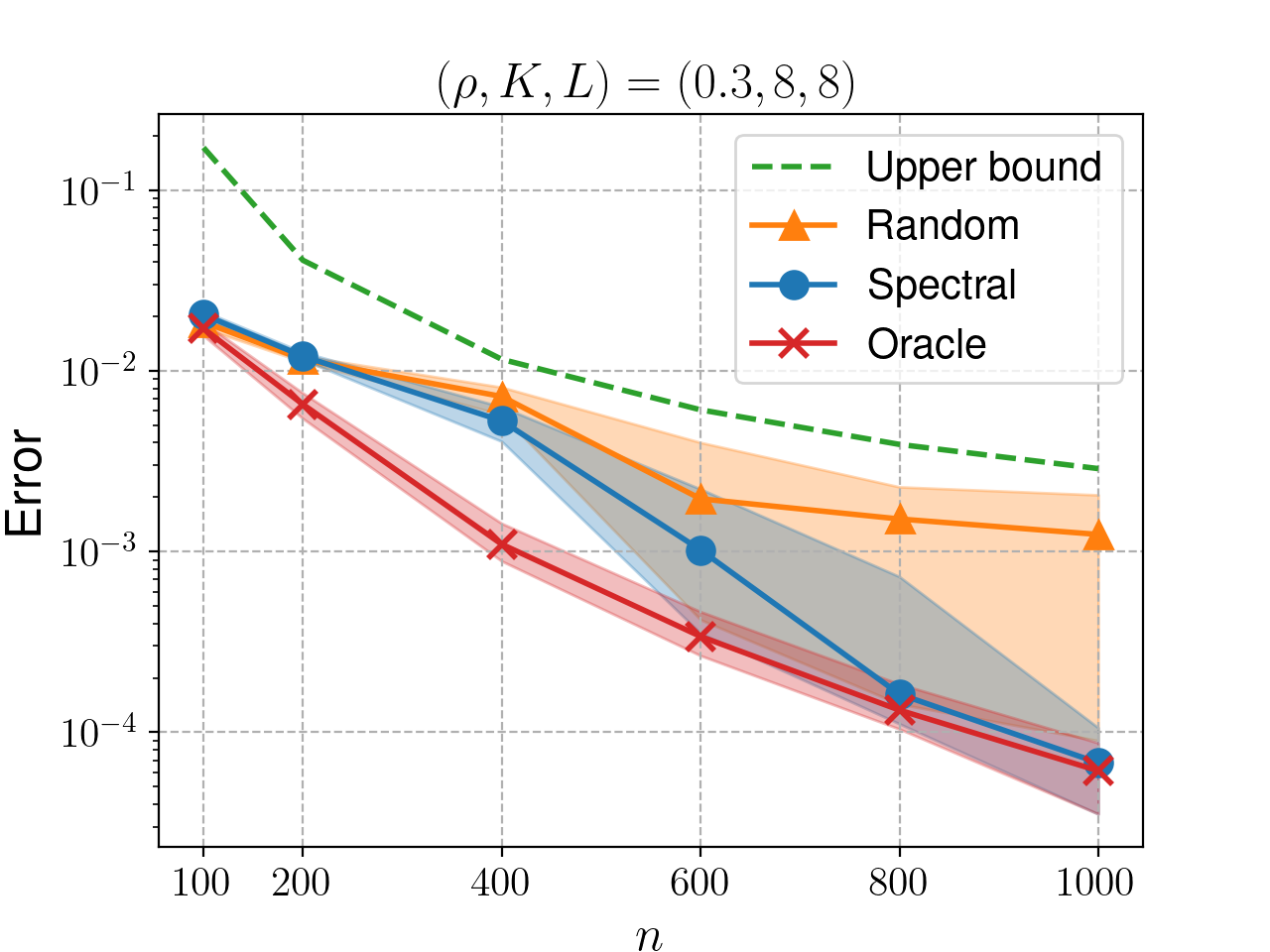}
%         \caption{$y=x$}
         \label{fig:(0.3 , 8 , 8)}
     \end{subfigure}
     \begin{subfigure}[b]{0.32\textwidth}
         \includegraphics[width=\textwidth]{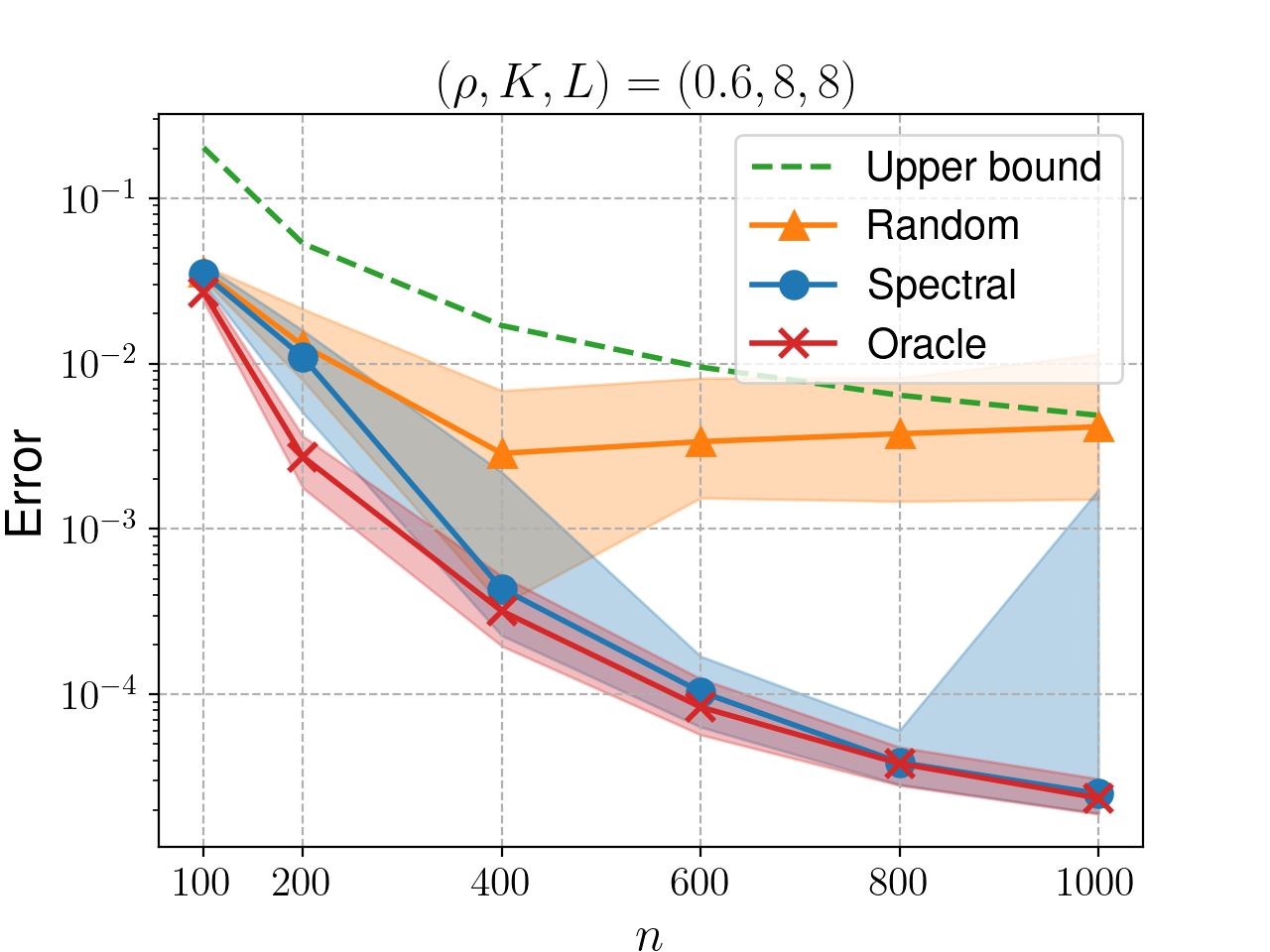}
%         \caption{$y=3sinx$}
         \label{fig:(0.6 , 8 , 8)}
     \end{subfigure}
     \begin{subfigure}[b]{0.32\textwidth}
         \includegraphics[width=\textwidth]{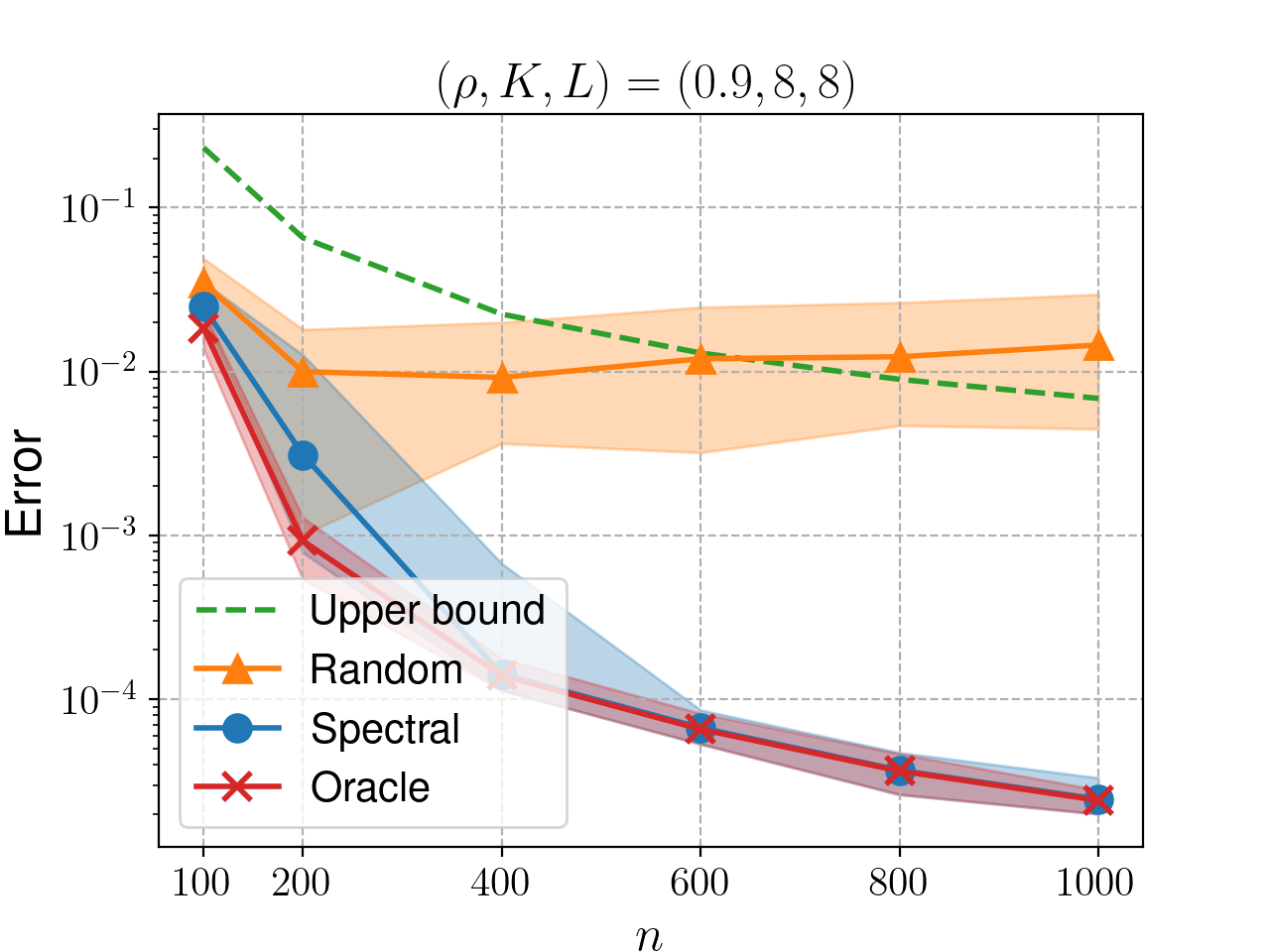}
%         \caption{$y=5/x$}
         \label{fig:(0.9 , 8 , 8)}
     \end{subfigure}
     \begin{subfigure}[b]{0.32\textwidth}
         \includegraphics[width=\textwidth]{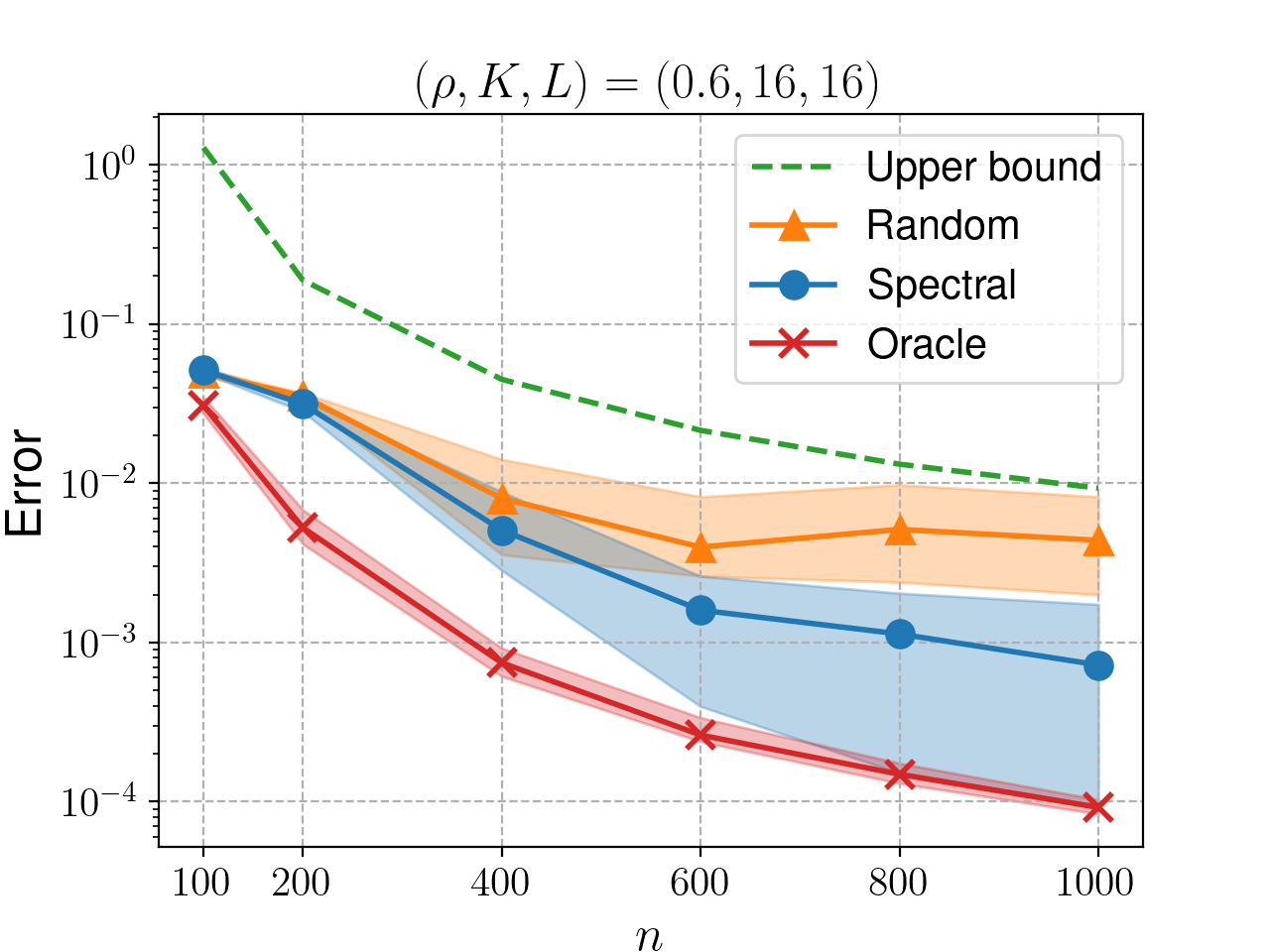}
%         \caption{$y=5/x$}
         \label{fig:(0.6 , 16 , 16)}
     \end{subfigure}
     \begin{subfigure}[b]{0.32\textwidth}
         \includegraphics[width=\textwidth]{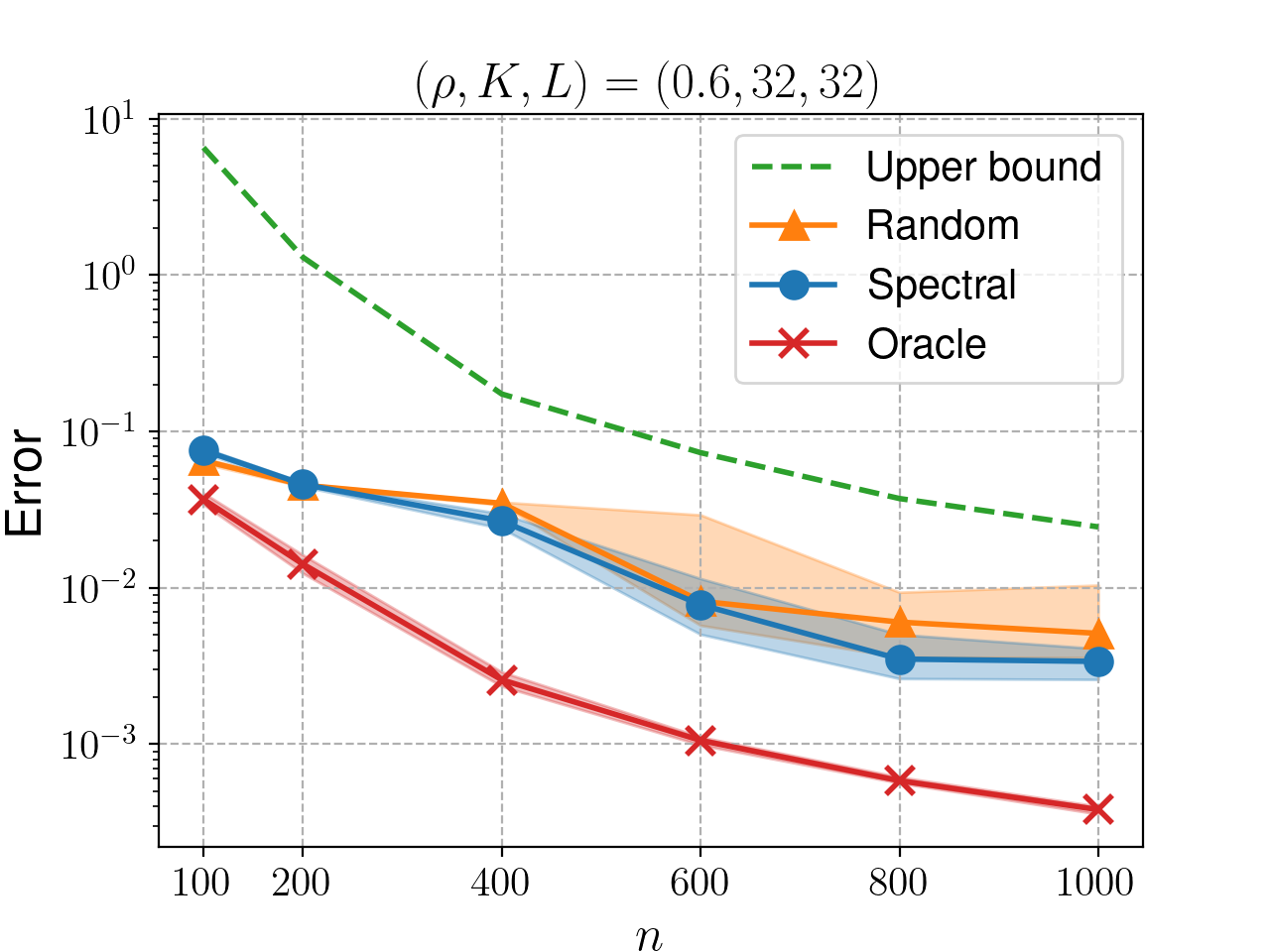}
%         \caption{$y=5/x$}
         \label{fig:(0.6 , 32 , 32)}
     \end{subfigure}
     \begin{subfigure}[b]{0.33\textwidth}
         \includegraphics[width=\textwidth]{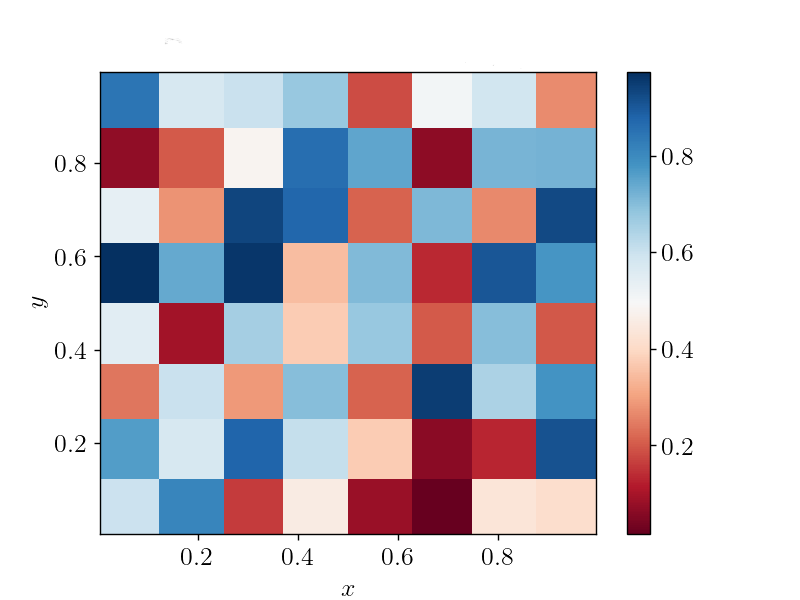}
%         \caption{$y=5/x$}
         \label{Graphon_rand}
     \end{subfigure}
        \caption{Graphon estimation in the rand-graphon set-up.}
        \label{fig:error_n}
\end{figure}

In the set-up of cos-graphon, we displayed in 
\Cref{fig:error_rho} the behavior of the error as
a function of $\rho$. For small values of $\rho$, 
the random initialization appears to be better than 
the spectral one. Moreover, the error is increasing 
for $\rho \in [0,\nicefrac12]$ for both 
initializations. The reason for such a behavior is 
that the estimator we computed tries to mimic the oracle
estimator, which knows the clusters and estimates the
matrix $\bTheta^*$ by computing cluster-wise averages. 
Thus, if $\bTheta ^* = (\bfZleft)^* \bfQ^* (\bfZright)^*$, 
where matrices $(\bfZleft)^*$ and  $(\bfZright)^*$ 
represent the clusters, then the oracle $\widehat 
\bTheta_o= (\bfZleft)^* \widehat \bfQ_o (\bfZright)^*$ 
satisfies 
\begin{align}
   \frac{\Ex \big[\| \hat\bTheta_o - \bTheta^* 
   \|_{\sf F}^2 \big]}{nm} &= \sum_{k,l} 
   \frac{Q_{kl}^*(1-Q_{kl}^*)}{nm} 
   =   \sum_{k,l}  \frac{\rho\tilde Q_{kl}(1- \rho 
   \tilde Q_{kl})}{nm}
   =\frac{\rho(\|\tilde\bfQ\|_{1,1} - \rho\|\tilde\bfQ
   \|_{\sf F}^2)}{nm} \label{tildeQ}
\end{align}
In the above formula, we used the matrix $\tilde\bfQ = 
\bfQ^*/\rho$, which has all its entries in $[0,1]$, 
and denoted by $\|\tilde\bfQ\|_{1,1}$ the sum of 
entries of the matrix $\tilde\bfQ$.  The right-hand 
side of \eqref{tildeQ} is a function of $\rho$ that 
increases for $\rho \in [0 , \|\tilde\bfQ\|_{1,1}/2
\|\tilde\bfQ\|^2_{\sf F}]$, and decreases outside 
this interval. 

\subsection{Estimation error for H\"older-continuous 
graphons}

To illustrate the behavior of the estimator of the 
graphon in the case where the latter is 
H\"older-continuous, we consider the function
displayed in \Cref{fig:Graphon} and given by
\begin{align}
    W^*(u,v)= \frac\rho2 \Big(1 + \exp\big\{-10
    \big((u - \nicefrac12)^2 + (v-\nicefrac12)^2
    \big)\big\} \Big).
\end{align}
This function being Lipschitz-continuous, we have
$\alpha = 1$.

The average squared error of estimation over 50 
repetitions for different values of $n$ and $\rho$ 
is depicted in \Cref{fig:error_hold}. Since the 
true distance $\delta$ is prohibitively hard to 
compute (because of the minimization over all 
measure preserving bijections), we computed an 
approximation of it denoted by $\tilde\delta$. 
Roughly speaking, $\tilde\delta$ is obtained
from $\delta$ by replacing the minimum over all
bijections $\tau_1,\tau_2$ by the value of
the cost function at the particular instances
of bijections, $\tau_{\sigma_1}$ and $\tau_{
\sigma_2}$, used in the proof of \Cref{prop:2}. 
More precisely, if $\sigma_1$ and $\sigma_2$
are permutations of $[n]$ and $[m]$, respectively,
such that the sequences $(U_{\sigma_1^{-1}(i)}
)_{i\in[n]}$ and $(V_{\sigma_2^{-1}(j)}
)_{j\in[m]}$ are nondecreasing, then 
\begin{align}
    \tau_{\sigma_1}(u) = \sum_{i=1}^n 
    \Big(\frac{\sigma_1(i)-1}{n} - u - 
    \frac{i-1}{n}\Big)\mathds 1_{[i-1,i)}(nu),
    \qquad \tau_{\sigma_1}(1) = 
    \frac{\sigma_1(n)}{n},\\
    \tau_{\sigma_2}(v) = \sum_{j=1}^m 
    \Big(\frac{\sigma_2(j)-1}{m} - v - 
    \frac{j-1}{m}\Big)\mathds 1_{[j-1,j)}(nu),
    \qquad \tau_{\sigma_2}(1) = 
    \frac{\sigma_2(m)}{m}.
\end{align}
Then, we define
 \begin{align}
    \tilde \delta^2(\widehat W , W^*) 
    &= \big\|  W^*\circ (\tau_{\sigma_1} \otimes 
    \tau_{\sigma_2}) - \widehat W\|^2_{\mathbb L_2}\\
        &= \big\|  W^*\|^2_{\mathbb L_2} -2 
        \sum_{i=1}^{n} \sum_{j=1}^{m}  
        \widehat \Theta_{ij} \int_{\frac{
        \sigma_1(i)-1}{n}}^{\frac{\sigma_1(i)}{n} } 
	\int_{\frac{\sigma_2(j)-1}{m} 
        }^{\frac{\sigma_2(j)}{m} } W^*(x, y) 
        dxdy + \| \widehat \bTheta \|_{\sf F}^2 .
 \end{align}
In numerical experiments, the integrals appearing 
in the right-hand side of the last display are 
approximated by the Riemann sums. In this case also 
we observe that the error curves obtained by Monte 
Carlo simulations are of the same shape as those
predicted by the theory. Interestingly, and somewhat
surprisingly, the random initialization behaves as
well as the spectral one. We do not have any explanation
for this observation at this stage. 

\begin{figure}%[h]
    \centering
     \begin{subfigure}[b]{0.32\textwidth}
         \includegraphics[width=\textwidth]{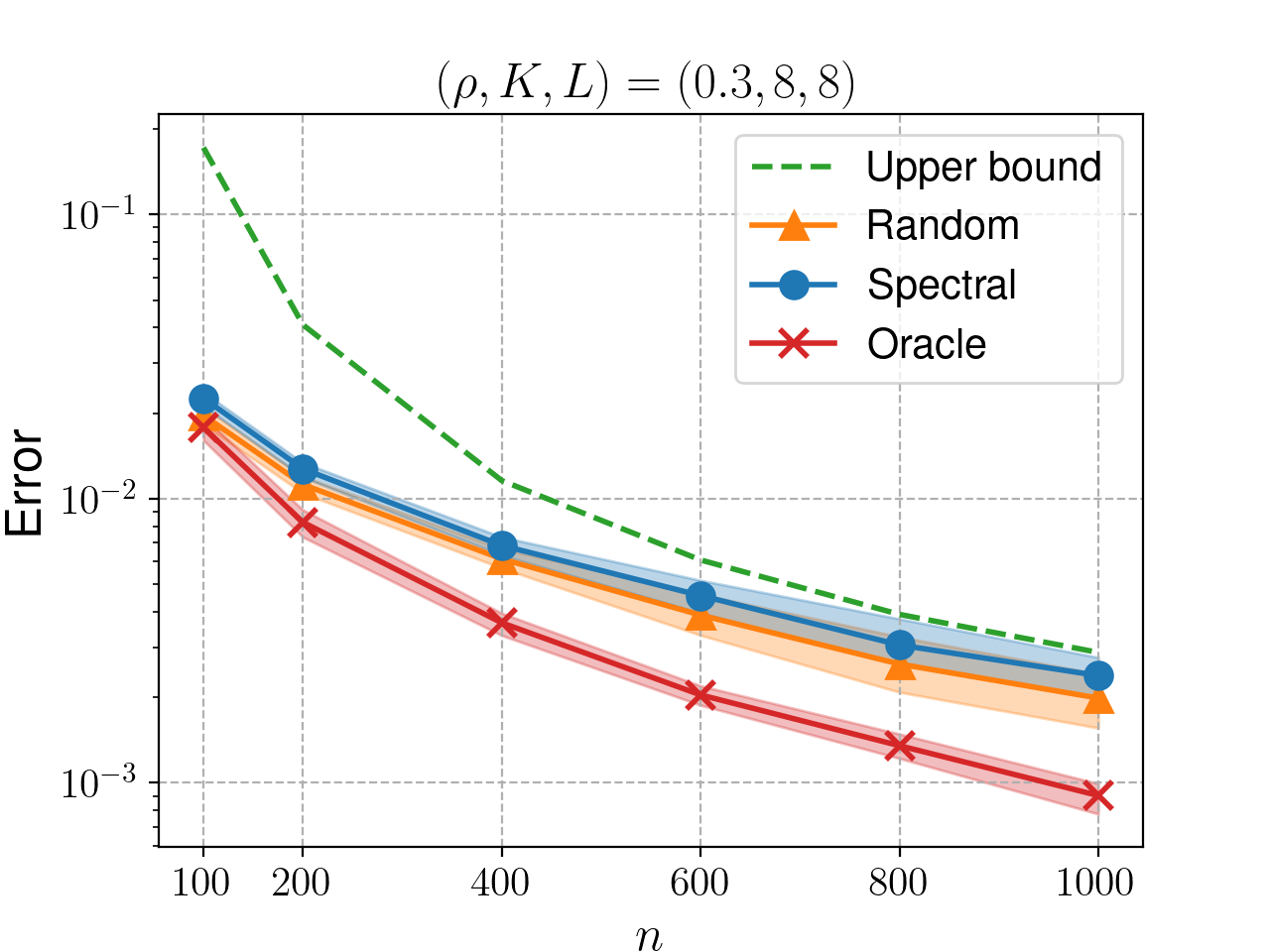}
%         \caption{$y=x$}
         \label{fig:(0.3 , 8 , 8)_cos}
     \end{subfigure}
     \begin{subfigure}[b]{0.32\textwidth}
         \includegraphics[width=\textwidth]{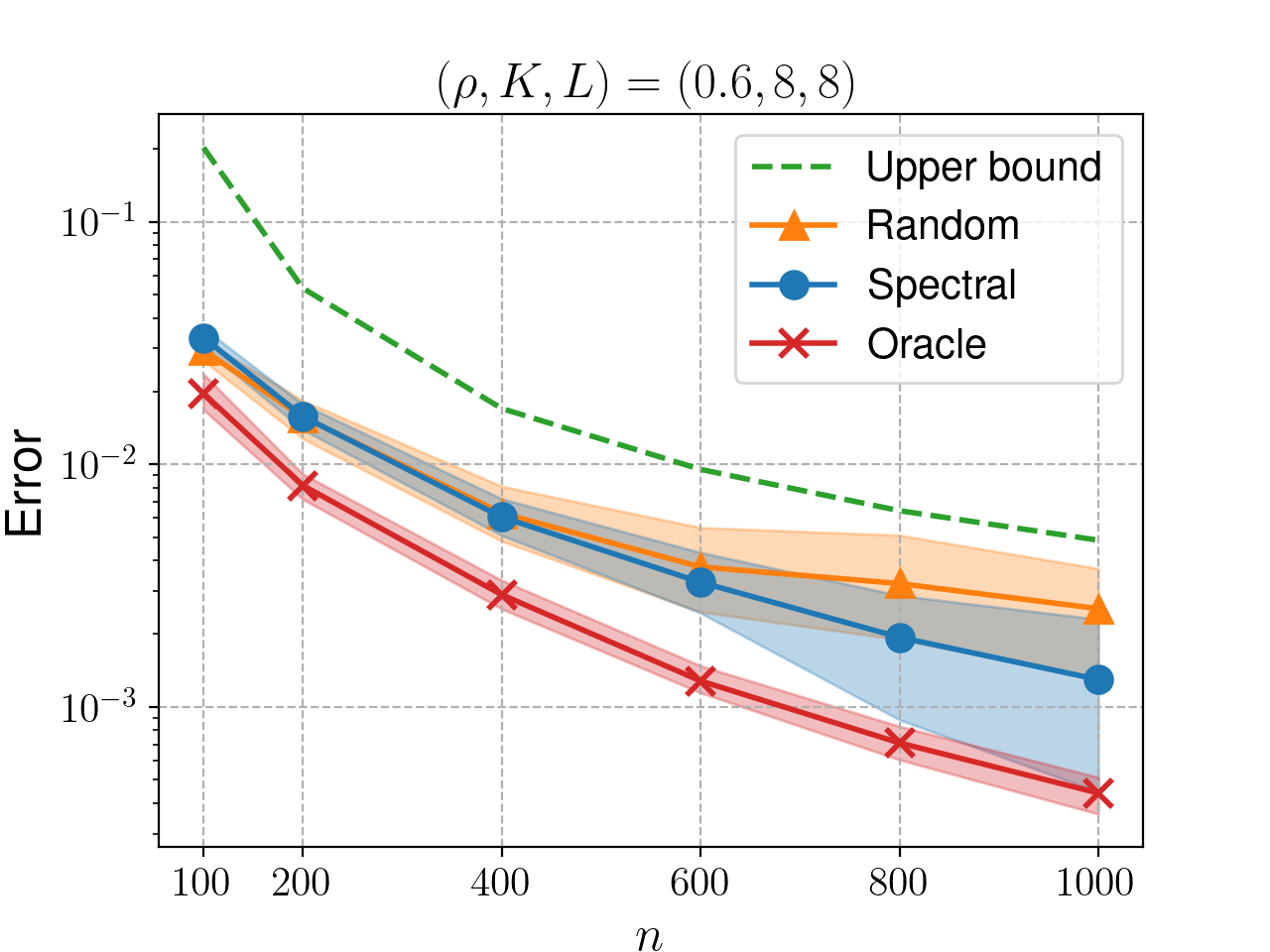}
%         \caption{$y=3sinx$}
         \label{fig:(0.6 , 8 , 8)_cos}
     \end{subfigure}
     \begin{subfigure}[b]{0.32\textwidth}
         \includegraphics[width=\textwidth]{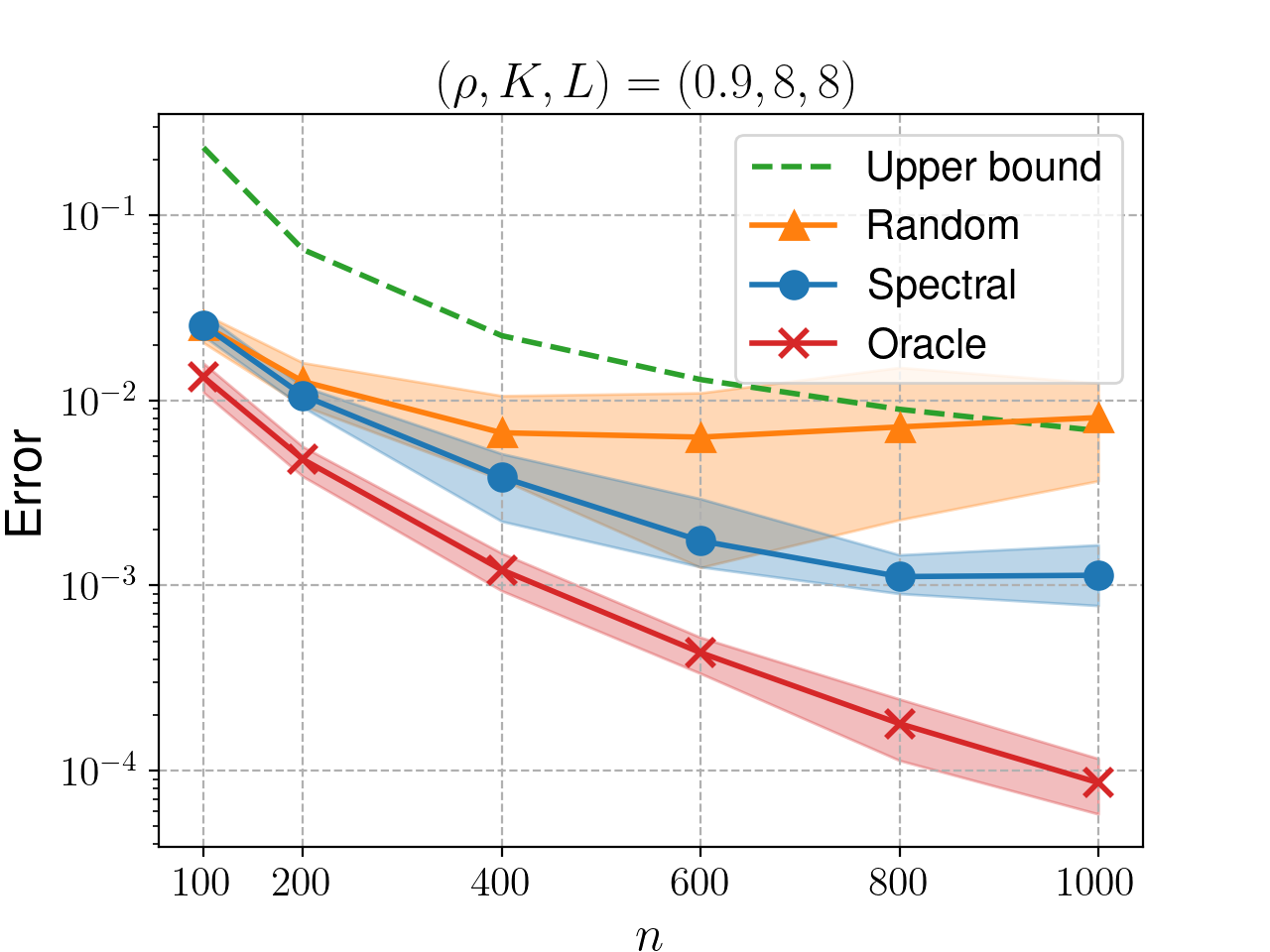}
%         \caption{$y=5/x$}
         \label{fig:(0.9 , 8 , 8)_cos}
     \end{subfigure}
     \begin{subfigure}[b]{0.32\textwidth}
         \includegraphics[width=\textwidth]{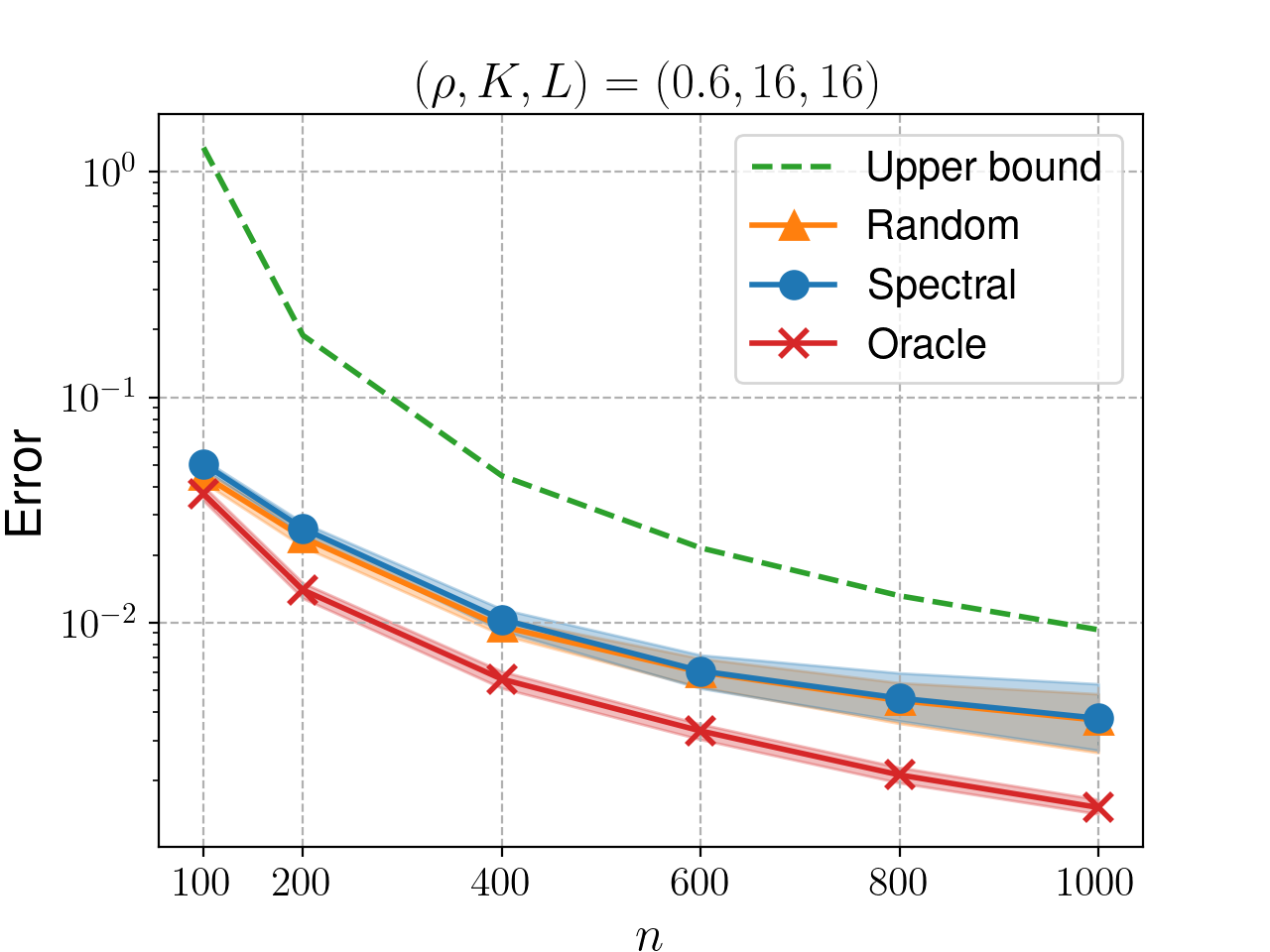}
%         \caption{$y=5/x$}
         \label{fig:(0.6 , 16 , 16)_cos}
     \end{subfigure}
     \begin{subfigure}[b]{0.32\textwidth}
         \includegraphics[width=\textwidth]{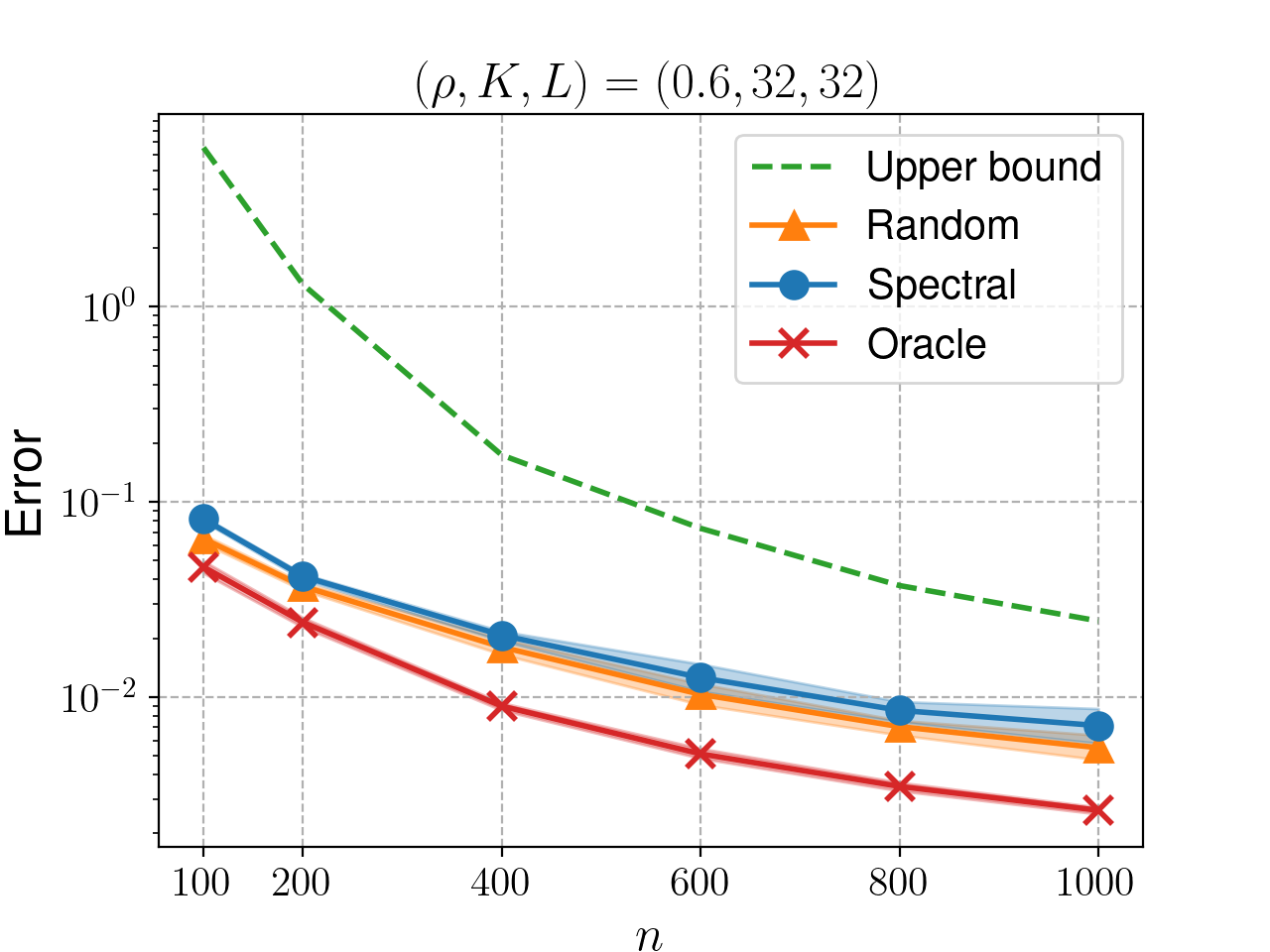}
%         \caption{$y=5/x$}
         \label{fig:(0.6 , 32 , 32)_cos}
     \end{subfigure}
     \begin{subfigure}[b]{0.33\textwidth}
         \includegraphics[width=\textwidth]{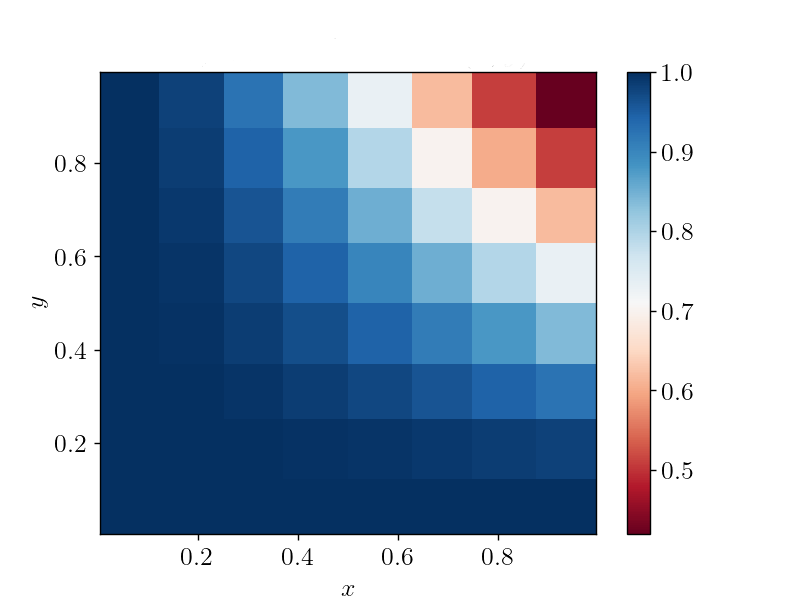}
%         \caption{$y=5/x$}
         \label{Graphon_cos}
     \end{subfigure}
        \caption{Graphon estimation in the cos-graphon 
		set-up, $W^*(x,y) = 
        \frac{2\rho}3 + \frac{\rho}3\cos \big(3 \pi \lfloor Kx 
        \rfloor \lfloor Ly \rfloor \big)$.}
        \label{fig:error_n_cos}
\end{figure}

\etocsettocdepth{subsection}
\section{Proofs of results stated in previous sections}
\label{sec:proofs}

%\secttoc

\subsection[Proof of  \texorpdfstring{\Cref{prob_matrix}} ( (risk bound 
for LSE of the mean)]{Proof of \Cref{prob_matrix} 
(risk bound for $\bThetaLS$)}

%\begin{proof}[Proof of \Cref{prob_matrix}]
Let us define
\begin{align}
    \Pi_{\calT}(\bTheta^*) = \arg  \min_{\bTheta \in \calT%_{n_0,m_0}[K,L] 
    } \| \bTheta - \bTheta^*  \|_{\sf F},
\end{align}
the best approximation of $\bTheta^*$ in Frobenius norm 
by a constant-by-block matrix. Note that the matrix 
$\Pi_{\calT}(\bTheta^*)$ has at most $KL$ distinct entries
each of which is the average of the entries of a submatrix
of $\bTheta^*$. Since $\widehat \bTheta = 
\hatbfZleft \widehat \bfQ (\hatbfZright)^\top $ is the 
least square estimator, we have 
\begin{align}
    \big\| \bfH - \widehat \bTheta \big\|_{\sf F}^2 \leqslant  
    \big\| \bfH - \Pi_{\calT}(\bTheta^*) \big\|_{\sf F}^2. \label{LS_ineq}
\end{align}
Let us define the mean-zero ``noise'' matrix $\bfE= \bfH - 
\Ex[\bfH] = \bfH - \bTheta^*$ and rewrite \eqref{LS_ineq} in the 
following form 
\begin{align}\label{decomp:0}
    \| \widehat \bTheta - \bTheta^* \|_{\sf F}^2 
        \leqslant \| \bTheta^* - \Pi_{\calT}(\bTheta^*) '.
        \|_{\sf F}^2 + 2 \langle \widehat \bTheta - 
        \bTheta^* , \bfE \rangle +  2 \langle \bTheta^* 
        - \Pi_{\calT}(\bTheta^*) , \bfE \rangle.
\end{align}
The expectation of $\bfE$ being zero, the same is true for the last term in the right-hand. We want to 
bound the expectation of $\langle \widehat \bTheta - 
\bTheta^* , \bfE \rangle$. To this end, we define 
\begin{align}
    \widehat \calT=\left\{ \bTheta : \, \exists \bfQ 
    \in \mathbb R^{K \times L} \text{ such that } 
    \bTheta = \hatbfZleft \bfQ (\hatbfZright)^\top  
    \right\} \subset\calT ,
\end{align}
and let $ \Pi_{\widehat\calT}(\bTheta^*) = \arg 
\min_{{\bTheta \in \widehat \calT}} \| \bTheta - 
\bTheta^*  \|_{\sf F}  $ be the best Frobenius 
approximation of $\bTheta^*$ in $\widehat \calT$. 
We use the decomposition 
\begin{align}\label{decomposition}
    \langle \widehat \bTheta - \bTheta^* , \bfE \rangle = \underbrace{\langle  \Pi_{\widehat\calT}(\bTheta^*) - \bTheta^* , \bfE \rangle}_{\Xi_1}
    + \underbrace{\langle \widehat \bTheta - \Pi_{\widehat\calT}(\bTheta^*) , \bfE \rangle}_{\Xi_2}.
\end{align}

\begin{figure}%[h]
    \centering
     \begin{subfigure}[b]{0.33\textwidth}
         \includegraphics[width=\textwidth]{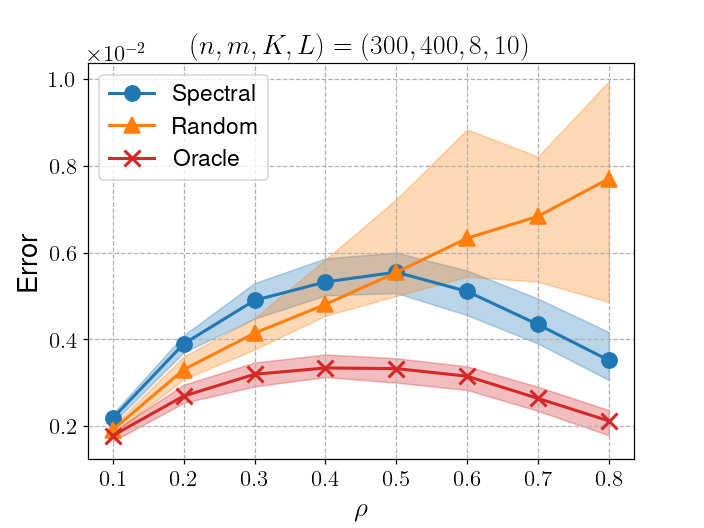}
%         \caption{$y=3sinx$}
         \label{fig:(400,200,20,10)}
     \end{subfigure}
     \vspace{-10pt}
     \caption{Estimation error as a function of $\rho$, 
     for $W^*(x,y)= \frac{2\rho}{3} +\frac{\rho}{3}
     \cos ( 3\pi \lfloor Kx \rfloor \lfloor Ly \rfloor)$.}
        \label{fig:error_rho}
\end{figure}

%%%%%%%%%%%%%%%%%%  Lemma Xi_1 %%%%%%%%%%%%%%%%%%%%%%%%%
\begin{lemma}\label{lem:Xi1}
    Under the conditions of \Cref{prob_matrix}, we have
    \begin{align}
        \Ex(\Xi_1) &\leqslant  \sigma \sqrt{ 2
        (n \log K + m \log L +1)}\, \Ex[ \| \widehat 
        \bTheta - \bTheta^* \|_{\sf F}^2]^{\nicefrac12}
        + b\rho  (n \log K + m \log L + 1).
    \end{align}
\end{lemma}
%%%%%%%%%%%%%  Proof of Lemma Xi_1 %%%%%%%%%%%%%%%%%%%%

\begin{proof} 
    The main steps of the proof consist in applying the
    Bernstein inequality to $\Xi_1$ for a fixed $\bTheta$
    instead of $\Pi_{\hat T}(\bTheta^*)$, using the union
    bound and then integrating the high-probability 
    bound. For the first step, let $\bTheta\in \mathbb 
    R^{n\times m}$ satisfy $\Theta_{i,j}\in [0,\rho]$ 
    for every $(i,j)\in[n]\times[m]$. By definition of 
    the inner product, we have $\langle \bTheta - 
    \bTheta^*, \bfE \rangle = \sum_{(i,j)\in[n]\times[m]}
    (\bTheta - \bTheta^*)_{ij} E_{ij}$. The random 
    variables $E_{ij}$ are independent and satisfy the
    $(\sigma^2, b)$-Bernstein condition. The
    $nm$-vector with entries $(\bTheta - \bTheta^*)_{ij}$ 
    has an infinity norm bounded by $\rho$. Therefore, 
    the version of the Bernstein inequality stated in
    \Cref{Bernstein:ineq} implies that for all $x>0$, we 
    have
    \begin{align}
        \prob\Big( \langle \bTheta - \bTheta^* 
        , \bfE \rangle \geqslant \sqrt{2x}\,\sigma 
        \| \bTheta - \bTheta^*\|_{\sf F} +
        b\rho  x \Big) \leqslant e^{-x}.
    \end{align}
    Let us define $\Omega_{\bfZ,\bfZ'}= \{(\hatbfZleft 
    ,\hatbfZright) = (\bfZ,\bfZ')\}$, for each pair of
    matrices $\bfZ \in \ZZ_{n,K,n_0}  $ and $\bfZ'\in
    \ZZ_{m,L,m_0}$. On the event $\Omega_{\bfZ,\bfZ'}$,
    the matrix $\Pi_{\widehat\calT}(\bTheta^*)$ is
    deterministic and its elements are averages of 
    the elements of $\bTheta^*$. Hence, $0\leqslant  
    (\Pi_{\widehat\calT} (\bTheta^*))_{ij} \leqslant
    \|\bTheta^*\|_\infty$ and
    \begin{align}
        \prob\left( \big\{ \langle  \Pi_{\widehat\calT} 
        (\bTheta^*) - \bTheta^* , \bfE \rangle
        \geqslant 
        \sqrt{2x}\,\sigma \| \Pi_{\widehat\calT} 
        (\bTheta^*) - \bTheta^*\|_{\sf F}  +
        b\rho  x \big\}\cap 
        \Omega_{\bfZ,\bfZ'}\right) \le e^{-x}. 
    \end{align}
    Note also that the cardinality of $\ZZ_{n,K,n_0}$
    is at most $K^n$. Combining the last display with
    the union bound, we get
    \begin{align}
        \prob\Big( \langle \Pi_{\widehat\calT}(\bTheta^*) 
        &- \bTheta^*, \bfE \rangle \geqslant \sqrt{2x
        \rho } \, \sigma \|\Pi_{\widehat\calT} (\bTheta^*) 
        - \bTheta^*\|_{\sf F}  + b\rho\sigma^2  x \Big)\\
        & \leqslant \sum_{(\bfZ,\bfZ')} \prob\left( 
        \big\{ \langle \Pi_{\widehat\calT}(\bTheta^*) 
        - \bTheta^* , \bfE \rangle \geqslant
        \sqrt{2x}\,\sigma \|\Pi_{\widehat\calT} 
        (\bTheta^*) - \bTheta^*\|_{\sf F}  + b\rho  
        x \big\} \cap \Omega_{\bfZ,\bfZ'}\right)\\
        &\leqslant K^n L^m e^{-x}, 
    \end{align}
    where in the first inequality in the above display
    the sum is over $(\bfZ,\bfZ')$ from the set
    $\ZZ_{n,K,n_0}  \times \ZZ_{m,L,m_0} $ and the
    factor $K^nL^m$ corresponds to an upper bound on
    the cardinality of this set. Finally, choosing 
    $x = n\log K + m\log L + t$ for some $t>0$ and
    using the basic inequality $uv \leqslant \lambda
    u^2 + v^2/(4\lambda)$ entails  
    \begin{align}
        \prob\left( \Xi_1 \geqslant \lambda \|
        \Pi_{\widehat\calT} (\bTheta^*) - \bTheta^* 
        \|_{\sf F}^2 + \Big(\frac{\sigma^2}{2\lambda} 
        + b\rho\Big) (n \log K + m \log L + t)\right)
        \leqslant e^{-t} \,
    \end{align}
    for any $\lambda>0$. \Cref{exp_bound} below 
    ensures that 
    \begin{align}
        \Ex(\Xi_1)  & \leqslant \lambda \Ex\Big [ \| \Pi_{\widehat\calT}(\bTheta^*) - \bTheta^* 
        \|_{\sf F}^2 \Big] + \Big(\frac{\sigma^2}{ 
        2\lambda} + b\rho \Big)   (n \log K 
        + m \log L +1)  \\
        & \leqslant  \lambda \Ex\Big [ \| \widehat 
        \bTheta - \bTheta^* \|_{\sf F}^2 \Big] + \Big(\frac{\sigma^2}{2\lambda} + b\rho \Big)  
        (n \log K + m \log L +1).
    \end{align}
     
 \begin{figure}%[h]
    \centering
    \begin{subfigure}[b]{0.45\textwidth}
         \includegraphics[width=\textwidth]{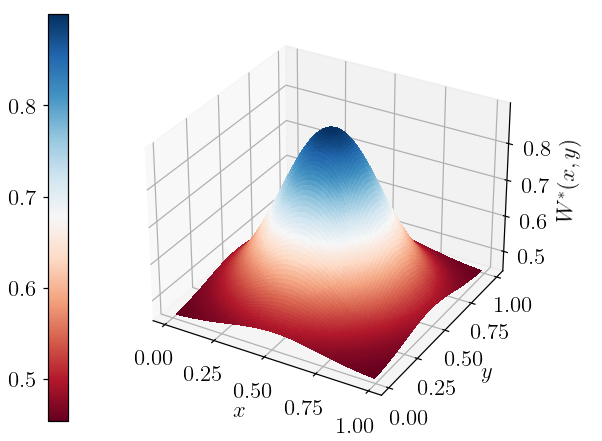}
         \caption{True graphon $W^*$}
    \end{subfigure}
    \begin{subfigure}[b]{0.45\textwidth}
         \includegraphics[width=\textwidth]{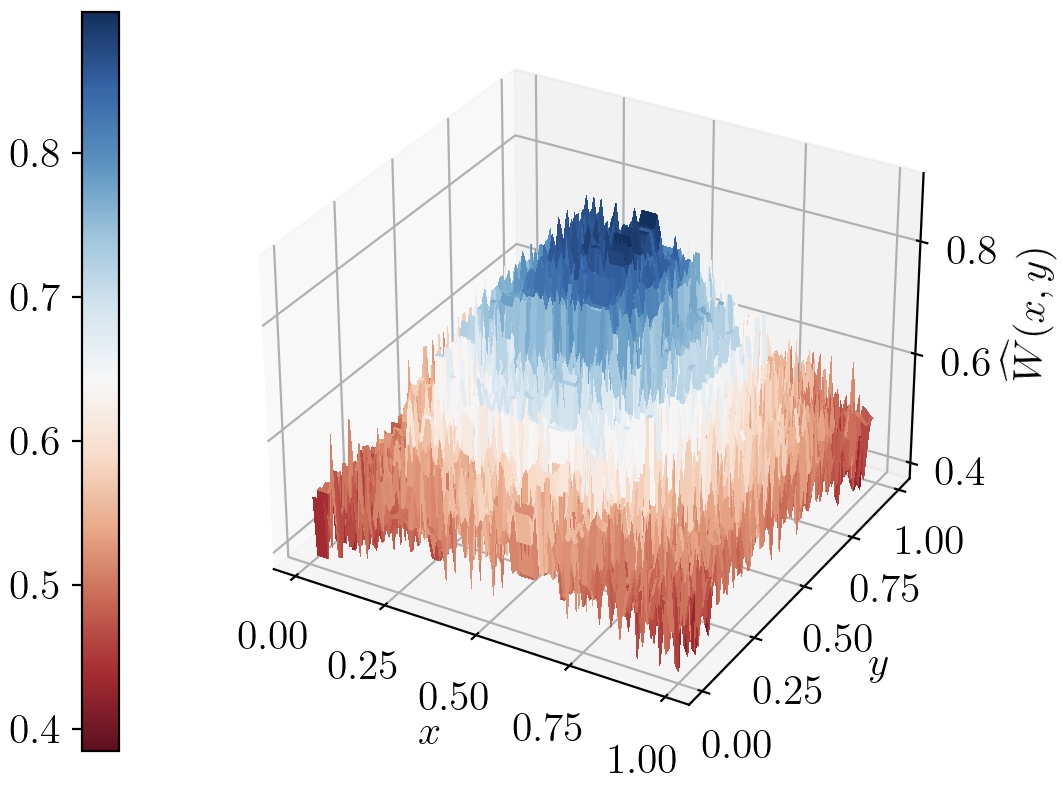}
         \caption{Estimated graphon $\widehat{W}$}
    \end{subfigure}
    \caption{Graphon representation for $\rho=0.9$.}
    \label{fig:Graphon}
 \end{figure}

    Optimizing with respect to $\lambda>0$, we get 
    \begin{align}
        \Ex(\Xi_1)  &\leqslant  \sigma\sqrt{2 
        (n \log K + m \log L +1)} \,\Ex[ \| \widehat 
        \bTheta - \bTheta^* \|_{\sf F}^2]^{\nicefrac12} 
        + b\rho   (n \log K + m \log L + 1).
        \label{bound:I}
    \end{align}
    This completes the proof of the lemma.
\end{proof}

We now switch to the evaluation of $\Ex(\Xi_2)$.  
To this end, we first notice that 
\begin{align}
    \Ex(\Xi_2) = \Ex\big[\langle \Pi_{\widehat\calT}(\bfH) 
    - \Pi_{\widehat\calT}(\bTheta^*) , \bfE \rangle\big] 
    = \Ex\big[\langle \Pi_{\widehat\calT}(\bfE) , \bfE \rangle\big]
    = \Ex\big[\| \Pi_{\widehat\calT}(\bfE)\|_{\sf F}^2\big]. 
\end{align}

\begin{lemma}\label{lem:Xi2} Under the conditions of 
    \Cref{prob_matrix}, we have
    \begin{align}
        \Ex\big[\|\Pi_{\widehat\calT}(\bfE)\|_{\sf F}^2\big]
        \leqslant  4(b+\sigma^2)\big(3KL+n\log K + m\log 
        L\big) \big( 2\rho + b\psi_{n,m}(n_0,m_0)\big).
    \end{align}
\end{lemma}

\begin{proof} 
Recall that $\psi_{n,m}(n_0,m_0) = \frac{3 \log 
(en/n_0)}{m_0}  + \frac{3\log (em/m_0)}{n_0}$. 
The scheme of the proof is to apply successively 
\Cref{partition1} and \Cref{partition2}. We will 
proceed by vectorizing the matrices in order to work with 
vectors only. To this end, let us consider an 
arbitrary bijection
\begin{align}
    \phi : [n] \times [m] \to [nm] \, . 
\end{align}
Let $\calN$ and $\calM$ be partitions of $[n]$ and $[m]$, 
respectively, satisfying 
\begin{align}
    |\calN|=K,\quad  |\calM|=L \qquad \text{and} \qquad
    \min_{A \in \calN} |A| \geqslant n_0\quad \min_{B \in \calM} 
    |B| \geqslant m_0.
\end{align}
We define $\calN \times \calM = \{A \times B : \, A \in \calN, \, 
B \in \calM \}$ which is a partition of $[n] \times [m]$ of 
cardinality $|\calN \times \calM|=KL$ satisfying $|A \times B| 
\geqslant n_0m_0 $ for every $A \times B \in \calN \times \calM$. 
We denote by $\calG$ the family of all partitions 
$\phi(\calN\times\calM)$, where $\calN$ and $\calM$ are as above.  

\begin{figure}%[h]
    \centering
    \hglue-3pt
    \includegraphics[width=0.35\textwidth]{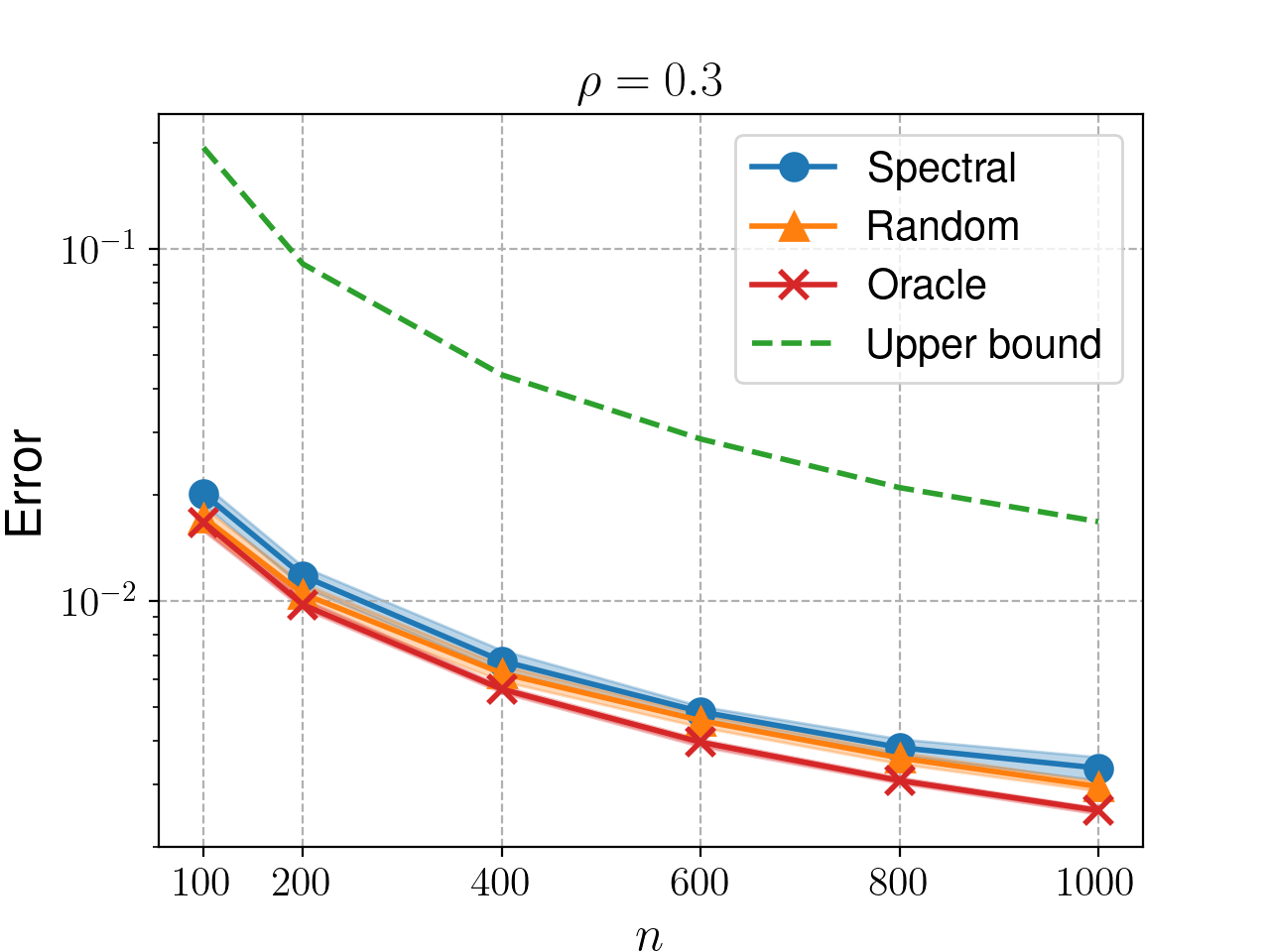}
    \hglue-13pt
    \includegraphics[width=0.35\textwidth]{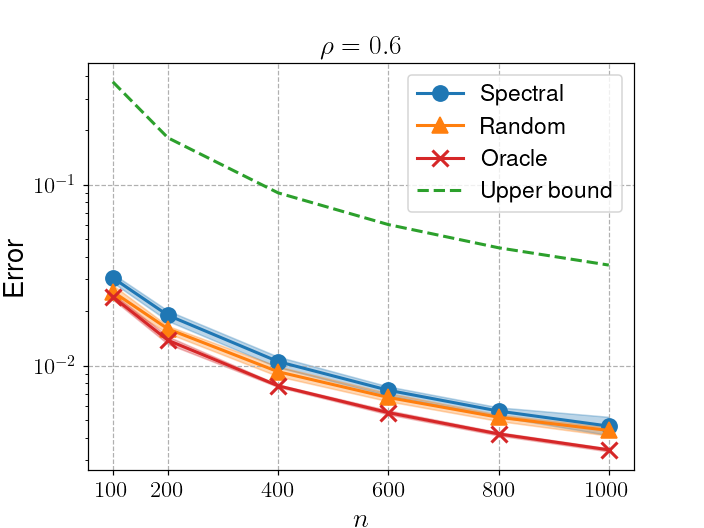}
    \hglue-13pt
    \includegraphics[width=0.35\textwidth]{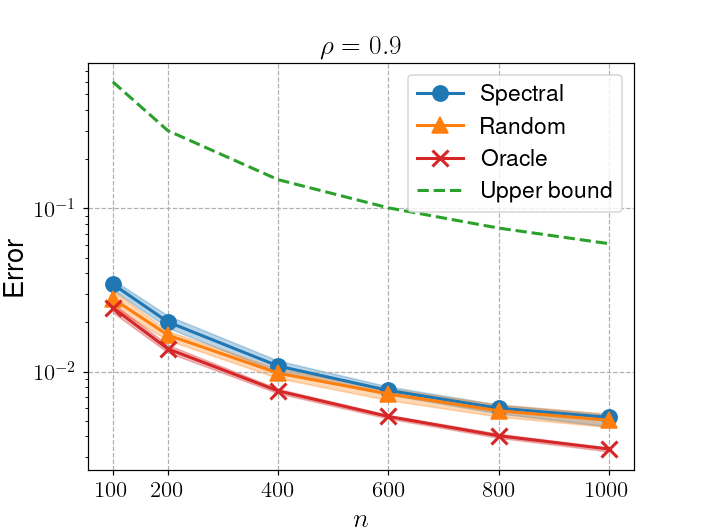}
    \vglue-8pt
    \caption{Evolution of the estimation error for a 
    Lipschitz graphon, $m={n}/{2}$.}
    \label{fig:error_hold}
\end{figure}

Since the entries of $\Pi_{\widehat\calT}(\bfE)$ are averages 
of coefficients of $\bfE$, if we vectorize $\bfE$ according 
to the map $\phi$, meaning that we define $(\vec \bfE)_i = 
E_{\phi^{-1}(i)}$ for all $i \in [nm]$, we have that
\begin{align}
    \| \Pi_{\widehat\calT}(\bfE)\|_{\sf F}^2 &  \leqslant 
    \max_{G \in \mathcal G}  \| \Pi_{G} \vec \bfE \|_2^2 %\\
    % & \leqslant \max_{\{A_1, \dots , A_{KL} \} \in \mathcal G} \sum_{j=1}^{KL} \frac{1}{|A_j|}\Big(\sum_{i \in A_j} 
    %E_{\phi^{-1}(i)} \Big)^2 \, .
\end{align}
with entries $(\vec \bfE)_i$ satisfying the assumptions 
of \Cref{partition1}. So by the union bound on $\calG$, we 
obtain that
\begin{align}\label{l2_norm}
    \| \Pi_{\widehat\calT}(\bfE)\|_{\sf F}^2 \leqslant  
    4 \big(t + \log (2M|\calG|)\big)\big( 
    2\sigma^2 + b \max_{G \in \mathcal G}
    \|\Pi_G\vec \bfE \|_\infty \big)
\end{align}
with probability at least $1 - 0.5e^{-t}$, where 
$\log M \leqslant KL \log 12 \leqslant 2.5 KL$.

Let $\calA$ be the family of all the cells of the
partitions in $\calG$, and let  $\calA_{(s,l)} = \{
\phi(A\times B)\in \calA :\, |A|=s \text{ and } |B| 
= l \}$ for $ j  = (s,l) \in [n]\times [m]$. Define 
\begin{align}
    \mathfrak F = \frac{\log(4nm)}{n_0m_0} + 
    \max_{s,l}\frac{\log|\calA_{(s,l)}|}{sl}
    \geqslant \frac{\log(4nm)}{n_0m_0} + 
    \frac{\log|\calA_{(n_0,m_0)}|}{n_0m_0} 
    \geqslant \frac{7}{n_0m_0},\qquad \forall 
    n,m\geqslant 3.
\end{align}
According to \Cref{partition2},  on an event of 
probability at least $1-0.5e^{-t}$,
\begin{align}
    \max_{G \in \mathcal G} \|\Pi_G\vec \bfE \|_\infty 
    & \leqslant \max_{A\in \calA}\frac1{|A|}\Big| 
    \sum_{\ell\in A} (\vec \bfE)_\ell\Big| \leqslant 
    \sigma \sqrt{2 \Big(\frac{t}{n_0m_0} + 
    \mathfrak F\Big)} + b\Big(\frac{t}{n_0m_0}
    + \mathfrak F\Big)\\
    & \leqslant \sigma \sqrt{2\mathfrak F (1+t/7)}
    + b \mathfrak F(1 + t/7)\\
    & \leqslant \sigma \sqrt{2\mathfrak F} \,( 1 
    + t/14 ) + b\mathfrak F(1+t/7)%\\
    %& \leqslant 0.5(1 + 2\sigma ^2\mathfrak F)( 1
    %+ t/14) + b\mathfrak F(1 + t/7)\\
    %& \leqslant (\rho + 2\sigma^2\mathfrak F(2b+1))(t/28)
    %+ 0.5\rho + \sigma^2 \mathfrak F(b+1)
    .\label{infite_nomr}
\end{align}
Finally, combining \eqref{l2_norm} and 
\eqref{infite_nomr}, we have that with probability 
at least $1 - e^{-t}$,
\begin{align}
    \| \Pi_{\widehat\calT}(\bfE)\|_{\sf F}^2 & 
    \leqslant 4\big(t + \log (2M|\calG|)\big)
    \big( 2\sigma^2 + b\sigma \sqrt{2\mathfrak F} \,( 1 
    + t/14 ) + b^2\mathfrak F(1+t/7) \big).
\end{align}
Using \Cref{exp_bound}, we obtain the following upper 
bound  on the expectation
\begin{align}
    \Ex\big[ \| \Pi_{\widehat\calT}(\bfE)\|_{\sf F}^2 \big] 
    & \leqslant 4\big(1 + \log (2M|\calG|)\big)
    \big( 2\sigma^2 + b\sigma \sqrt{2\mathfrak F} \,( 1 
    + 1/7 ) + b^2\mathfrak F(1+2/7)  \big)\\
    &\leqslant \big(1 + \log(2M|\calG|)\big)
    \big( 9\sigma^2  + 16 \mathfrak F b^2\big).
\end{align}
For any $(s,l)$ such that $n_0 \leqslant s \leqslant n$ 
and $m_0 \leqslant l \leqslant m$, we have $|\calA_{(s,l)}| \leqslant \binom{n}{s} \binom{m}{l}$, which implies 
\begin{align}
|\calA_{(s,l)}| \leqslant  \Big(\frac{en}{s}\Big)^{s} \Big(\frac{em}{l}\Big)^{l}.
\end{align}
Therefore, taking the logarithm of the two sides, we get
\begin{align}
    \frac{\log |\calA_{(s,l)}|}{sl} & \leqslant \frac{\log(en/s)}{l} + \frac{\log(em/l)}{s} 
    \leqslant \frac{\log (en/n_0)}{m_0} + \frac{\log (em/m_0)}{n_0} \, .
\end{align}
This implies that\footnote{This is true since
$\frac{\log(2n)}{n_0\log(e n/n_0)} = \frac1{n_0} + 
\frac{\log(2n_0/e)}{n_0\log(e n/n_0)}\leqslant 
\frac{\log(4n_0)}{n_0\log(2e)}\leqslant 0.5$ provided
that $n_0\geqslant 3$.} $\mathfrak F \leqslant
(1/2)\psi_{n,m}(n_0,m_0)$. Therefore,
\begin{align}
    \Ex\big[ \| \Pi_{\widehat\calT}(\bfE)\|_{\sf F}^2 \big] 
    & \leqslant \big(1 + \log(2M|\calG|)\big)
    \big( 9\sigma^2 + 8b^2 \psi_{n,m}(n_0,m_0)\big)
\end{align}
where $|\calG| = |\calF | \leqslant K^n L^m$ and $M 
\leqslant 12^{KL}$. Taking into account the fact that 
$K \geqslant 2$ and $ L\geqslant 2$, this leads to
\begin{align}
    1 + \log(2M|\mathcal G|) &  \leqslant 1 + \log 2 
    + KL \log 12 + n \log K + m \log L \\
    & \leqslant 3KL + n \log K + m \log L.
\end{align} 
The term $\Ex\big[ \| \Pi_{\widehat\calT}(\bfE)\|_{\sf
F}^2 \big]$ is eventually bounded as follows:
\begin{align}
    \Ex\big[\|\Pi_{\widehat\calT}(\bfE)\|_{\sf F}^2\big]
    \leqslant  \big(3KL+n\log K + m\log L \big) \big(
    9\sigma^2 + 8b^2 \psi_{n,m}(n_0,m_0)\big).
\end{align}
This completes the proof of the lemma.
\end{proof}

In order to ease notation in the rest of the proof, 
let us set $A =  n \log K + m \log L$. To conclude, 
we use the bounds  on $\Xi_1$ and $\Xi_2$ obtained 
in \Cref{lem:Xi1} and \Cref{lem:Xi2}, respectively, 
as well as decompositions \eqref{decomp:0} and 
\eqref{decomposition}. Since $\Ex[\langle \bTheta^* 
- \Pi_{\calT}(\bTheta^*) , \bfE \rangle] = \langle 
\bTheta^* - \Pi_{\calT}(\bTheta^*) , \Ex[\bfE] \rangle 
=0$, we arrive at 
\begin{align}
    \Ex[ \| \widehat\bTheta - \bTheta^*\|_{\sf F}^2] 
    &\leqslant \|\bTheta^* - \Pi_{\calT}\bTheta^*
    \|_{\sf F}^2 + 2 \Ex[\Xi_1] + 2 \Ex[\Xi_2]\\
    &\leqslant \|\bTheta^* - \Pi_{\calT}\bTheta^*
    \|_{\sf F}^2 + 2\sigma \sqrt{2(A+1)}\, \Ex[ \|
    \widehat \bTheta - \bTheta^* \|_{\sf F}^2]^{\nicefrac12}
    + 2b\rho(A+1) + 2 \Ex[\Xi_2].
\end{align}
One can check that the last inequality leads to
\begin{align}
    \big(\Ex[ \| \widehat \bTheta - \bTheta^* \|_{
    \sf F}^2]^{\nicefrac12} - \sigma\sqrt{2(A+1)} \big)^2
    &\leqslant \| \bTheta^* - \Pi_{\calT}
    (\bTheta^*) \|_{\sf F}^2 + 2(A+1) (\sigma^2+b\rho) 
    + 2 \Ex[\Xi_2].
\end{align}
This readily yields
\begin{align}
    \Ex[ \| \widehat \bTheta - \bTheta^* \|_{\sf F}^2
    ]^{\nicefrac12} &\leqslant \| \bTheta^* - \Pi_{\calT}
    (\bTheta^*) \|_{\sf F} + \sigma\sqrt{2 (A+1)} + 
    \big(2(A+1)(\sigma^2 + b\rho) + 2 \Ex[\Xi_2]
    \big)^{\nicefrac12}\\
    & \leqslant \| \bTheta^* - \Pi_{\calT}
    (\bTheta^*) \|_{\sf F} + \big((A+1)(8\sigma^2 + 
    4b\rho) + 4\Ex[\Xi_2]\big)^{\nicefrac12}\\
    & \leqslant \| \bTheta^* - \Pi_{\calT}
    (\bTheta^*) \|_{\sf F} + \big(17\sigma^2 + 4b\rho 
    + 8 b^2 \psi_{n,m}(n_0,m_0) \big)^{\nicefrac12} 
    (3KL + A)^{\nicefrac12},
\end{align}
where in the second line we have used the inequality 
$\sqrt{x} + \sqrt{y} \leqslant \sqrt{2x + 2y}$. Finally,
under the condition $\psi_{n,m}(n_0,m_0)\leqslant 
(\sigma/b)^2$, we get the claim of the theorem.
%\end{proof}

\subsection[Proof of \Cref{thm:1} (risk bound for the 
LSE of the graphon)]{Proof of \Cref{thm:1} (risk bound 
for $\hat W^{\sf LS}$)}

%\begin{proof}[Proof of \Cref{thm:1}]
\paragraph*{First claim: piecewise constant graphon}
In view of \eqref{eq5:0}, the fact that
$\hat W = W_{\hat\bTheta}$ and \Cref{prop:2}, 
we have
\begin{align}
    \Ex[\delta(W_{\hat\bTheta},W^*)^2]^{\nicefrac12} 
    &\leqslant  \frac{\Ex[\|\hat\bTheta - 
    \bTheta^*\|_{\sf F}^2]^{\nicefrac12}}{\sqrt{nm}} 
    + \Ex[\delta(W_{\bTheta^*},W^*)^2]^{\nicefrac12}\\
    &\leqslant \frac{\Ex[\|\hat\bTheta - 
    \bTheta^*\|_{\sf F}^2]^{\nicefrac12}}{\sqrt{nm}} 
    + \frac{\rho}{\sqrt{2}}\,\bigg(\sqrt{\frac{K
    }{n}} + \sqrt{\frac{L}{m}}\bigg)^{\nicefrac12}.
    \label{eq5:6b}
\end{align}
Let $\hat{\calT}$ and $\calT^*$ be the 
sets of all $n\times m$ matrices with real 
entries that are constant by block on the same 
blocks as $\hat\bTheta$ and $\bTheta^*$, 
respectively. Clearly, $\hat{\calT}$ and 
$\calT^*$ are linear subspaces of the 
space of $n\times m$ real matrices equipped 
with the scalar product $\langle \bfM_1, 
\bfM_2\rangle = \tr(\bfM_1^\top \bfM_2)$. Let
$\Pi_{\hat{\calT}}$ be the orthogonal projections 
onto $\hat{\calT}$. We have $\Pi_{\hat{\calT}}\bfH 
= \hat\bTheta$. Therefore, 
\begin{align}
    \|\hat\bTheta - \bTheta^*\|_{\sf F}
     = \|\Pi_{\hat{\calT}}\bfH - \bTheta^*\|_{\sf F} 
    &\stackrel{\circled{\tiny 1}}{\leqslant} 
    \|\Pi_{\hat{\calT}}(\bfH - \bTheta^*)\|_{\sf F}
    + \|\Pi_{\hat{\calT}}\bTheta^* - \bTheta^*\|_{\sf F}
    \\
    &\stackrel{\circled{\tiny 2}}{\leqslant} 
    \|\bfH - \bTheta^* \|_{\sf F} + \|(\rho/2)
    \mathbf 1_n\mathbf 1_m^\top - \bTheta^*\|_{\sf F}.
\end{align}
Above, $\circled{\tiny 1}$ is a consequence of the 
triangle inequality, whereas $\circled{\tiny 2}$ follows
from the  fact that $\Pi_{\hat{\calT}}$ is an orthogonal 
projection (hence, a contraction) and the matrix 
$(\rho/2)\mathbf 1_n\mathbf 1_m^\top$ belongs to the 
image of $\Pi_{\hat{\calT}}$. Hence
\begin{align}
    \frac1{nm}\Ex\big[\|\hat\bTheta - \bTheta^*
    \|_{\sf F}^2 \,\big|\,\bU,\bV \big] \leqslant 
    (\sigma + 0.5 \rho)^2.
\end{align}
For every $k\in[K]$ and $\ell\in[L]$, we define 
$n_k = n|a_k-a_{k-1}|$, $N_k = \#\{i:U_i\in[a_{k-1}, 
a_k[\}$, $m_\ell = m|b_{\ell}-b_{\ell-1}|$ and $M_\ell 
= \#\{j:V_j\in[b_{\ell-1}, b_\ell[\}$. We also define
the event $\Omega_0 = \{N_k\geqslant n_k/2; M_\ell 
\geqslant m_\ell/2\ \text{for all}\ k\in[K]\ \text{and}
\ \ell\in[L]\}$. Since the event $\Omega_0^c$ is
$(\bU,\bV)$-measurable, we get
\begin{align}
    \frac1{nm}\Ex\big[\|\hat\bTheta - \bTheta^*\|_{
    \sf F}^2 \mathds 1_{\Omega_0^c}\big] = \frac1{nm}
    \Ex\Big(\Ex\big[\|\hat\bTheta - 
    \bTheta^*\|_{\sf F}^2 \,\big|\,\bU,\bV \big]
    \mathds 1_{\Omega_0^c}\Big)\leqslant (\sigma 
    + 0.5\rho)^2\mathbb P(\Omega_0^c).\label{eq5:5}
\end{align}
Using the union bound and the Chernoff inequality, 
one can check that
\begin{align}
    \mathbb P(\Omega_0^c) &\leqslant \sum_{k=1}^K
    \mathbb P(N_k\leqslant n_k/2) + \sum_{\ell=1}^L
    \mathbb P(M_\ell \leqslant m_\ell/2)
    \leqslant \sum_{k=1}^K e^{-n_k/8} + 
    \sum_{\ell=1}^L e^{-m_\ell/8}.
\end{align}
Since we have assumed that $n_k\geqslant 8\log (nK)$ 
and $m_\ell \geqslant 8\log(mL)$, we get $\mathbb P( 
\Omega_0^c) \leqslant n^{-1} + m^{-1}$. If the parameters
$n_0$ and $m_0$ used in the definition of the least 
squares estimator $\hat\bTheta$ satisfy $n_0 = \min_k 
n_k/2 = n\Delta^{(K)}/2$ and $m_0 = \min_\ell m_\ell/2 
= m\Delta^{(L)}/2$, then on the event $\Omega_0$ 
we ca apply \Cref{prob_matrix}. One can check that
$\psi_{n,m}(n_0,m_0) = \psi_{n,m}(\Delta^{(K,L)})$. 
This, in conjunction with the previous inequalities, 
implies that
\begin{align}
    \frac{\Ex[\|\hat\bTheta - \bTheta^*\|_{\sf F}^2 
    ]}{{nm}} &= \frac{\Ex[\|\hat\bTheta - 
    \bTheta^*\|_{\sf F}^2\mathds 1_{\Omega_0}] + 
    \Ex[\|\hat\bTheta - \bTheta^* \|_{\sf F}^2
    \mathds 1_{\Omega_0^c}]}{nm}\\
    &\leqslant\big(25\sigma^2 + 4b\rho \big)
    \Big(\frac{3KL}{nm} + \frac{\log K}{m} + 
    \frac{\log L}{n}\Big) + \frac{
    (\sigma + 0.5\rho)^2}{n} + \frac{
    (\sigma + 0.5\rho)^2}{m}\\
    &\leqslant\Big\{\big(27\sigma^2 + 4b\rho\big)^{\nicefrac12}
    \Big(\frac{3KL}{nm} + \frac{\log K}{m} + 
    \frac{\log L}{n}\Big)^{\nicefrac12} + \frac{\rho}{2}
    \sqrt{\frac1n + \frac1m}\Big\}^2, 
\end{align}
under condition that $\psi_{n,m}(\Delta^{(K,L)})
\leqslant (\sigma/b)^2$. One can also check that if $K,L
\geqslant 2$ and $n,m\geqslant 5$, it holds 
\begin{align}
    \frac1n + \frac1m \le \frac13
    \bigg(\sqrt{\frac{K}n} + \sqrt{\frac{L}m}\bigg).
\end{align}
This inequality, combined 
with \eqref{eq5:6b}, completes the proof of the theorem. 
%\end{proof}

\paragraph*{Second claim: H\"older continuous graphons}

Using \eqref{eq5:0} and the Minkowski inequality, we get
\begin{align}
    \Ex[\delta(\WLS , W^*)^2]^{\nicefrac12} & = 
    \Ex[\delta(W_{\bThetaLS} , W^*)^2]^{\nicefrac12}
    \leqslant \frac{\Ex[\|\bThetaLS - \bTheta^*\|_{\sf F}^2
    ]^{\nicefrac12}}{\sqrt{nm}} + \Ex[\delta(W_{\bTheta^*} , 
    W^*)^2]^{\nicefrac12}.
\end{align}
Let us set 
\begin{align}
    K = L = \bigg\lfloor\bigg(\frac{3nm \calL^2}{25
    \sigma^2 +4b\rho}\bigg)^{1/2(\alpha+1)}\bigg\rfloor.
\end{align}
In view of \eqref{th2:1}, we have
\begin{align}\label{th2:2}
    (K/m)^{2(\alpha+1)}  \leqslant \frac{3n \calL^2}{(25
    \sigma^2 + 4b\rho)m^{2\alpha+1}} \leqslant \frac{3
    (\sigma/2b)^{4(\alpha+1)}\wedge 1}{\log^4(2n)}.
\end{align}
Let us choose $n_0 = \lfloor n/K\rfloor $ and $m_0 = 
\lfloor m/K\rfloor$. Thanks to \eqref{th2:2}, we
have
\begin{align}
    \frac{m}{K} \geqslant \big(\log^4(2n)\big)^{
    1/2(\alpha+1)} \geqslant \log 8 >2.
\end{align}
This implies that $m_0\geqslant 2$ and, therefore $n_0\geqslant 
2$. Using once again \eqref{th2:2}, one can check that
%question here
\begin{align}
    \frac{6\log(em/2)}{n_0}\leqslant 
    \frac{6\log(en/2)}{m_0} \leqslant 
    (\sigma/b)^2.
\end{align}
This implies that
\begin{align}
    \psi_{n,m}(n_0,m_0) & = 
    \frac{3\log (en/n_0)}{m_0} + \frac{3\log (em/m_0)}{n_0} 
     \leqslant \frac{3\log (en/2)}{m_0} + \frac{3\log (em/2)
     }{n_0} \leqslant (\sigma/b)^2.
\end{align}
Combining \Cref{prob_matrix}, \Cref{prop_aprox_error} and 
claim 2 of \Cref{prop:2}, we arrive at
\begin{align} 
    \Ex\big[\delta(\WLS , W^*)^2\big]^{\nicefrac12} 
    &\leqslant \frac{3\calL}{2K^\alpha} + \frac{3\calL}{2L^\alpha}
     + \big(25\sigma^2 + 4b\rho\big)^{\nicefrac12} \Big( 
    \frac{3KL}{nm} + \frac{\log K}{m} + \frac{\log L}{n}
    \Big)^{\nicefrac12}\\
    &\qquad + \frac{2\calL}{n^{\alpha/2}} + 
    \frac{2\calL}{m^{\alpha/2}}\\
    & \leqslant \frac{3\calL}{K^\alpha} 
     + \big(25\sigma^2 + 4b\rho\big)^{\nicefrac12} \Big( 
    \frac{3K^2}{nm} + \frac{2\log K}{m}\Big)^{\nicefrac12} 
    +  \frac{4\calL}{m^{\alpha/2}}\\
    &\leqslant \frac{3\calL}{K^\alpha}  + 3K\Big(\frac{25
    \sigma^2 + 4b\rho}{3nm}\Big)^{\nicefrac12} + \Big(
    \frac{(50\sigma^2 + 8b\rho)\log K}{m} \Big)^{
    \nicefrac12} + \frac{4\calL}{m^{\alpha/2}}.\label{ing_holder}
\end{align}
In the last display, replacing $K$ with its expression
\eqref{th2:2}, we get the claim of the theorem.

\begin{funding}
    This work was partially supported by the grant Investissements d’Avenir (ANR-11-IDEX0003/Labex Ecodec/ANR-11-LABX-0047), 
 funding from the European Research Council (ERC) under the European
Union's Horizon 2020 research and innovation programme grant agreement 741467-FIRMNET and the center Hi! PARIS.
\end{funding}

\bibliography{references.bib}

\itemsep=1cm
\newpage
%\begin{appendices}
%\dosecttoc
%\faketableofcontents
%\part{Appendix} % Start the appendix part
%\addcontentsline{toc}{part}{Appendix} % Add the appendix text to the document TOC
%\decrementstc
%
%\parttoc
\section{Proofs of the propositions}
%\decrementstc
%\secttoc
\subsection{Proof of \Cref{prop:2} (approximation error
for a graphon)}
\label{proof:prop2}

%\begin{proof}[Proof of \Cref{prop:2}]
\paragraph*{First claim (piecewise constant graphon)}
In what follows, $\lambda$ refers to the Lebesgue 
measure on $\mathbb R$ and $\lambda_2$ is the Lebesgue 
measure on $\mathbb R^2$. Let $W^*$ be a graphon such
that for some $K\times L$ matrix $\bfQ^*$ and some 
sequences $a_0<\ldots<a_K$, $b_0<\ldots<b_L$ 
satisfying $a_0=b_0=0$ and $a_K= b_L=1$, we have 
$W^*(u,v) = Q^*_{k,\ell}$ for every $u\in[a_{k-1},a_k)$
and $v\in[b_{\ell-1}, b_\ell)$. Equivalently, 
\begin{align}
    W^*(u,v)=\sum_{k=1}^K \sum_{\ell=1}^L \bfQ_{k,\ell}^*
    \mathds 1_{[a_{k-1},a_k)\times[b_{\ell-1},b_\ell)}
    (u,v).
\end{align}
Let us also define the ``weight'' sequences $w_k^{(1)} =
a_k-a_{k-1}$, $w_\ell^{(2)} = b_\ell-b_{\ell-1}$ and 
\begin{align}
    \hat w_k^{(1)} = \frac{1}{n} \sum_{i=1}^n 
    \mathds 1_{[a_{k-1},a_k[}(U_i) \qquad 
    \text{and} \qquad \hat w_\ell^{(2)} = 
    \frac{1}{m} \sum_{j=1}^m \mathds 1_{ [b_{\ell-1},
    b_\ell[}(V_j).
\end{align}
Notice that all the four weight sequences $w^{(1)}$, 
$w^{(2)}$, $\hat w^{(1)}$ and $\hat w^{(2)}$ are positive
and sum to one. As proved in \citep[p16]{Klopp2017}, 
there exist  two functions $\psi_1 : [0,1] \to [K]$ and
$\psi_2 : [0,1] \to [L]$ such that 
\vspace{-8pt}
\begin{enumerate}\itemsep = 0pt
    \item For all $k \in [K]$ and $x \in \big[a_{k-1},
    (a_{k-1} + \hat w_k^{(1)}) \wedge a_k \big)$, 
    we have $\psi_1(x) = k$ 
    \item  For all $\ell \in [L]$ and $x \in \big[ 
    b_{\ell-1}, (b_{\ell-1} +\hat w_\ell^{(2)}) \wedge
    b_\ell \big)$, we have $\psi_2(x) = \ell$
    \item $\lambda(\psi_1^{-1}(k))= \hat w_k^{(1)}$ 
    for all $k \in [K]$ 
    \item $\lambda(\psi_2^{-1}(\ell))= 
    \hat w_\ell^{(2)}$ for all $\ell \in [L]$ .
\end{enumerate}
Using these mappings $\psi_1$ and $\psi_2$, we construct 
the graphon $W_\psi^*(u,v) = \bfQ^*_{\psi_1(u),\psi_2(v)}$
which satisfies $\delta(W_{\bTheta^*}, W_\psi^*) = 0$. 
This leads to
\begin{align}
    \delta^2(W_{\bTheta^*},W^*) = \delta^2(W_\psi^*, 
    W^*) \leqslant \| W_\psi^* - W^* \|^2_{\mathbb L_2}
    \leqslant (B-A)^2 \lambda_2\big((u,v) : W^*_\psi(u,v) 
    \neq W^*(u,v)\big).
\end{align}
Our choice of $W_\psi^*$ ensures that $W^*(u,v) = 
W_\psi^*(u,v)$ except if $u \in \big[ (a_{k-1} 
+ \hat w_k^{(1)}) \wedge a_k, a_{k}\big)$ or $v \in 
\big[ (b_{\ell-1} + \hat w_\ell^{(2)}) \wedge b_\ell, 
b_\ell \big)$. This implies that
\begin{align}\label{dist_up}
    \delta^2(W_{\bTheta^*},W^*) \leqslant  (B-A)^2 
    \bigg( \sum_{k=1}^K ( w_k^{(1)} - \hat w_k^{(1)})_+ +
    \sum_{\ell=1}^L ( w_\ell^{(2)} - \hat w_\ell^{(2)})_+
    \bigg).
\end{align}
Since $U_i$ are i.i.d.\ random variables uniformly 
distributed in $[0,1]$, $n\hat w_k^{(1)}$ follows 
the binomial distribution $\mathcal B(n, w_k^{(1)})$. 
This implies that $\Ex[\hat w_k^{(1)}] = w_k^{(1)}$ 
for every $k$. Therefore, 
\begin{align}\label{eq5:4}
    \Ex[( w_k^{(1)} - \hat w_k^{(1)})_\pm] 
    = \frac12
    \Ex[ |w_k^{(1)} - \hat w_k^{(1)}|] \leqslant 
    \bigg(\frac{w_k^{(1)}(1-w_k^{(1)})}{4n}\bigg)^{\nicefrac12}
    \leqslant \frac{(w_k^{(1)})^{\nicefrac12}}{2\sqrt{n}}.
\end{align}
A similar upper bound can be obtained for $\Ex[( 
w_\ell^{(2)} - \hat w_\ell^{(2)})_+]$. Therefore, 
combining \eqref{dist_up}, \eqref{eq5:4}, the
Cauchy-Schwarz inequality and the fact that the 
weights $w^{(1)}$ sum to one, we get
\begin{align}
    \Ex[\delta^2(W_{\bTheta^*},W^*)] &\leqslant 
    (B-A)^2 \bigg( \sum_{k=1}^K \Ex[(w_k^{(1)} - 
    \hat w_k^{(1)})_+] + \sum_{\ell=1}^L \Ex[
    (w_\ell^{(2)} - \hat w_\ell^{(2)})_+] \bigg)\\
    &\leqslant \frac{(B-A)^2}{2}\bigg(\sqrt{\frac{K}{n}} + 
    \sqrt{\frac{L}{m}}\bigg).
\end{align}
We thus conclude the proof by taking the square root of 
the obtained inequality.

\paragraph*{Second claim (H\"older continuous graphon)}
To make the subsequent formulae more compact, we set 
$a_n(i) =  i/n$ and assume that $\mathcal L= 1$. We introduce 
$I_i = \big [a_n(i-1), a_n(i)\big[$ and 
$J_j = \big [a_m(j-1), a_m(j) \big[$. 
For every positive $k\in\mathbb N$, let $\mathfrak{S}_n$ 
be the set of the permutations of $[k]$. For $\sigma \in 
\mathfrak{S}_k$, let $\tau_\sigma$ be the specific 
measure-preserving application 
\begin{align}\label{eq5:8}
    \tau_\sigma(x) &= \sum_{i=1}^k  
    \Big(\frac{\sigma(i)-1}{k} + x - \frac{i-1}{k} 
    \Big) \mathds 1_{I_i}(x) \qquad \forall x 
    \in [0,1[,\qquad \tau_\sigma(1) = 
    \frac{\sigma(k)}{k}.
\end{align}
Notice that $\tau_\sigma$ corresponds to 
permutation of intervals $\{I_i:i\in[k]\}$ 
in accordance with $\sigma$. 

Using the definition of $\delta$, we have 
\begin{align}
    \delta^2(W_{\bTheta ^*} , W^*) &\leqslant 
    \inf_{\begin{subarray}{c} \sigma_1 \in 
    \mathfrak{S}_n \\ \sigma_2 \in \mathfrak{S}_m 
    \end{subarray}} \sum_{ i \in [n],  j \in [m]} 
    \iint_{I_i\times J_j} \big( \Theta^*_{ij} - 
    W^*(\tau_{\sigma_1}(x), \tau_{\sigma_2}(y)) 
    \big)^2 dx dy \\
    & = \inf_{\begin{subarray}{c} \sigma_1 \in 
    \mathfrak{S}_n \\ \sigma_2 \in \mathfrak{S}_m 
    \end{subarray}} \sum_{ i \in [n],  j \in [m]} 
    \iint_{I_{\sigma_1(i)}\times J_{\sigma_2(j)}}
    \big( \Theta^*_{ij} - W^*(x, y) \big)^2 dx dy \, .
    \label{eq5:7}
\end{align}
Let $\sigma_1$ be a random permutation satisfying 
$U_{(\sigma_1(i))}=U_i$; for example, let $\sigma_1$ 
be such that $U_{\sigma_1^{-1}(1)} \leqslant U_{
\sigma_1^{-1}(2)} \leqslant \dots \leqslant U_{
\sigma_1^{-1}(n)}$. We choose $\sigma_2$ similarly, 
so that $V_{(\sigma_2(j))}=V_j$. Recall that 
$\Theta^*_{ij} = W^*(U_i,V_j)$. Setting $i' = 
\sigma_1(i)$, $j' = \sigma_2(j)$ and applying 
the triangle inequality, we get
\begin{align}
    | \Theta^*_{i,j} - W^*(x,y) |  & = | W^*(x,y)  -
    W^*(U_{(i')}, V_{(j')}) |\\
    &\leqslant \big|W^*(x,y) - W^*\big( a_{n+1}(i'),
    a_{m+1}(j')\big) \big|\\
    &\qquad + \big| W^*\big(a_{n+1}(i'), 
    a_{m+1}(j') \big) - W^*(U_{(i')}, V_{(j')}) \Big|. 
    \label{triangle_decomp}
\end{align}
If $(x,y) \in I_{\sigma_1(i)} \times J_{\sigma_2(j)} = I_{i'}
\times J_{j'}$, as $\big( a_{n+1}(i'), a_{m+1}(j')
\big)$ belongs to the same set $I_{i'} \times
J_{j'}$, the Hölder property yields
\begin{align}\label{first_term}
    \Big(W^*(x,y) - W^*\big( a_{n+1} (i'), 
    a_{m+1}(j') \big) \Big)^2 \leqslant 
    \Big(\frac{1}{n^2} + \frac{1}{m^2} \Big)^{\alpha}.
\end{align}
For the second term in \eqref{triangle_decomp}, we 
use again the H\"older property, which leads to
\begin{align}
    \big\{ W^*\big( a_{n+1}(i'), a_{m+1}(j') \big) 
    - W^*(U_{(i')},V_{(j')})\big\}^2 \leqslant \big\{
    \big| a_{n+1}(i') - U_{(i')} \big|^{2} + \big|
    a_{m+1}(j') - V_{(j')} \big|^{2} \big\}^{\alpha}.
\end{align}
Then, denoting by $\sum_{i,j}$ the double sum 
$\sum_{i \in [n]}\sum_{j\in [m]}$, we have 
\begin{align}
    \Ex\Bigg[ \sum_{i,j} \iint_{I_{i'}\times J_{j'}}
    \Big( \frac{i'}{n+1} - U_{(i')} \Big)^{2} \Bigg] 
    &= \frac{1}{n} \Ex\bigg[ \sum_{i=1}^n \Big( 
    \frac{i'}{n+1} - U_{(i')} \Big)^2 \bigg] 
    = \frac{1}{n} \Ex\bigg[ \sum_{i=1}^n \Big( 
    \frac{i}{n+1} - U_{(i)} \Big)^2 \bigg] \\
    &\leqslant \max_{i=1, \dots , n} [\var(U_{(i)})]
    \leqslant \frac{1}{4n} \label{second_term}
\end{align}
where we used the fact that $U_{(i)}$ is drawn from 
the beta distribution $\beta(i,n+1-i)$. The term with 
$V_{(j')}$ is treated similarly. Combining 
\eqref{eq5:7}, the Minkowski inequality, 
\eqref{first_term} and \eqref{second_term}, we get
\begin{align}
    \Ex[ \delta^2(W_{\bTheta ^*} , W^*) ]^{\nicefrac12}
    &\leqslant  \Big(\frac{1}{n^2} + \frac{1}{m^2} 
    \Big)^{\alpha/2}\! +  \Ex\Bigg[ \sum_{i,j} \iint_{I_{i'}\times J_{j'}}
    \Big( \frac{i'}{n+1} - U_{(i')} \Big)^{2} + 
    \Big( \frac{j'}{m+1} - V_{(j')} \Big)^{2} 
    \Bigg]^{\alpha/2}\\
    &\leqslant  \Big(\frac{1}{n^2} + \frac{1}{m^2} 
    \Big)^{\alpha/2} + \Big(\frac{1}{4n} + 
    \frac{1}{4m} \Big)^{\alpha/2} \leqslant 2 
    \Big(\frac{1}{n} + \frac{1}{m} \Big)^{\alpha/2}.
\end{align}
This completes the proof. 
%\end{proof}

\subsection[Proof of \Cref{prop_aprox_error} 
(approximation error for the mean matrix)]{Proof 
of \Cref{prop_aprox_error} (approximation error for 
the matrix $\bTheta^*$)}

Without loss of generality, we prove the desired 
inequality in the case where $K = \lfloor n/n_0 
\rfloor$ and $L = \lfloor m/m_0\rfloor$. Indeed, if
the inequality is true for some value of $K$, it is
necessarily true for any smaller value as well. The 
same is true for $L$. Furthermore, we assume 
$\mathcal L = 1$.

We construct the constant-by-block matrix $\widetilde 
\bTheta \in \calT$, where $\calT$ is defined by 
\eqref{TT} as follows. Let $n=n_0K + r$ with $0 
\leqslant r < n_0$ and $m=m_0L +s$  with $0 \leqslant 
s < m_0$. For all $k \in [K]$ and 
$\ell \in [L]$, set 
\begin{align}
    I_k= [(k-1)n_0 +1 , (kn_0)\wedge n] &\qquad \text{and} 
    \qquad J_\ell= [(\ell-1)m_0 +1 , (\ell m_0)\wedge m]. 
\end{align}
The number of integers contained in 
each set of the form $I_k\times J_\ell$ is 
denoted by $n_{k,\ell}$. We set
\begin{align}
    \widetilde Q_{k,\ell} 
    = \frac{1}{n_{k,\ell}} \sum_{ i \in I_k, 
    j \in J_\ell } W^*(U_{(i)},V_{(j)}). 
\end{align}
Finally, for every $(i,j)\in[n]\times [m]$, we set 
\begin{align}
    \tilde\Theta_{i,j} = \tilde Q_{\lceil\sigma_U(i)
    /n_0\rceil,\lceil\sigma_V(j)/m_0\rceil}  
    =\sum_{k,\ell} \tilde Q_{k,\ell}\mathds 1_{\{
    \sigma_U(i)\in I_k\}} \mathds 1_{\{\sigma_V(j)
    \in J_\ell\}},
\end{align}
where $\sigma_U$ (resp.\ $\sigma_V$) is a permutation
of $[n]$ (resp.\ $[m]$) transforming $(U_i:i\in[n])$
(resp.\ $(V_j:j\in[m])$) into a nondecreasing sequence. 
In other terms, $U_{\sigma_U(i)} = U_{(i)}$ and 
$V_{\sigma_V(j)} = V_{(j)}$ for all $i$ and $j$. 
To bound the approximation error we are interested in, 
note that ($\sum_{k,\ell}$ stands for $\sum_{k\in[K]}
\sum_{\ell\in[L]}$)
\begin{align}
    \Ex\bigg[ \frac{1}{nm} \| \widetilde \bTheta - 
    \bTheta^* \|_{\sf F}^2 \bigg] &= \frac{1}{nm} 
    \sum_{k,\ell }     \sum_{i:\sigma_U(i)\in I_k}
    \sum_{j:\sigma_V(j)\in J_\ell} 
    \Ex\big[ (  \widetilde Q_{k,\ell} - \Theta^*_{ij})^2 
    \big] \\
    &= \frac{1}{nm} \sum_{k,\ell} \sum_{i \in I_k}
    \sum_{j \in J_\ell}\Ex\bigg[ \bigg( \frac{1}{
    n_{k,\ell}} \sum_{i' \in I_k} \sum_{j' \in J_\ell}  
    W^*(U_{(i')},V_{(j')}) - W^*(U_{(i)} , V_{(j)})  
    \bigg)^2 \bigg] \\
    \text{(Jensen)} \qquad  & \leqslant \frac{1}{nm} 
    \sum_{k ,\ell} \sum_{ i \in I_k} \sum_{ j \in J_\ell}   
    \bigg( \frac{1}{n_{k,\ell}} \sum_{ i' \in I_k } 
    \sum_{j' \in J_\ell } \Ex\Big[ \big( W^*(U_{(i')},
    V_{(j')}) - W^*(U_{(i)} , V_{(j)}) \big)^2 \Big]  \bigg).
\end{align}
Using the H\"older property and the Jensen inequality we 
obtain 
\begin{align}
    \Ex\Big[ \big( W^*(U_{(i')},V_{(j')}) - W^*(U_{(i)}, 
    V_{(j)}) \big)^2 \Big] & \leqslant \Ex\bigg[ 
    \bigg\| \begin{bmatrix} U_{(i')} - U_{(i)} \\ 
    V_{(j')} - V_{(j)}\end{bmatrix}\bigg\|^{2\alpha} \bigg]  
    \leqslant \bigg\{\Ex\bigg[ \bigg\| \begin{bmatrix} 
    U_{(i')} - U_{(i)} \\ V_{(j')} - V_{(j)}\end{bmatrix}
    \bigg\|^{2} \bigg]\bigg\}^\alpha .
\end{align}
Since $|i-i'|<n_0+1 \leqslant (3/2)n_0\leqslant 
(3n)/(2K)$, \citep[Lemma 4.10, p27]{Klopp2017} 
leads to
\begin{align}
    \Ex\big[| U_{(i')} - U_{(i)}|^2 \big] \leqslant 
    9/(2K)^2. 
\end{align}
Similarly, $\Ex\big[| V_{(j')} - U_{(j)}|^2 \big]
\leqslant 9/(2L)^2$. Therefore, 
\begin{align}
    \Ex\Big[ \frac{1}{nm} \| \widetilde \bTheta - 
    \bTheta^* \|_{\sf F}^2 \Big]^{\nicefrac12} & \leqslant 
    \calL \Big[ \Big(\frac{3}{2K} \Big)^{2} +
    \Big(\frac{3}{2L} \Big)^{2} \Big]^{\alpha/2} 
    \leqslant \frac{3\calL}{2}  \Big( \frac{1}{K^{\alpha}} 
    + \frac{1}{L^{\alpha}} \Big) ,
\end{align}
where we have used the inequality $(a+b)^{c} \leqslant 
a^c + b^c$ for $a,b\geqslant 0$ and for 
$c\in(0,1]$.

\subsection{Proof of \Cref{prop:polytope} (relaxation to 
a linear program)}

For further reference, we recall that we are interested in solving the problem
\begin{align}
    \min_{\bfZ\in\ZZ (n,K,n_0)} \calL(\bfZ, \bfQ, \bfZright). 
    \tag{OPT 1} \label{OPT1}
\end{align}
\paragraph*{First claim (linearization of the 
cost function)} 
We have
\begin{align}
    \calL(\bfZ,  \bfQ, \bfZright) - \|\bfH\|_{\sf F}^2
    &= -2\tr(\bfZ \bfQ (\bfZright)^\top\bfH^\top) + \tr(\bfZ 
    \bfQ (\bfZright)^\top\bfZright\bfQ^\top\bfZ^\top)\\
    &=  -2\tr(\bfZ \bfQ (\bfZright)^\top\bfH^\top) + \tr( 
    \bfQ \mathbf D\bfQ^\top\bfZ^\top\bfZ).
\end{align}
We will show that $\phi(\bfZ) = -2\tr(\bfZ \bfQ(\bfZright
)^\top \bfH^\top) + \tr(\bfQ \mathbf D\bfQ^\top\bfZ^\top 
\bfZ)$. We notice that because of the constraint on 
the rows of $\bfZ$, its columns are orthogonal, and 
$\bfZ^\top \bfZ = \textsf{diag}(n_k;k\in[K])$, where
$n_k = \bZ_{\bullet \, k}^\top {\mathbf 1}_n$ is the 
number of nonzero entries in the $k$-th column of $\bfZ$. 
So we get 
\begin{align}
    \tr(\bfQ \mathbf D\bfQ^\top\bfZ^\top\bfZ) = 
    \sum_{k=1}^K n_k \bQ_{k,\bullet}\mathbf D\bQ_{k,
    \bullet}^\top 
    = \sum_{k=1}^K \bZ_{\bullet \, k}^\top {\mathbf 1}_n\, \bQ_{k,\bullet}\mathbf D\bQ_{k,
    \bullet}^\top. 
\end{align}
This completes the proof of the first claim and entails
that \eqref{OPT1} is equivalent to 
\begin{align}
    \arg \underset{\bfZ \in \ZZ(n,K,n_0)}{\min} 
    \phi(\bfZ) \tag{OPT 2} \label{OPT2}.
\end{align} 

\paragraph*{Second claim (characterization of extreme 
points)}
An extreme point of a convex polytope $\calP$ is defined as a point in $\calP$ that can not be written as a nontrivial convex combination of two elements in $\calP$. First, let us prove that any point in $\ZZ(n,K,n_0)$ is an extreme point. Let $\bfZ \in \ZZ(n,K,n_0)$, $\lambda \in (0,1)$, $\bfZ_1$ and $\bfZ_2 \in \widetilde \ZZ(n,K,n_0)$ such that 
\begin{align}
    \bfZ &= \lambda \bfZ_1 + (1-\lambda)\bfZ_2 \, .
\end{align}
Fix some $i \in [n]$, because of the constrain on the lines of $\bfZ$, there exist $j \in [K]$ such that
\begin{align}\label{convex}
    1=Z_{ij}= \lambda (Z_1)_{ij} + (1- \lambda)(Z_2)_{ij} \,.
\end{align}
The only way to satisfy \eqref{convex} is to have $(Z_1)_{ij}= (Z_2)_{ij}=1$ because $\lambda \in (0,1)$. Then $(Z_1)_{i, \bullet} = (Z_2)_{i, \bullet} = Z_{i, \bullet} $ because of the row constraints. Finally, this holds for each $i \in [n]$, so $\bfZ=\bfZ_1=\bfZ_2$, which ensures that $\bfZ$ is an extreme point. 

Now it remains to prove that any extreme point of $\widetilde \ZZ(n,K,n_0) $ has all its entries in $\{0,1\}$. Let $\bfZ \in \widetilde \ZZ(n,K,n_0)$ which has at least one entry in $(0,1)$. Let us prove that $\bfZ$ can not be an extreme point, that is, we can write $\bfZ$ as a convex combination of two elements in $\widetilde \ZZ(n,K,n_0) $. The proof uses the next two lemmas.

\begin{lemma}\label{non_null_entries}
Let $i_0 \in [n]$ , $k_0 \in [K]$ such that $Z_{i_0k_0} \in (0,1)$. Then
\vspace{-7pt}
\begin{enumerate}\itemsep=2pt
    \item There exists $k_0' \in [K]$ such that $Z_{i_0,k_0'} 
    \in (0,1)$.
    \item Either $\sum_{i=1}^n Z_{i,k_0} \notin \mathbb N$, 
    or there exists $i_0' \in [n]$ such that $Z_{i_0' k_0} 
    \in (0,1)$.
\end{enumerate}
 \end{lemma}
\begin{proof}[Proof of \Cref{non_null_entries}]
The proof is a straightforward consequence of the row and column constraints. Indeed, as the sum of the elements of each line is an integer, if a coefficient is not 0 or 1, then there is another coefficient that lives in $(0,1)$ on the same line, which proves $(i)$. Moreover, if $\sum_{i=1}^n Z_{i,k_0} \in \mathbb N$, the same argument gives that there is another coefficient in $(0,1)$ on the column $k_0$.
\end{proof}
If we see matrix $\bfZ$ as a bi-adjacency matrix of a bipartite graph $\calG$, with weighted edges given by the entries of $\bfZ$, the next lemma formally says that either $\calG$ contains a cycle, or it has a path with extreme points that correspond to a column that sums to a number strictly greater than $n_0$ (see \Cref{fig:z_epsilon}).

\begin{lemma}\label{cycle} Let $\bfZ \in \widetilde 
\ZZ(n,K,n_0) \backslash\{0,1\}^{n \times K}$. There 
exists $T \geqslant 1$,  and two sequences $(i_t)_{t=1}^T$ and $(k_t)_{t=0}^T$ of different indices (with possibly 
$k_0=k_T$) such that 
\vspace{-7pt}
\begin{enumerate}\itemsep=2pt
     \item $Z_{i_t k_t} \in (0,1)$ for all $1 \leqslant t \leqslant T$.
     \item $Z_{i_{t+1} k_t} \in (0,1)$ for all $0 \leqslant t \leqslant T-1$.
     \item Either $\sum_{i=1}^n Z_{i k_t} >n_0$ for $t \in \{0,T\}$ (we say that $\bfZ$ has a dead end path), or $k_T=k_0$ ($\bfZ$ has a cycle).
\end{enumerate}
\end{lemma}
\begin{proof}[Proof of \Cref{cycle}] Let us first assume that all columns of $\bfZ$ sum to some integers. We denote $Z_{i_1 k_0}$ one element of $\bfZ$ that is in $(0,1)$ According to part $(i)$ of \Cref{non_null_entries}, there exists $k_1 \neq k_0$ such that $Z_{i_1 k_1} \in (0,1)$. Following the same \Cref{non_null_entries}, we can find $i_2 \neq i_1$, and $k_2 \neq k_1$ such that $Z_{i_2 k_1} \in (0,1)$ and $Z_{i_2 k_2} \in (0,1)$ (so $T \geqslant 2$). We iterate the same process until iteration $T$, with $T$ define as the first time at which we have $i_T=i_{t_0}$ or $k_T=k_{t_0}$ for some $0 \leqslant t_0 \leqslant T-1$, meaning that we met a row or a column we already had in the previous iterations. Notice that $t_0 \leqslant T-2$ because we always have $i_t \neq i_{t+1}$ and $k_t \neq k_{t+1}$. Then we consider the following shifted sequences.
\begin{itemize}
    \item $i^\#=(i_{t+t_0-1})_{t=1}^{T-t_0}$ and $k^\#_t=k_{t+t_0}$ for $t \in \{0, \dots , T-t_0-1\}$ and $k^\#_{T-t_0}=k_{t_0}$ in the case where $i_T=i_{t_0}$, meaning that we first met a row we already had in the previous iterations. In this case, $t_0\geqslant1$.
    \item $i^\#=(i_{t+t_0})_{t=1}^{T-t_0}$ and $k^\#=(k_{t+t_0})_{t=0}^{T-t_0}$ in the case where $k_T=k_{t_0}$ meaning that we first met a column we already had in the previous iterations.
\end{itemize}
 These sequences $i^\#$ and $k^\#$ satisfy the cycle conditions of the lemma by construction.

Now we assume that there is a column $k_0$ which has a sum not in $\mathbb N$. Then it has a coordinate $Z_{i_1 k_0} \in (0,1)$. Applying part $(i)$ of \Cref{non_null_entries} gives $k_1 \neq k_0$ such that $Z_{i_1 k_1} \in (0,1)$. 
\begin{itemize}
    \item If $\sum_{i=1}^n Z_{i k_1} > n_0$, then lemma is proven for $T=1$ and the dead end path setting.
    \item Else, $\sum_{i=1}^n Z_{i k_1} = n_0$ then we can iterate similarly the previous process until iteration $T$ define as the first time at which we have $i_T=i_{t_0}$ or $k_T=k_{t_0}$ for some $0 \leqslant t_0 \leqslant T-2$ or $\sum_{i=1}^n Z_{i k_T} > n_0$. In the first two cases, we consider the shifted sequences $i^\#$ and $k^\#$ as before, that satisfy the conditions of the lemma. In the last case, the sequences $(i_t)_{t=1}^T$ and $(k_t)_{t=0}^T$ satisfy the desired dead end path conditions.
\end{itemize}
Notice that $T < +\infty $ because the number of rows and columns is finite.
\end{proof}
The end of the proof of \Cref{prop:polytope} is  
similar to the one in \citep[Prop 3.4]{transport}, 
if we see $\bfZ$ as the bi-adjacency matrix of a
bipartite graph. Let us rewrite the proof adapted 
for our purpose. We introduce
\begin{align}
    \varepsilon_1 &=  \underset{k \in [K]}{\min} 
    \Big\{ \sum_{i=1}^n \frac{Z_{i,k} - n_0}{2} :
    \sum_{i=1}^n Z_{i,k}>n_0 \Big\} \wedge 1 ,\quad 
    \varepsilon_2 = \min_{ i , k } 
    \Big\{ \frac{Z_{ik}}{2} : Z_{ik} > 0 \Big\} 
\end{align}
and $\varepsilon_3 = (1/2)\min_{ i , k } \big\{ 1 -
Z_{ik} : Z_{ik} < 1 \big\}$. 
So $\varepsilon_1$, $\varepsilon_2$ and $\varepsilon_3$ are positive real numbers. By convention, the minimum of the empty set is $+\infty$. We finally define $\varepsilon=\min(\varepsilon_1, \varepsilon_2,\varepsilon_3) \in (0,1)$. Hence, when we add or subtrack $\varepsilon$ to one entry of $\bfZ$ that is in $(0,1)$, it remains in $(0,1)$. Moreover, if this entry is in a column that sums to strictly more than $n_0$, if we subtract $\varepsilon$ from this entry, the sum of the column remains strictly greater than $n_0$. We apply \Cref{cycle} which gives to sequences $(i_t)_{t=1}^T$ and $(k_t)_{t=0}^T$ that satisfy the conditions of the lemma. Then we define $\bfZ_{\varepsilon}$ such that (see \Cref{fig:z_epsilon} for a visual construction of $\bfZ_\varepsilon$)
\begin{align} (Z_{\varepsilon})_{ij} =
\begin{cases}
0  & \text{ if } \qquad i\neq i_t  \text{ or } k \neq k_t \\
+ \varepsilon  &  \text{ if } \qquad i=i_{t+1}, k=k_{t} \qquad \text{ for some } 0 \leqslant t \leqslant T-1 \\
- \varepsilon & \text{ if } \qquad i=i_{t},k=k_{t} \qquad \text{ for some } 1 \leqslant t \leqslant T \, .
\end{cases}    
\end{align}

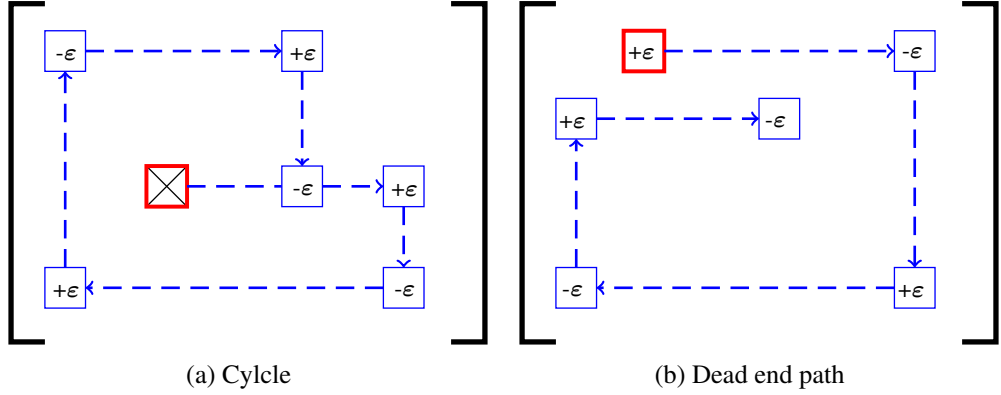
\begin{figure}%[h]
    \centering
    \begin{subfigure}[b]{0.43\textwidth}
    \begin{tikzpicture}[scale=0.9]
    % Braket
    \draw[line width=2pt] (2, 0) -- (1.5, 0) -- (1.5, 5) -- (2, 5)  ;
    \draw[line width=2pt] (8, 0) -- (8.5, 0) -- (8.5, 5) -- (8, 5)  ;

    % First square
    \draw[line width=0.5pt] (3.5, 2) -- (4.1, 2.6);
    \draw[line width=0.5pt] (3.5, 2.6) -- (4.1, 2);
    \draw[line width=1.5pt,red] (3.5, 2) rectangle (4.1, 2.6);

    % Squares
    \draw[line width=0.5pt,blue] (7, 2) rectangle (7.6, 2.6);
    \draw[line width=0.5pt,blue] (7, 0.5) rectangle (7.6, 1.1);
    \draw[line width=0.5pt,blue] (2, 0.5) rectangle (2.6, 1.1);
    \draw[line width=0.5pt,blue] (2, 4) rectangle (2.6, 4.6);
    \draw[line width=0.5pt,blue] (5.5, 4) rectangle (6.1, 4.6);
    \draw[line width=0.5pt,blue] (5.5, 2) rectangle (6.1, 2.6);

    % Arrows
    \draw[dash pattern=on 7pt off 4pt,line width=1pt,blue] (4.1,2.3) -- (5.5,2.3);
    \draw[->,dash pattern=on 7pt off 4pt,line width=1pt,blue] (6.1,2.3) -- (7,2.3);
    \draw[->,dash pattern=on 7pt off 4pt,line width=1pt,blue] (7.3,2) -- (7.3,1.1);
    \draw[->,dash pattern=on 7pt off 4pt,line width=1pt,blue] (7,0.8) -- (2.6,0.8);
    \draw[->,dash pattern=on 7pt off 4pt,line width=1pt,blue] (2.3,1.1) -- (2.3,4);
    \draw[->,dash pattern=on 7pt off 4pt,line width=1pt,blue] (2.6,4.3) -- (5.5,4.3);
    \draw[->,dash pattern=on 7pt off 4pt,line width=1pt,blue] (5.8,4) -- (5.8,2.6);

    % Sqaure labels
    \node at (7.32,2.25) {+$\boldsymbol{\varepsilon}$};
    \node at (7.32,0.75) {-$\boldsymbol{\varepsilon}$};
    \node at (2.32,0.75) {+$\boldsymbol{\varepsilon}$};
    \node at (2.32,4.25) {-$\boldsymbol{\varepsilon}$};    
    \node at (5.82,4.25) {+$\boldsymbol{\varepsilon}$};
    \node at (5.82,2.25) {-$\boldsymbol{\varepsilon}$};    
    \end{tikzpicture}
         \caption{Cylcle}
    \end{subfigure}
    \hspace{0.5cm}
    \begin{subfigure}[b]{0.43\textwidth}
    \begin{tikzpicture}[scale=0.9]
    % Braket
    \draw[line width=2pt] (2, 0) -- (1.5, 0) -- (1.5, 5) -- (2, 5)  ;
    \draw[line width=2pt] (8, 0) -- (8.5, 0) -- (8.5, 5) -- (8, 5)  ;

    % Squares
    \draw[line width=1.5pt,red] (3, 4) rectangle (3.6, 4.6);
    \draw[line width=0.5pt,blue] (7, 4) rectangle (7.6, 4.6);
    \draw[line width=0.5pt,blue] (7, 0.5) rectangle (7.6, 1.1);
    \draw[line width=0.5pt,blue] (2, 0.5) rectangle (2.6, 1.1);
    \draw[line width=0.5pt,blue] (2, 3) rectangle (2.6, 3.6);
    \draw[line width=0.5pt,blue] (5, 3) rectangle (5.6, 3.6);

    % Arrows
    \draw[->,dash pattern=on 7pt off 4pt,line width=1pt, blue] (3.6,4.3) -- (7,4.3);
    \draw[->,dash pattern=on 7pt off 4pt,line width=1pt, blue] (7.3,4) -- (7.3,1.1);
    \draw[->,dash pattern=on 7pt off 4pt,line width=1pt, blue] (7,0.8) -- (2.6,0.8);
    \draw[->,dash pattern=on 7pt off 4pt,line width=1pt, blue] (2.3,1.1) -- (2.3,3);
    \draw[->,dash pattern=on 7pt off 4pt,line width=1pt, blue] (2.6,3.3) -- (5,3.3);

    % Sqaure labels
    \node at (3.27,4.27) {+$\boldsymbol{\varepsilon}$};
    \node at (7.27,4.27) {-$\boldsymbol{\varepsilon}$};
    \node at (7.255,0.755) {+$\boldsymbol{\varepsilon}$};
    \node at (2.25,0.75) {-$\boldsymbol{\varepsilon}$};
    \node at (2.25,3.25) {+$\boldsymbol{\varepsilon}$};    
    \node at (5.25,3.25) {-$\boldsymbol{\varepsilon}$};
    \end{tikzpicture}
         \caption{Dead end path}
    \end{subfigure}
    % \begin{tabular}{cc}
    %   \includegraphics[scale=0.2]{Matrix.png} & \hspace{0.6cm} \includegraphics[scale=0.2]{Matrix2.png} \\
    %   Cycle & Dead end path
    % \end{tabular}
    \caption{
    The starting point is depicted as a red square. In the left-hand side illustration, we encounter a row identical to one we previously encountered, akin to the procedure outlined in \Cref{cycle}. Subsequently, we disregard the preceding steps, symbolized by the crossed-out square. This process generates the sequences necessary for the construction of $\mathbf{Z}_\varepsilon$. On the right-hand side illustration, we have not encountered any previously encountered lines or columns. Nonetheless, the procedure halts due to the last column summing to a value exceeding $n_0$. This occurrence provides the pathway for constructing $\mathbf{Z}_\varepsilon$.}
    \label{fig:z_epsilon}
\end{figure}
The properties of the sequences imply that
\vspace{-8pt}
\begin{itemize}\itemsep=1pt
    \item $\bfZ_\varepsilon {\mathbf 1}_K = {\bf0}_n$.
    Indeed, the $i$-th row of $\bfZ$ sums to 0 when 
    $i \notin \{i_1, \dots , i_T\}$. But if $i=i_t$ 
    for some $t \in [T]$, then there is a 
    $+\varepsilon$ on the $k_t$-th column, and a 
    $-\varepsilon$ on the $k_{t-1}$-th column, and 
    0 anywhere else.
    
    \item Moreover we have $\bfZ_\varepsilon^\top 
    {\mathbf 1}_n = \delta_\varepsilon {\mathbf 1}_K$ 
    where $\delta_\varepsilon \in \{0,\pm\varepsilon
    \}$. Indeed, the $k$-th column sums 
    to 0 for $k \notin \{k_0, \dots , k_T\}$. For 
    $k=k_t$ with $ t \in[T-1]$, 
    then there is a $+\varepsilon$ on the $i_{t+1}$-th 
    row and a $-\varepsilon$ on the $i_{t}$-th row and 
    0 anywhere else. For $t \in \{0,T\}$ we have two 
    cases. If $k_T=k_0$, then  the $k_0$-th column 
    sums to 0. If $k_T\neq k_0$, there is only 
    a $+\varepsilon$ on the $k_0$-th column and 
    only a $-\varepsilon$ on the $k_T$-th column.
\end{itemize}

The choice of $\varepsilon$  ensures that $\bfZ_+=\bfZ+\bfZ_\varepsilon$ and $\bfZ_-=\bfZ-\bfZ_\varepsilon$ live in $\widetilde \ZZ(n,K,n_0)$. Finally, $\dis \bfZ = (\bfZ_+ + \bfZ_-)/2$ and 
$\bfZ$ can not be an extreme point.

\newpage
\section{Proof of \Cref{lower_bound} (lower bounds)}
%\secttoc

Although the general structure of the proof and some 
important parts of it are similar to those of the proof 
of \citep[Proposition 3.4]{Klopp2017}, there are some 
technical differences that are due to the fact that 
the graphon and the observed matrix are not symmetric. 
Furthermore, the bounds involve some terms depending on 
$L$ and $m$, which was not the case in lower bounds 
proved in the literature. 

To get the desired lower bound, we divide the problem 
into the following three minimax lower bounds
\begin{align}
    \inf_{\widehat W} \sup_{ W^* \in \calW_\rho[K,L] } \Ex_{W^*} \big[ \delta^2(\widehat W , W^*) \big] &\geqslant c \rho^2 \Big( \sqrt{\frac{K}{n}}  +\sqrt{\frac{L}{m}} \Big) \label{minimax1} \\
    \inf_{\widehat W} \sup_{ W^* \in \calW_\rho[K,L] } \Ex_{W^*} \big[ \delta^2(\widehat W , W^*) \big] &\geqslant c\rho  \Big(\frac{KL}{Nnm} \wedge \rho  
    \Big) \label{minimax2} \\
    \inf_{\widehat W} \sup_{ W^* \in \calW_\rho[2,2]} \Ex_{W^*} \big[ \delta^2(\widehat W , W^*) \big] &\geqslant c\rho  \Big(\frac1{N\sqrt{nm}}\wedge \rho  \Big) \label{minimax3}
\end{align}
If these three inequalities hold true, then the desired
result will be true with the constant $c/3$. The rest
of this section is split into three subsections, each 
of which contains the proof of one of the inequalities
\eqref{minimax1}, \eqref{minimax2} and \eqref{minimax3}. 

\subsection{Proof of \eqref{minimax1}: error due
to the unknown partition}

Since $(K,n)$ and $(L,m)$ play symmetric roles, we will  only prove the lower bound $\rho^2 \sqrt{{K}/{n}}$. The 
same arguments will lead to the lower bound $\rho^2 \sqrt{{L}/{m}}$. As a consequence, we will get the
lower bound $\rho^2(\sqrt{{K}/{n}}\vee \sqrt{{L}/{m}})
\geqslant \rho^2(\sqrt{{K}/{n}} + \sqrt{{L}/{m}})/2$.

Without loss of generality,  we can assume that $K$ 
is a multiple of 16. Indeed, by choosing $C>17$, for
any $K>17$, setting $K'= 16\lfloor K/16\rfloor$, the 
inequality
\begin{align}
    \inf_{\widehat W} \sup_{W^* \in \calW_\rho
    [K',L]} \Ex_{W^*} \big[ \delta^2(\widehat W , W^*)
    \big]\geqslant c \rho^2 \sqrt{\frac{K'}{n}} 
\end{align}
would imply
\begin{align}
    \inf_{\widehat W} \sup_{ W^* \in \calW_\rho[K,L] } \Ex_{W^*} \big[ \delta^2(\widehat W , W^*) \big] & \geqslant  \inf_{\widehat W} \sup_{W^* \in \calW_\rho
    [K',L]} \Ex_{W^*} \big[ \delta^2(\widehat W , W^*)
    \big] 
    \geqslant \frac{c}{4} \rho^2 \sqrt{\frac{K}{n}}.
\end{align}
To establish the desired lower bound, we follow the standard recipe \citep[Theorem 2.7]{Nonparametric} 
consisting in designing a finite set 
of graphons that has the following two properties: 
the graphons from this set are well separated when 
the distance is measured by the metric $\delta^2$ 
and, in the same time, the distributions generated 
by these graphons are close, which makes it difficult to differentiate them based on the data matrix $\bfH$. 

To define this set, we choose a $K\times L$ matrix 
$\bfQ$ with entries from $\{0,\rho\}$ and two positive numbers $\varepsilon,\varepsilon'$; conditions on
$\bfQ$, $\varepsilon$ and $\varepsilon'$ will be
specified later. For any $K\in\mathbb N$, define
\begin{align}
    \calC^K=\Big\{ \bfs \in \{-1, + 1\}^K: 
    \sum_{k=1}^K s_k =0  \Big\},\quad 
    I_{\bfs,\varepsilon}^a = \bigg[\sum_{k=1}^{a-1}
    \Big(\frac1K + \varepsilon s_k\Big), \sum_{k=1
    }^{a} \Big(\frac1K + \varepsilon s_k\Big)\bigg)
\end{align}
for $a \in [K]$ and $\bfs \in \calC^K$, with the 
convention that $\sum_{k=1}^0 = 0$. Similarly, for
every $\bft \in \calC^L$,  we set $J_{\bft, 
\varepsilon'}^b = \big[\sum_{\ell=1}^{b-1} 
\big(\frac1L + \varepsilon' t_\ell\big), \sum_{\ell=1}^b \big(\frac1L + \varepsilon' 
t_\ell\big)\big)$. The length of an interval $I$ will be
denoted by $|I|$. We can now define the class of 
graphons $W_{\bfs, \bft, \varepsilon}$ for each $\bfs \in \calC^K$ and $\bft \in\calC^L$ by
\begin{align}
    W_{\bfs,\bft,\varepsilon}(u,v)= \sum_{k=1}^K
    \sum_{\ell=1}^L
    Q_{k,\ell} \mathds 1_{I_{\bfs,\varepsilon}^k} 
    (u)\mathds 1_{J_{\bft,\varepsilon'}^\ell} (v). 
\end{align}
We denote by $\bfP_{W_{\bfs,\bft,\varepsilon}}$ the 
distribution of $\bfH= (H_{ij}, i \in [n], j\in [m])$, 
where $\bfH$ is sampled according to the Binomial model with graphon $W_{\bfs,\bft,\varepsilon}$. The next three
lemmas, the proofs of which are postponed to the end
of this subsection, provide the main technical tools
necessary to establish the desired lower bound.

\begin{lemma}\label{lem:lowerKL}
If\/ $4K\varepsilon \leqslant 1$, then for all 
$\bfs$ and $\bfs'$ from $\calC^K$ and all $\bft 
\in\calC^L$, the following inequality $\KL( 
\bfP_{W_{\bfs, \bft,\varepsilon}} || \bfP_{
W_{\bfs', \bft,\varepsilon}}) \leqslant 6 n
(K\varepsilon)^2$ holds true.
\end{lemma}

\begin{lemma}\label{matrix_B} Assume that $K,L$ are 
large enough integers multiple of 16 and satisfying 
$KL\geqslant L\log^2 L + K\log^2 K$. There exists 
$\bfB \in \{ -1,1 \}^{K \times L}$ satisfying the 
following two properties.
\vspace{-8pt}
\begin{enumerate}\itemsep=2pt
    \item[i)] For all $(k_1,k_2) \in [K]^2$, $k_1 \neq k_2$, 
    and for all $(\ell_1,\ell_2) \in [L]^2$, $\ell_1 
    \neq \ell_2$, it holds that
    \begin{align}
        |\langle \bfB_{k_1 \, \bullet} , \bfB_{k_2 \, \bullet} \rangle| &\leqslant L/4 \qquad
        \text{and}\qquad
        |\langle \bfB_{\bullet \, \ell_1 } , \bfB_{ \bullet \, \ell_2} \rangle | \leqslant K/4
    \end{align}
    \item[ii)]  Let $\pi_i : [K/16] \to [K]$ 
    and $\nu_i : [L/16] \to [L]$, $i=1,2$ 
    be arbitrary bijections such that either 
    $\Img(\pi_1)\cap \Img(\pi_2) = \varnothing$ 
    or $\Img(\nu_1)\cap \Img(\nu_2) = \varnothing$.
    Then
    \begin{align}
        \sum_{k=1}^{K/16}\sum_{\ell=1}^{L/16} 
        \big( B_{\pi_1(k), \nu_1(\ell)} - B_{\pi_2(k),
        \nu_2(\ell)} \big)^2 \geqslant \frac{KL}{512}.
    \end{align}
\end{enumerate}
\end{lemma}

For each $\bfs\in\calC^K$, let us define
$\calA_{\bfs}= \{ k \in [K]: s_k = 1 \}$. Notice 
that $|\calA_{\bfs}| = K/2$.

\begin{lemma}\label{lem:lower2}
Choose $\bfQ = \bfB + \mathbf 1_K\mathbf 
1_L{}\!\!\!^\top$, 
where $\bfB$ is given by \Cref{matrix_B}. 
Let $\bfs$ and $\bfs'$ be two distinct vectors from
$\calC^K$ such that $|\calA_{\bfs} \Delta \calA_{\bfs'}| 
\geqslant {K}/{4}$. For $\varepsilon'= 1/(4L)$, 
for any $K\in\mathbb N$, $\varepsilon \in [0,1/(4K)]$
and $\bft\in\calC^L$, we have 
\begin{align}
    \delta^2(W_{\bfs,\bft,\varepsilon}, W_{\bfs',\bft,\varepsilon})
    \geqslant \frac{\varepsilon K\rho^2}{512}.
\end{align}
\end{lemma}

Lemma 4.4 in \citep{Klopp2017} implies that there 
exists a subset $\widetilde \calC^K \subset \calC^K$
such that $\log |\widetilde \calC^K| \geqslant 
{K}/{16}$ and $|\calA_{\bfs} \triangle 
\calA_{\bfs'}| \geqslant {K}/{4}$ for any $\bfs 
\neq \bfs'$ from $\widetilde \calC^K$. We consider
the set $\{W_{\bfs,\bft,\varepsilon} : \bfs \in 
\tilde\calC^K\}$ for a fixed $\bft\in\calC^L$ and 
for $\varepsilon^{-1} = 2^{4}\sqrt{6nK}$. 
In view of \Cref{lem:lowerKL} and \Cref{lem:lower2}, 
for any  $\bfs,\bfs' \in \widetilde\calC^K$ such
that $\bfs\neq\bfs'$, we have
\begin{align}
    \delta^2(W_{\bfs,\varepsilon},  W_{\bfs',\varepsilon}) &\geqslant \frac{\rho^2}{2^{13}} \sqrt{\frac{K}{6n}}  \quad
    \text{and}\quad
    \KL(\bfP_{W_{\bfs,\varepsilon}} ||\, \bfP_{W_{\bfs',\varepsilon}})      \leqslant \frac{1}{16^2}K  \leqslant \frac{1}{16} \log |\widetilde\calC^K|.
\end{align}
Therefore, we can apply \citep[Theorem 2.7]{Nonparametric} to get
\begin{align}
    \inf_{\widehat W} \sup_{ W^* \in \calW_\rho[K,L] } \Ex_{W^*} \big[ \delta^2(\widehat W , W^*) \big] \geqslant c \rho^2  \sqrt{\frac{K}{n}} \, 
\end{align}
for some universal constant $c>0$. This completes the
proof of \eqref{minimax1}, modulo the proofs of 
three technical lemmas appended below. 

\begin{proof}[Proof of \Cref{lem:lowerKL}]
One can check that, for every matrix $\bfA\in[N]^{n\times m}$, and for every $\bfs\in \calC^K$,   
\begin{align}
    \bfP_{W_{\bfs,\bft,\varepsilon}} (\bfH = \bfA) 
    &= \prod_{i=1}^n\prod_{j=1}^m \int_{[0,1]^2} \binom{N}{A_{ij}}
    W_{\bfs,\bft,\varepsilon}(u_i,v_j)^{A_{ij}}
    (1-W_{\bfs,\bft,\varepsilon} (u_i,v_j) )^{N-A_{ij}}
    \,du_i\,dv_j.
\end{align}
Using the fact that $W_{\bfs,\bft,\varepsilon}$ is 
piecewise constant, we get
\begin{align}
    \bfP_{W_{\bfs,\bft,\varepsilon}} (\bfH = \bfA) 
    & = \prod_{i,j} \sum_{k=1}^K\sum_{\ell=1}^L 
    \int_{I_{ \bfs,\varepsilon}^k}  \int_{J_{\bft, 
    \varepsilon'}^\ell} \binom{N}{A_{ij}} Q_{k\ell}^{A_{i,j}}(1 - 
    Q_{k\ell})^{N-A_{i,j}}\,du_i\,dv_j\\
    & = \prod_{i,j} \sum_{k,\ell} \binom{N}{A_{ij}} Q_{k\ell}^{A_{i,j}}(1 - 
    Q_{k\ell})^{N-A_{i,j}} \big|I_{\bfs, 
    \varepsilon}^k\big|\cdot \big| J_{\bft, 
    \varepsilon'}^\ell\big|\\
    & = \sum_{\bfK,\bfL} \prod_{i,j} 
    \binom{N}{A_{ij}} Q_{K_{ij}L_{ij}}^{A_{ij}} 
    (1-Q_{K_{ij}L_{ij}})^{N - A_{ij}}
    \big| I_{\bfs,\varepsilon}^{K_{ij}} \big|\cdot
    \big|J_{\bft, \varepsilon'}^{L_{ij}} \big|,
\end{align}
where the outer sum of the last line is
over all matrices $\bfK$ and $\bfL$ having
entries respectively in $[K]$ and in $[L]$. 
Let us define
\begin{align}
    \Psi(\bfK,\bfL,\bfA) &= \prod_{i,j} \binom{N}{A_{ij}} Q_{K_{ij}
    L_{ij}}^{A_{ij}} (1-Q_{K_{ij}L_{ij}})^{N-A_{ij}}, 
    \quad w_{\bfs,\bft, \varepsilon}(\bfK,\bfL) =
    \prod_{i,j}  \big| I_{\bfs,\varepsilon}^{K_{ij}} 
    \big|\cdot \big|J_{\bft, \varepsilon'}^{L_{ij}} 
    \big|. 
\end{align} 
The computations above imply that 
\begin{align}
    \KL(\bfP_{W_{\bfs,\bft,\varepsilon}} ||\, \bfP_{W_{\bfs',\bft,\varepsilon}}) 
    &= \sum_{\bfA \in \{0,1\}^{n\times m}} 
    \bfP_{W_{\bfs,\bft,\varepsilon}} (\bfH = \bfA) \,\log\Big(
    \frac{\bfP_{W_{\bfs,\bft,\varepsilon}} (\bfH = \bfA)}{
    \bfP_{W_{\bfs',\bft,\varepsilon}} (\bfH = \bfA)}\Big)\\ 
    &= \sum_{\bfA,\bfK,\bfL} \Psi(\bfK,\bfL,\bfA) w_{\bfs,\bft,\varepsilon}
    (\bfK,\bfL) \log\bigg(\frac{\sum_{\bfK',\bfL'}
    \Psi(\bfK',\bfL',\bfA) w_{\bfs,\bft,\varepsilon}
    (\bfK',\bfL')}{\sum_{\bfK',\bfL'} \Psi(\bfK',
    \bfL',\bfA) w_{\bfs',\bft,\varepsilon}(\bfK',\bfL')}\bigg).
\end{align}
It is clear that $w_{\bfs,\bft,\varepsilon}(\bfK,\bfL) 
\geqslant 0$ and  $\sum_{\bfK,\bfL} w_{\bfs,\bft, 
\varepsilon}(\bfK,\bfL) = 1 $.  Since the function
$(x,y) \mapsto x \log(x/y)$ is convex, we apply the 
Jensen inequality to get
\begin{align}
    \KL(\bfP_{W_{\bfs,\bft\varepsilon}} || \,\bfP_{W_{\bfs',\bft,\varepsilon}}) 
    &\leqslant  \sum_{\bfA,\bfK,\bfL} \Psi(\bfK,
    \bfL,\bfA) w_{\bfs,\bft,\varepsilon} (\bfK,\bfL)
    \log\bigg(\frac{\Psi(\bfK,\bfL, \bfA)
    w_{\bfs,\bft,\varepsilon}(\bfK,\bfL)}{
    \Psi(\bfK, \bfL,\bfA) w_{\bfs',\bft,\varepsilon}
    (\bfK,\bfL)}\bigg)\\
    &= \sum_{\bK,\bfL} w_{\bfs,\bft,\varepsilon} 
    (\bfK,\bfL) \log\bigg(\frac{
    w_{\bfs,\bft,\varepsilon}(\bfK,\bfL)}{w_{\bfs',\bft,\varepsilon}(\bfK,\bfL)}\bigg).
\end{align}
The last expression can be seen as the 
Kullback-Leibler divergence between two product
distributions on $[K]^{n\times m}\times[L]^{n\times 
m}$. Since the Kullback-Leibler divergence between
product distributions is the sum of Kullback-Leibler
divergences, we get
\begin{align}
    \sum_{\bfK,\bfL} w_{\bfs,\bft,\varepsilon} 
    (\bfK,\bfL) \log\bigg(\frac{
    w_{\bfs, \bft, \varepsilon}(\bfK,\bfL)}{w_{\bfs',\bft,\varepsilon}(\bfK,\bfL)}\bigg) 
    &= n\sum_{k=1}^K \big|I_{\bfs,\varepsilon}^k\big|
    \log\bigg(\frac{|I_{\bfs,\varepsilon}^k|}{|I_{\bfs',\varepsilon}^k|}\bigg)
    \leqslant  n \sum_{k=1}^K \frac{\big( 
    |I_{\bfs,\varepsilon}^k|- |I_{\bfs',\varepsilon
    }^k|\big)^2}{|I_{\bfs',\varepsilon}^k|},
\end{align}
where the last inequality follows from the fact that
the Kullback-Leibler divergence does not exceed the 
chi-square divergence. Since $|I_{\bfs',\varepsilon
}^k| = (1/K) + \varepsilon s'_k \geqslant (1/K) -
\varepsilon \geqslant 3/(4K)$, we get 
\begin{align}
    \sum_{\bfK,\bfL} w_{\bfs,\bft,\varepsilon} 
    (\bfK,\bfL) \log\bigg(\frac{
    w_{\bfs,\bft,\varepsilon}(\bfK,\bfL)}{w_{\bfs',\bft,\varepsilon}(\bfK,\bfL)}
    \bigg) &\leqslant
    n \sum_{k=1}^K \frac{(\varepsilon s_k - 
    \varepsilon s'_k\big)^2}{3/(4K)}\leqslant 6 nK^2
    \varepsilon^2.
\end{align}
This completes the proof of the lemma.
\end{proof}

\begin{proof}[Proof of \Cref{matrix_B}] 
Let  $\bXi$ be a $K\times L$ random matrix with iid 
Rademacher entries $\xi_{k,\ell}$, \textit{i.e}, $\bfP(
\xi_{k,\ell} = \pm1) = 1/2$. Then $\langle \bXi_{
k_1 \, \bullet} , \bXi_{k_2 \, \bullet} \rangle = \sum_{\ell=1}^L \xi_{k_1, \ell}\xi_{k_2,\ell}$.  
By the Hoeffding inequality
 \begin{align}
     \bfP ( | \langle \bXi_{k_1 \, \bullet} , 
     \bXi_{k_2 \, \bullet} \rangle | \geqslant L/4 ) \leqslant 2 e^{-L/32}\, .
 \end{align}
By the union bound, we obtain that for all 
$k_1 \neq k_2 \in [K], | \, \langle \bXi_{k_1 \, \bullet} , \bXi_{k_2 \, \bullet} \rangle | \leqslant L/4$ with probability larger than $1-2K^2 e^{-L/32}$, which is larger than $3/4$ for $L \geqslant 480$.  Similarly, one checks
that 
\begin{align}
    \bfP(\max_{\ell_1\neq\ell_2}| \langle \bXi_{\bullet,
    \ell_1} , \bXi_{\bullet,\ell_2} \rangle | 
    \leqslant K/4 ) > 3/4.
\end{align}
Thus, we get 
\begin{align}\label{bxi1}
    \bfP\Big(\max_{\ell_1\neq\ell_2}| \langle 
    \bXi_{\bullet,\ell_1} , \bXi_{\bullet,\ell_2} 
    \rangle | \leqslant K/4 \text{ and } 
    \max_{k_1\neq k_2}| \langle \bXi_{k_1,\bullet}, \bXi_{k_2,\bullet} \rangle | 
    \leqslant L/4\Big) > 1/2.
\end{align}
For the second property stated in the lemma, we fix some $\calX_i,\calY_i$ and $\pi_i,\nu_i$ as in the statement 
and define
\begin{align}
    T(\pi_{1:2},\nu_{1:2},\bXi) = \frac14 \sum_{k=1}^{K/16} 
    \sum_{\ell=1}^{L/16} 
    \big( \xi_{\pi_1(k), \nu_1(\ell)} - \xi_{
    \pi_2(k),\nu_2(\ell)} \big)^2 .
\end{align}
Clearly, $T[\pi_{1:2},\nu_{1:2},\bXi]$ is a sum of 
$KL/2^8$ i.i.d Bernouilli random variables with 
parameter 1/2. Applying again the Hoeffding inequality, 
we have
\begin{align}
    \bfP \Big(T[\pi_{1:2},\nu_{1:2},\bXi] \leqslant 
    \frac{KL}{2^{11}} \Big) &= \bfP \Big(
    \frac{KL}{2^9}-T \geqslant \frac{3KL}{2^{11}} 
    \Big) \leqslant e^{-{9KL}/{2^{13}}} \, .
\end{align}
There are no more than $(K/16)!^2(L/16)!^2$ functions $\pi_1,\pi_2,\nu_1,\nu_2$ satisfying conditions of ii).  Therefore, the union bound implies that
with probability at least  $1-(K/16)!^2(L/16)!^2 e^{-
9KL/2^{13}}$, we have $T[\pi_{1:2},\nu_{1:2},\bXi] 
\geqslant KL/2^{11}$ for all $\pi_1,\pi_2,\nu_1,\nu_2$. 
Choosing $K$ and $L$ large enough, and using the 
condition $KL\geqslant K\log^2 K + L\log^2 L$, 
we get that 
\begin{align}\label{bxi2}
    \bfP\bigg(\min_{\pi_1,\pi_2,\nu_1,\nu_2}
    \sum_{k=1}^{K/16}\sum_{\ell=1}^{L/16} 
    \big( \xi_{\pi_1(k), \nu_1(\ell)} - 
    \xi_{\pi_2(k),c\nu_2(\ell)} \big)^2 
    \geqslant \frac{KL}{512} \bigg) >1/2.
\end{align}
Combining \eqref{bxi1} and \eqref{bxi2}, we get that
the probability that the random matrix $\bXi$ satisfies
properties i) and ii) is strictly positive. This implies
that the set of such matrices is not empty. 
\end{proof}

\begin{proof}[Proof of \Cref{lem:lower2}]
Without loss of generality, throughout this proof, 
we assume that $\rho=1$. Furthermore, since $\bft$ is
fixed, we will often drop it in the notation and 
write $W_{\bfs,\varepsilon}$ instead of $W_{\bfs,\bft,\varepsilon}$. 

It suffices to prove that 
for all measure preserving bijections $\tau_1 
:[0,1] \to [0,1]$ and $\tau_2:[0,1]\to[0,1]$,  
\begin{align}
    \| W_{\bfs,\varepsilon} - W_{\bfs',
    \varepsilon} \circ (\tau_1 \otimes \tau_2) \|^2_{\mathbb L_2} \geqslant \frac{\varepsilon K}{512}.
\end{align}
If $u \in I_{\bfs,\varepsilon}^k$ and $u' \in 
I_{\bfs,\varepsilon}^{k'}$ for some $k,k' \in [K]$, 
then
\begin{align}
    \bigg| \int_0^1 (W_{\bfs,\varepsilon}(u,v) - 1/2 ) ( W_{\bfs,\varepsilon}( u' ,v) - 1/2) dv \bigg| &= 
    \bigg|\sum_{\ell=1}^L (1/L+\varepsilon' t_\ell) 
    \bigg( Q_{k,\ell} -\frac12\bigg)
    \bigg( Q_{k',\ell} -\frac12\bigg)\bigg| \\ 
    &\leqslant \frac{1}{4L}|\langle \bfB_{k,\bullet},
    \bfB_{k',\bullet}\rangle| + \frac{L\varepsilon'}4
    \leqslant \frac18. \label{inner}
\end{align}
For $k,k' \in [K]$, let $\omega_{kk'} = \lambda \big\{ I_{\bfs,\varepsilon}^k \cap \tau_1^{-1} \big( I_{\bfs',\varepsilon}^{k'}\big) \big\}$ where 
$\lambda$ is the Lebesgue measure on $\mathbb R$ . 
Notice that $\sum_{k=1}^K \omega_{kk'} = ({1}/{K}) 
+ \varepsilon s'_{k'}$ and $\sum_{k'=1}^K 
\omega_{kk'} = ({1}/{K}) + \varepsilon s_k$. 
We also introduce $h_{\bfs,k}(v) = W_{\bfs, 
\varepsilon}(u_{\bfs,\varepsilon}^k,v) - 1/2 $, where
$u_{\bfs,\varepsilon}^k$ is any point from 
$I_{\bfs,\varepsilon}^k$. We have
\begin{align}
    \| W_{\bfs,\varepsilon} - W_{\bfs',\varepsilon}(\tau_1\otimes \tau_2) \|_{\mathbb L_2}^2 &= \sum_{k=1}^K\sum_{k'=1}^K \omega_{k,k'} 
    \|h_{\bfs,k} - h_{\bfs',k'}\circ\tau_2\|_{\mathbb L_2}^2.
\end{align}
In view of the fact that $|h_{\bfs',k'}(v) |=1/2$ 
for all $v\in [0,1]$ and \eqref{inner}, 
for any $k'\neq k$, we have 
\begin{align}
    \|h_{\bfs',k'}\circ\tau_2 - h_{\bfs',k''}\circ\tau_2\|_{
    \mathbb L_2}^2 &= \|h_{\bfs',k'} \|_{
    \mathbb L_2}^2 + \| h_{\bfs',k''}\|_{
    \mathbb L_2}^2 - 2 \langle h_{\bfs',k'}
    , h_{\bfs',k''} \rangle \geqslant 1/2 - 1/4 
    =  1/4.
\end{align}
By the triangle inequality
\begin{align}
    \|h_{\bfs,k} - h_{\bfs',k'}\circ\tau_2 \|_{\mathbb L_2} 
    + \|h_{\bfs,k} - h_{\bfs',k''}\circ\tau_2 \|_{
    \mathbb L_2} \geqslant  \|h_{\bfs',k'}\circ\tau_2 - h_{\bfs',k''}\circ\tau_2\|_{
    \mathbb L_2} \geqslant 1/2 \, .
\end{align}
As a consequence, for any $k\in [K]$, there exists at most one $k' \in [K]$ such that $\|h_{\bfs,k} - h_{\bfs',k'} 
\circ\tau_2\|_{\mathbb L_2} < 1/4$. If such a $k'$ exists, we denote it by $\pi(k)$. If 
it does not exist, we set $\pi(k)=k$.  Using the same arguments, for any $k' \in [K]$, there is at most one $k \in [K]$ such that $\|h_{\bfs,k} - h_{\bfs',k'}
\circ\tau_2\|_{\mathbb L_2} < 1/4$. 
This implies that $\pi$ is injective and then it is a permutation of $[K]$. 
Furthermore, we get
\begin{align} \label{sum_K}
    \| W_{\bfs,\varepsilon} - W_{\bfs',\varepsilon}\circ 
    (\tau_1 \otimes \tau_2) \|_{\mathbb L_2}^2 & \geqslant \frac{1}{16} \sum_{k=1}^K \sum_{k' \neq \pi(k)} \omega_{k,k'} = \frac{1}{16} \sum_{k= 1}^K (1/K + \varepsilon s_k - \omega_{k,\pi(k)} ) \, .
\end{align}
If the sum $\sum_{k=1}^K (1/K + \varepsilon s_k - \omega_{k,\pi(k)})$ 
is larger than $K\varepsilon/16$, then the lemma is proved. 

In the sequel, we check that the same
is true if $\sum_{k=1}^K (1/K + \varepsilon s_k - \omega_{k, \pi(k)} ) 
< K\varepsilon/16$ as well. Note that
the last inequality can be rewritten as $\sum_{k=1}^K  \omega_{k,\pi(k)} > 1-
{K\varepsilon}/{16}$. 
Let us show that the cardinality of the
set $A = \{k\in\calA_{\bfs}: s'_{\pi(k)} > 0 \text{ and } \omega_{k,\pi(k)} 
\geqslant 1/K \} $ is at least ${7K}/{16}$. Indeed, notice that because $\omega_{k,k'} \leqslant \big( 1/K + \varepsilon s_k \big) \wedge \big(1/K + \varepsilon s_{k'} \big)$, $\omega_{k,\pi(k)} \geqslant 1/K$ implies $ s_k > 0$ and $s'_{\pi(k)}>0$. Therefore, 
\begin{align}
    1-\frac{K\varepsilon}{16} & \leqslant \sum_{k=1}^K  \omega_{k, \pi(k)} 
     = \sum_{k : s_k < 0 } \omega_{k ,\pi(k)} + \sum_{k \in A} \omega_{k, \pi(k)} + \sum_{ \begin{subarray}c k \notin A \\ s_k >0\end{subarray}} \omega_{k, \pi(k)} \\
    & \leqslant \frac{K}{2}\Big(\frac{1}{K} - \varepsilon \Big) + |A| \Big(\frac{1}{K} + \varepsilon \Big) + \Big(\frac{K}2 -|A| \Big) 
    \frac{1}{K} 
    = 1 + \varepsilon\Big(|A| - \frac{K}2\Big),
\end{align}
which leads to $|A|\geqslant 7K/16$. 

Since $\bfs,\bfs'\in\tilde\calC^K$ are such that
$|\calA_{\bfs}\triangle \calA_{\bfs'}|\geqslant K/4$, 
we have  $|\calA_{\bfs} \cap \calA_{\bfs'} | 
\leqslant {3K}/{8}$. Let us choose $B \subset A 
\cap \calA_{\bfs'}^c$ of cardinality $K/16$ and 
set $C=\pi(B)$. Such a choice is possible since
\begin{align}
    |A \cap \calA_{\bfs'}^c| = 
    |A| - |A \cap \calA_{\bfs'}|
    \geqslant \frac{7K}{16} - 
    |\calA_{\bfs} \cap \calA_{\bfs'}|
    \geqslant \frac{7K}{16} - \frac{3K
    }{8} = \frac{K}{16}.
\end{align}
Note also that $B\cap C = \varnothing$. Indeed, 
if $k\in B$, then $\pi(k) \in \calA_{\bfs'}$. 
Therefore, $\pi(k) \notin B$ meaning that 
$k\notin C$. 

For $\ell,\ell' \in [L]$, let $\omega'_{\ell,\ell'} 
= \lambda \big\{ J_{\bft,\varepsilon'}^\ell
\cap \tau_2^{-1} ( J_{\bft,\varepsilon'}^{\ell'} ) 
\big\}$. Using the same arguments as above, we 
obtain the existence of a permutation $\nu:[L]\to[L]$ 
such that, akin to \eqref{sum_K},  
\begin{align}
    \| W_{\bfs,\varepsilon} - W_{\bfs',\varepsilon}
    \circ 
    (\tau_1 \otimes \tau_2) \|_{\mathbb L_2}^2  & \geqslant \frac{1}{16 }\sum_{b=1}^L (1/L + \varepsilon' t_\ell  - \omega'_{\ell,\nu(\ell)} 
    ).
\end{align}
Define the set $A' = \{ \ell \in [L]: t_\ell>0,
\omega'_{\ell,\nu(\ell)} \geqslant 1/L \}$. If $|A'|
\leqslant L/16$, then 
\begin{align}
    \sum_{\ell=1}^L (1/L + \varepsilon' t_\ell 
    - \omega'_{\ell, \nu(\ell)} ) &\geqslant 
    1 - \sum_{t_{\ell}<0}\omega'_{\ell, \nu(\ell)}
    - \sum_{t_{\ell}>0}\omega'_{\ell,\nu(\ell)}  
    (\mathds 1_{\ell\notin A'} + \mathds 1_{\ell\in
    A'})\\
    & \geqslant 1 - \frac{L}2\Big(\frac1L - 
    \varepsilon'\Big) -\Big(\frac{L}{2} - |A'|\Big)
    \frac1L - |A'|\Big(\frac1L+\varepsilon'\Big) 
    = \Big(\frac{L}{2} - |A'|\Big) \varepsilon' \\
    &\geqslant \frac{7L\varepsilon'}{16} \geqslant
     1/16
\end{align}
and therefore $\| W_{\bfs,\varepsilon} - W_{\bfs', 
\varepsilon} \circ (\tau_1 \otimes \tau_2) \|_{
\mathbb L_2}^2 \geqslant 1/256\geqslant K\varepsilon
/64$, where we used that $ 4 K\varepsilon<1 $. 

Suppose now that $|A'|> L/16$. Let $B'$
be an arbitrary subset of $A'$ of cardinality 
$L/16$. We have 
\begin{align}
    \| W_{\bfs,\varepsilon} - W_{\bfs',\varepsilon}
    \circ (\tau_1 \otimes \tau_2) \|_{\mathbb L_2}^2  
    & \geqslant \sum_{k \in B}\sum_{\ell \in B'}
    \int_{I_{\bfs,\varepsilon}^k \times 
    J_{\bfs',\varepsilon}^{\ell}} \big( 
    W_{\bfs,\varepsilon} - W_{\bfs',\varepsilon}(\tau_1\otimes\tau_2) \big)^2(u,v)\, du\,dv \\
    & \geqslant \sum_{k \in B}\sum_{\ell \in B'} \omega_{k, \pi(k)}\omega'_{\ell, \nu(\ell)} 
    (Q_{k,\ell} - Q_{\pi(k),\nu(\ell)})^2 \\
    & \geqslant \frac{1}{4KL} \sum_{k \in B}
    \sum_{\ell \in B'} (B_{k,\ell} - B_{\pi(k), 
    \nu(\ell)})^2 .
\end{align}
By \Cref{matrix_B}, the last term is larger than 
$1/(4\times 512) \geqslant  K \varepsilon/512$, 
and the claim of the lemma follows.
\end{proof}

\subsection{Proof of \eqref{minimax2}: error
due to the unknown values of the graphon}

Similarly to the previous proof, we will use 
\citep[Theorem 2.7]{Nonparametric}, which needs 
a class of graphons that are well separated for 
the distance $\delta$ and that generate similar
distributions on the space of $n\times m$ matrices. 
In this proof, all the graphons of the set
will have the same partitions and will differ
only by the values of the function taken on
this partition. 

Let $\calQ_0 = \{(\rho/2)(1-\varepsilon), (\rho/2)
(1 + \varepsilon) \}^{K \times L}$ be the set of 
all $K\times L$ matrices with entries equal
either $(\rho/2)(1-\varepsilon)$ or $(\rho/2)(1 
+\varepsilon)$, where $\varepsilon\in(0,1/2)$ 
will be specified later. For any  $\bfQ \in 
\calQ_0$, we define  the graphon
\begin{align}
    W_\bfQ(u,v) = \sum_{k=1}^K\sum_{\ell=1}^L
    Q_{k,\ell} \mathds 1_{[(k-1)/K, k/K)}(u) 
    \mathds 1_{[(\ell-1)/L,\ell/L )}(v) \,.
\end{align} 
We need two technical lemmas for completing the
proof. These lemmas are stated below, whereas 
their proofs are postponed to the end of this
subsection.  For any pair of permutations 
$\pi:[K]\to[K]$ and $\nu:[L]\to[L]$, and any matrix $\bfQ$, 
we denote by $\bfQ^{\pi,\nu}$ the matrix 
with permuted rows and columns $Q^{\pi,\nu
}_{k,\ell} = Q_{\pi(k),\nu(\ell)}$.

\begin{lemma}\label{matrix_B_2}
For $K$ and $L$ large enough satisfying 
$KL\geqslant L\log^2 L + K\log^2 K$, there exists 
a set $\calQ \subset \calQ_0$ satisfying 
$\log|\calQ| \geqslant {KL}/{32}$ and 
$\min_{\pi,\nu}\| \bfQ_1 - \bfQ^{\pi,\nu}_2 
\|_{\sf F}^2 \geqslant {\rho^2\varepsilon^2 
KL}/{8}$ for every ${\bfQ_1,\bfQ_2\in\calQ}$
such that $\bfQ_1\neq \bfQ_2$.
\end{lemma}

\begin{lemma}\label{lem:low3}
The following assertions hold true
\vspace{-8pt}
\begin{enumerate}\itemsep=2pt
    \item If $\bfQ_1$ and $\bfQ_2$ are such that 
    $\min_{\pi,\nu} \|\bfQ_1 - \bfQ_2^{\pi,\nu}\|_{
    \sf F}^2 \geqslant \rho^2\varepsilon^2 KL/8$,  
    then $\delta^2(W_{\bfQ_1}, W_{\bfQ_2}) \geqslant
    \rho^2 \varepsilon^2/8$.
    
    \item For any pair of matrices $\bfQ_1$ and 
    $\bfQ_2$ from $\calQ$, we have $\KL(\bfP_{W_{
    \bfQ_1 }} ||\, \bfP_{W_{\bfQ_2}} ) \leqslant 
    6Nnm\rho \varepsilon^2$.
\end{enumerate}
\end{lemma}

We set $\varepsilon^2 = \frac{1}{54 \times 32} 
\big(\frac{KL}{Nnm \rho} \wedge 1 \big)$, which allows 
us to apply \citep[Theorem 2.7]{Nonparametric}, since 
in view of \Cref{lem:low3} and \Cref{matrix_B_2},
\begin{align}
    \KL(\bfP_{W_{\bfQ_1}}||\,\bfP_{W_{\bfQ_2}}) 
    &\leqslant \frac{1}{9} \log |\calQ|,\qquad 
    \delta^2(W_{\bfQ_1}, W_{\bfQ_2}) 
    \geqslant \frac{1}{13 824}\Big(\frac{\rho KL}{Nnm} \wedge \rho^2 \Big),
\end{align}
for every $\bfQ_1,\bfQ_2\in\calQ$. 
This completes the proof of \eqref{minimax2}.

\begin{proof}[Proof of \Cref{matrix_B_2}]
Without loss of generality, we assume in this 
proof that $\rho =2$. 
We define the pseudo-distance $ \delta (\bfQ_1, 
\bfQ_2) = \min_{\pi,\nu} \|\bfQ_1 - \bfQ^{\pi, 
\nu}_2\|_{\sf F}$, where the minimum is taken 
over all the permutations of $[K]$ and $[L]$. 
Let $\calQ$ be a maximal subset of $\calQ$ of 
matrices $\bfQ$ that are $r :=\rho\varepsilon 
\sqrt{{KL}/{2}}$-separated with respect to 
$\delta$. By maximality of $\calQ$, we have 
the inclusion
\begin{align}
    \calQ \subset \bigcup_{
    \bfQ \in \calQ} \mathbb B_\delta(\bfQ,r), 
\end{align}
where $\mathbb B_\delta(\bfQ,r)$ is the 
ball centered at $\bfQ$ with radius $r$ with 
respect to the $\delta$-distance. Hence, $|\calQ| 
\cdot|\mathbb B_\delta(\bfQ,r)| \geqslant 2^{KL}$ since all 
the balls have a the same cardinality. Notice 
that $\mathbb B_\delta(\bfQ,r) \subset 
\bigcup_{\pi, \nu} \mathbb B_{\sf F}(\bfQ^{\pi,
\nu},r)$ yielding $|\mathbb B_\delta(\bfQ,r)| 
\leqslant  K!L! \,|\mathbb B_{\sf F}(\bfQ,r)|$. 
If $\bfQ_1 , \bfQ_2 \in \calQ$, we have  $\|
\bfQ_1-\bfQ_2 \|_{\sf F}^2= 4\rho^2\varepsilon^2 
d_{\sf H} (\bfQ_1, \bfQ_2)$ with $d_{\sf H}$ 
the Hamming distance. 
Then $\mathbb B_{\sf F}(\bfQ, r)=\mathbb B_{
\sf H}(\bfQ,{r^2}/{(2\rho\varepsilon)^2})$ with 
${r^2}/{(2\rho\varepsilon)^2} = {KL}/{8}$. The Varshamov-Gilbert 
lemma \citep[Lemma 2.9]{Nonparametric} yields
\begin{align}
    |\calQ'|\cdot|\mathbb B_{\sf H}(\bfQ,KL/8)|
    \leqslant 2^{KL}
\end{align}
with $\calQ'$ a maximal subset of matrices 
$\bfQ$ that are ${KL}/{8}$-separated, 
and $|\calQ'| \geqslant 2^{KL/8}$. Thus we obtain
\begin{align}
    |\calQ| &\geqslant  \frac{2^{KL}}{ |\mathbb B_\delta(\bfQ,r)|} \geqslant \frac{2^{KL}}{K!
    L!\,|\mathbb B_{\sf F}(\bfQ,r)|} = \frac{2^{KL}
    }{K! L!\,|\mathbb B_{\sf H}(\bfQ,KL/8)|}
    \geqslant \frac{2^{KL}}{K!L! 2^{KL} 2^{-KL/8}}\\ 
    &= \exp\Big( \frac{KL}{8} \log 2 - \log(K!L!) \Big) 
    \geqslant \exp\Big( \frac{KL}{16} - K\log K - 
    L\log L \Big) \,.
\end{align}
The last term is greater than $e^{KL/32}$ for $K$ and $L$ greater than some constants.
\end{proof}

\begin{proof}[Proof of \Cref{lem:low3}]
~
\paragraph*{First claim}
Let $\tau_1,\tau_2 :[0,1] \to [0,1]$ be two measure preserving bijections. We want to prove that
\begin{align}
    \|W_{\bfQ_1} - W_{\bfQ_2}\circ(\tau_1\otimes 
    \tau_2)\|_{\mathbb L_2}^2 \geqslant 
    \varepsilon^2/8.
\end{align}
For any $k,k' \in [K]$, let $\omega_{k,k'} 
= \lambda \big([(k-1)/K,k/K] \cap \tau_1^{-1}
\big( [(k'-1)/K,k'/K] \big) \big)$ where 
$\lambda$ is the Lebesgue measure on $[0,1]$. 
Similarly, for any $\ell,\ell'\in [L]$, let $\omega_{\ell,\ell'}'  = \lambda \big([(\ell-1)
/L,\ell/L] \cap \tau_2^{-1} \big( [(\ell'-1)/L, 
\ell'/L] \big) \big)$. 

We have that $\sum_{k} \omega_{k,k'} = {1}/{K} =
\sum_{k'} \omega_{k,k'}$, that is the matrix 
$M\omega$ is doubly stochastic. For any permutation $\pi$ of $[K]$, denote $A(\pi)$ the corresponding permutation matrix. By the Birkhoff-von Neumann theorem, $M\omega$ is a convex combination of permutation matrices, so there exist positive numbers $\gamma_{\pi}$ such that $\omega= \sum_{\pi} \gamma_{\pi} A(\pi)$ and $\sum_{\pi} \gamma_{\pi} 
= {1}/{K}$, where the sums are taken over all the permutations of $[K]$. Thus
\begin{align}
    \|W_{\bfQ_1} - W_{\bfQ_2}\circ(\tau_1
    \otimes\tau_2)\|^2_{\mathbb L_2}& 
    = \sum_{k,k' \in [K]} \sum_{\ell,\ell' 
    \in [L]} \omega_{k,k'} \omega'_{\ell,\ell'} | Q_{1,k,\ell} - Q_{2,k',\ell'} |^2 \\
    & =  \sum_{\pi, \nu} \sum_{k,k' \in [K]} 
    \sum_{\ell,\ell' \in [L]} \gamma_{\pi}
    \gamma_{\nu} A(\pi)_{k,k'} A(\nu)_{\ell,\ell'}|
    Q_{1,k,\ell} - Q_{2,k',\ell'} |^2 \\
    &= \sum_{\pi, \nu} \sum_{k \in [K]} \sum_{\ell 
    \in [L]} \gamma_{\pi}\gamma_{\nu} | Q_{1,k,\ell} - Q_{2,\pi(k),\nu(\ell)} |^2 \\
    &= \sum_{\pi, \nu}\gamma_{\pi}\gamma_{\nu}  \underbrace{\| \bfQ_1 - \bfQ_2^{\pi,\nu} \|_{\sf F}^2}_{\geqslant \rho^2\varepsilon^2KL/{8}} 
    \geqslant \frac{\rho^2\varepsilon^2}{8}
\end{align}
and the claim of the lemma follows.

\paragraph*{Second claim}
Let $\bzeta=(\zeta_1, \dots, \zeta_n)$ be a vector of $n$ i.i.d random variables uniformly distributed on $[K]$. We also denote by $\bchi=(\chi_1, \dots, \chi_m)$ a vector of $m$ i.i.d random variables uniformly distributed on $[L]$. Let $\bTheta_1 \in [0,1]^{n 
\times m}$ with entries $(\Theta_1)_{ij}=
(\bfQ_1)_{\zeta_i, \chi_j}$. Assume that $\bfH_1$ is, conditionally on $\bzeta$ and $\bchi$, a matrix sampled according to the binomial model with parameter $N$ and the probability matrix $\bTheta_1$. Notice that $\bfH_1$ has distribution $\bfP_{W_{\bfQ_1}}$. We introduce $\alpha_\bfa= \bfP(\bzeta=\bfa)$, $\beta_\bfb=\bfP(\bchi=\bfb)$ and $p_{\bfH \bfa \bfb}^{(1)} = \bfP(\bfH_1=\bf H | 
\bzeta = \bfa, \bxi = \bfb)$ for any $\bfa \in [K]^n, 
\bfb \in [L]^m$ and $\bfH \in \{0,1\}^{n \times m}$. 
We have similar notation $p_{\bfH \bfa \bfb}^{(2)}$ replacing the indices above. Then
\begin{align}
    \KL(\bfP_{W_{\bfQ_1}},\bfP_{W_{\bfQ_2}}) &= \sum_{\bfH} \sum_\bfa \sum_\bfb \alpha_\bfa \beta_\bfb p_{\bfH \bfa \bfb}^{(1)} \log \bigg(\frac{\sum_\bfa \sum_\bfb \alpha_\bfa \beta_\bfb p_{\bfH \bfa \bfb}^{(1)}}{\sum_\bfa \sum_\bfb \alpha_\bfa \beta_\bfb p_{\bfH \bfa \bfb}^{(2)}} \bigg) \\
    \text{(Jensen)} \qquad &\leqslant \sum_\bfa \sum_\bfb \alpha_\bfa \beta_\bfb \sum_\bfH p_{\bfH \bfa \bfb}^{(1)} \log \bigg( \frac{p_{\bfH \bfa \bfb}^{(1)}}{p_{\bfH \bfa \bfb}^{(2)}} \bigg) \, .
\end{align}
When $\bfa$ and $\bfb$ are fixed, the sum over $\bfH$
is the Kullback-Leibler divergence between two $nm$-product of Binomial measures, each of which has as parameter either $(N,p)$ or $(N,q)$ with $ p := \rho(1+\varepsilon)/2$ and $q:=\rho(1-\varepsilon)/2$. This gives
\begin{align}
    \KL(\bfP_{W_{\bfQ_1}}||\,\bfP_{W_{\bfQ_2}}) 
    \leqslant Nnm \kappa(p,q)
\end{align}
where $\kappa(p,q)$ is the Kullback-Leibler divergence between two Bernoulli measures with parameter $p$ and $q$ respectively. We have $\kappa(p,q) \leqslant (p-q)^2(p^{-1} +q^{-1}) = 4\rho\varepsilon^2/(1-\varepsilon^2) \leqslant 16\rho\varepsilon^2/3$.
This completes the proof of the lemma. 
\end{proof}

\subsection{Proof of \eqref{minimax3}} 

Fix some $\varepsilon \in (0,1/4)$, and let $W_1 \equiv \rho/2$ be a constant graphon. We define also $W_2(u,v)= \rho(1/2 + \varepsilon)$ if $(u,v) \in [0,1/2)^2 \cup [1/2,1]^2$ and $W_2(u,v)= \rho(1/2 - \varepsilon)$ elsewhere. We get
\begin{align}
    \delta(W_1, W_2)= \rho\varepsilon\,.
\end{align}
Thus we have
\begin{align}
    \inf_{\widehat W} \max_{W^* \in \{ W_1,  W_2 \}} \Ex_{W^*} \big[ \delta^2(\widehat W , W^*) \big] &\geqslant \frac{1}{2} \bigg( \int \delta^2(
    \widehat W, W_1) \, d \bfP_{W_1} + \int \delta^2(\widehat W, W_2) \, d \bfP_{W_2} \bigg) \\
    & \geqslant \frac{1}{2}  \int \delta^2(\widehat W, W_1) + \delta^2(\widehat W, W_2)\, \min(d \bfP_{W_1}, d \bfP_{W_2}) \\
    & \geqslant \frac{\delta^2(W_2, W_1)}{4} \int \min(d \bfP_{W_1}, d \bfP_{W_2}) \\
    & \geqslant \frac{\rho^2 \varepsilon^2}{8} \exp
    (-\chi^2(\bfP_{W_1}||\,\bfP_{W_2})),
\end{align}
where $\chi^2(\bfP_{W_1}||\,\bfP_{W_2})$ stands for 
the chi-square divergence between $\bfP_{W_1}$ and $\bfP_{W_1}$. In the last inequality, we used (2.24) and (2.26) from \citep{Nonparametric}. Finally, the next lemma gives an upper-bound on the chi-square divergence, which allows us to complete the proof, taking $\varepsilon^2={c_0}/{12 \rho N\sqrt{nm}}$.

\begin{lemma}
    If $14 \sqrt{nm} N \rho \varepsilon^2\leqslant 1$, 
    then $\chi^2(\bfP_{W_1}||\,\bfP_{W_2}) \leqslant 
    1/4$.
\end{lemma}
\begin{proof}
Let $L(\bfH)$ be the Radon-Nikodym derivative of
$\bfP_{W_2}$ with respect to $\bfP_{W_1}$. We have 
that $\chi^2(\bfP_{W_1}||\,\bfP_{W_2})=\Ex_{W_1}
[L(\bfH)^2]-1$, so it remains to prove that 
$\Ex_{W_1}[L(\bfH)^2] \leqslant 5/4$. In the sequel,
$\Ex[\,\cdot\,]$ will refer to $\Ex_{W_1}[\,\cdot\,]$.
We also introduce $p_0=\rho/2$, $p_1= \rho(1/2 + 
\varepsilon)$ and $p_2= \rho(1/2 - \varepsilon)$. 
To avoid too long formulae, we will use notation
$\bar p_j = 1 - p_j$ for $j=0,1,2$. 

As the graphon $W_2 \in \calW_\rho[2,2]$, we can 
assume that $\{U_i\}$ are i.i.d.\ Bernouilli random variables with parameter 1/2, and similarly for $\{V_j\}$. Given $\{U_i\}$ and $\{V_j\}$, define the set $S= \{ (a,b) : U_a=V_b \}$. For $(i,j) \in S$ (resp. $S^c$), $H_{i j}$ has the binomial distribution of parameters $(N,p_1)$ (resp. $(N,p_2)$). Let $\mu$ be the distribution of $S$, then we have 
\begin{align}
    L(\bfH) = \int L_S(\bfH) \, d\mu(S)
\end{align}
with
\begin{align}
    L_S(\bfH) =\bigg( \frac{\bar p_1}{\bar p_0} \bigg)^{N|S|}  \bigg( \frac{\bar p_2}{\bar p_0} \bigg)^{Nnm-N|S|} \prod_{(a,b) \in S} \bigg( \frac{p_1\bar p_0}{p_0\bar p_1} \bigg)^{H_{a,b}}  \prod_{(a,b) \in S^c} \bigg( \frac{p_2\bar p_0}{p_0\bar p_2} \bigg)^{H_{a,b}} \, .
\end{align}
By Fubini theorem, we can write $\dis \Ex[L(\bfH)^2]= \int \Ex[L_{S_1}(\bfH)L_{S_2}(\bfH)] \, d \mu(S_1)d \mu(S_2)$ with
\begin{align}
    \Ex[L_{S_1}(\bfH)L_{S_2}(\bfH)] = & \Big( \frac{\bar p_1}{\bar p_0} \Big)^{(|S_1|+|S_2|)N} \Big( \frac{\bar p_2}{\bar p_0} \Big)^{(2nm-|S_1|-|S_2|)N} \Ex \Bigg[ \prod_{(a,b) \in S_1 \cap S_2} \Big( \frac{p_1\bar p_0}{p_0\bar p_1} \Big)^{2H_{a,b}} \\
    & \times\prod_{(a,b) \in S_1^c \cap S_2^c} \bigg( \frac{p_2\bar p_0}{p_0\bar p_2} \bigg)^{2 H_{a,b}} \times \prod_{(a,b) \in S_1 \Delta S_2} \bigg( \frac{p_1p_2\bar p_0^2}{p_0^2\bar p_1\bar p_2} \bigg)^{H_{a,b}}    \Bigg] .
\end{align}    
We write
\begin{align}
    \Big( \frac{\bar p_1}{\bar p_0} \Big)^{N|S_1|+N|S_2|}
    &= \Big( \frac{\bar p_1}{\bar p_0} \Big)^{2N|S_1\cap S_2|+N|S_1\Delta S_2|}\\
    \Big( \frac{\bar p_2}{\bar p_0} \Big)^{2Nnm-N|S_1|-N|S_2|} &= \Big( \frac{\bar p_2}{\bar p_0} \Big)^{2N|S_1^c\cap S_2^c| + N|S_1\Delta S_2|}. 
\end{align}
Furthermore, since the expectation is taken with respect to $\bfP_{W_1}$ and  variables $H_{ij}$ are independent conditionally on $\{U_i\}, \{V_j\}$, we have

\begin{align}
    \Ex \bigg[ \prod_{(a,b) \in S} \Big( \frac{q_1q_2\bar p_0^2}{p_0^2\bar q_1\bar q_2} \Big)^{H_{a,b}}\bigg]
    & = \Big\{\Big( \frac{q_1q_2\bar p_0^2}{p_0^2
    \bar q_1\bar q_2} \Big)p_0 + \bar p_0\Big\}^{N|S|} \\
    &= \Big\{\frac{\bar p_0^2}{\bar q_1\bar q_2}
    \Big(  1+ \frac{(q_1 - p_0)(q_2 - p_0))}{p_0\bar p_0}\Big)\Big\}^{N|S|}
\end{align}
for any subset $S$ of $[n]\times [m]$ and for any 
$q_1,q_2\in[0,1]$. Applying this formula to 
the triplets $(S_1\cap S_2, p_1,p_1)$, $(S_1^c\cap 
S_2^c, p_2,p_2)$ and $S_1\Delta S_2, p_1,p_2$, we get
\begin{align}
    \Ex \bigg[ \prod_{(a,b) \in S_1 \cap S_2} \Big( \frac{p_1\bar p_0}{p_0\bar p_1} \Big)^{2H_{a,b}}\bigg]
    &= \Big\{\frac{\bar p_0^2}{\bar p_1^2}\Big(  1+ \frac{(p_1 - p_0)^2}{p_0\bar p_0}\Big)\Big\}^{N|S_1\cap S_2|},\\
    \Ex \bigg[ \prod_{(a,b) \in S_1^c \cap S_2^c} \Big( \frac{p_2\bar p_0}{p_0\bar p_2} \Big)^{2H_{a,b}}\bigg]
    &= \Big\{\frac{\bar p_0^2}{\bar p_2^2}\Big(  1+ \frac{(p_2 - p_0)^2}{p_0\bar p_0}\Big)\Big\}^{N|S_1^c
    \cap S_2^c|},\\
    \Ex \bigg[ \prod_{(a,b) \in S_1 \Delta S_2} \Big( \frac{p_1p_2\bar p_0^2}{p_0^2\bar p_1\bar p_2} \Big)^{H_{a,b}}\bigg]
    &= \Big\{\frac{\bar p_0^2}{\bar p_1\bar p_2}
    \Big(  1+ \frac{(p_1 - p_0)(p_2 - p_0))}{p_0\bar p_0}\Big)\Big\}^{N|S_1\Delta S_2|}.
\end{align}
Combining these results, we arrive at the expression
\begin{align}
    \Ex[L_{S_1}(\bfH)L_{S_2}(\bfH)] = & A^{N|S_1 \cap S_2|} B^{N|S_1^c \cap S_2^c|} C^{N|S_1 \Delta S_2|} 
\end{align}
with 
\begin{align}
    A &= 1+ \frac{(p_1-p_0)^2}{p_0(1-p_0)} = 1 +
    \frac{\rho^2 \varepsilon^2}{p_0(1-p_0)}\leqslant 
    1 + 4\rho \varepsilon^2\\
    B &= 1+ \frac{(p_2-p_0)^2}{p_0(1-p_0)} = 1 +
    \frac{\rho^2 \varepsilon^2}{p_0(1-p_0)}= A\\
    C &= 1+ \frac{(p_1 -  p_0)(p_2 - p_0)}{p_0(1-p_0)} 
    = 1 - \frac{\rho^2 \varepsilon^2}{p_0(1-p_0)}
    \leqslant \frac1A.
\end{align}
This leads to
\begin{align}
    \Ex[L_{S_1}(\bfH)L_{S_2}(\bfH)] 
    &= A^{N|S_1 \cap S_2|+N|S_1^c \cap S_2^c|} C^{N|S_1 \Delta S_2|} 
    \leqslant A^{N|S_1 \cap S_2|+N|S_1^c \cap S_2^c| - N|S_1 \Delta S_2| } \\
    & = A^{ 2N|S_1 \cap S_2| + 2N|S_1^c \cap S_2^c| - Nnm}\\
    & \leqslant  \exp \Big[ \Big( 2|S_1 \cap S_2| + 
    2 |S_1^c \cap S_2^c| - nm  \Big) 4 N \rho \varepsilon^2 \Big] \, .
\end{align}
Thus, it remains to bound an exponential moment of $T =  |S_1 \cap S_2| + |S_1^c \cap S_2^c| $ where $S_1$ and $S_2$ are independent and distributed according to $\mu$. We denote by $(\{U_i\},\{V_j\})$ and $(\{U_i'\},\{V_j'\})$ 
the independent collections of variables that generate 
$S_1$ and $S_2$ respectively. 
This means that $S_1 = \{(i,j)\in[n]\times[m] : U_i = V_j\}$ and $S_2 = \{(i,j)\in[n]\times[m] : U'_i = V'_j\}$. 

Let us also introduce the random variables 
\begin{align}
    X = \sum_{i=1}^n U_iU_i' + (1-U_i)(1-U_i'),\quad
    Y = \sum_{j=1}^m V_jV_j' + (1-V_j)(1-V_j').
\end{align}
One can check that $X$ and Y are independent, drawn 
respectively from the binomial distributions $\mathcal
B(n,1/2)$ and $\mathcal B(m,1.2)$, and that 
\begin{align}
    T & = XY + (n - X)(m-Y) = 2( X - n/2)(Y - m/2) + nm/2.
\end{align}
The Hoeffding inequality \citep[Eq. 4.16]{Hoeffding} implies that
\begin{align}
    \Ex[e^{\lambda (X - n/2)(Y - m/2) }] & \leqslant 
    \Ex[e^{m\lambda^2( X- n/2)^2/8}].
\end{align}
Furthermore, the fact that $X-n/2$ is $n/8$-sub-Gaussian
(once again, according to the Hoeffding inequality)
entails that  $(X-n/2)^2 - n/4$ is sub-exponential 
with the parameter $2n$ \cite[Lemma 1.12]{rigollet2015high}. Hence,
\begin{align}
    \Ex[e^{m\lambda^2( X- n/2)^2/8}] &\leqslant 
    e^{nm\lambda^2/32}e^{n^2m^2\lambda^4/32}
    \leqslant e^{5nm\lambda^2/32}
\end{align}
provided that $\lambda \leqslant 2/\sqrt{nm}$. 
Combining all these results, we arrive at
\begin{align}
    \Ex[L(\bfH)^2] =
    \Ex[e^{(2T - nm) 4N\rho\varepsilon^2}] = 
    \Ex[e^{(X-n/2)(Y-m/2)16N\rho \varepsilon^2}]\\
    &\leqslant e^{40 nmN^2\rho^2\varepsilon^4}
    \leqslant 5/4
\end{align}
provided that $14 \sqrt{nm} N \rho \varepsilon^2 
\leqslant 14\sqrt{(1/40)\log(5/4)}\approx 1.04$. 
This completes that proof of the lemma.
\end{proof}
Since all three lower bounds have been proved,  
the proof of \Cref{lower_bound} is complete.

\section{Auxiliary results}\label{Auxiliary}

\begin{lemma}\label{exp_bound}
Let $X$ be a random variable and $a\in\mathbb R$, $b,c,d
\geqslant 0$ be some constants. If 
\begin{align}
    \prob(X \geqslant a + bt + ct^2) & \leqslant de^{-t}
    \qquad \text{for all $t\geqslant0$},
\end{align}
then $\Ex[X] \leqslant a + bd + 2cd$.
\end{lemma}

\begin{proof}
In the case $c=0$, this inequality is well-known. 
Therefore, we consider only the case $c>0$. Without 
loss of generality, we assume that $a=0$ and $c=1$. 
Indeed, we can always reduce to this case by 
considering the random variable $X' = (X-a)_+/c$ 
with $b' = b/c$. Thus, we know that $\prob(X\geqslant
t^2+bt)\leqslant de^{-t}$ for every $t\geqslant 0$. 
Note that the condition $b\geqslant 0$ entails that 
the mapping $t\mapsto t^2+bt$ defined on $[0,+\infty)$ 
is bijective. Setting $z = t^2+bt$, this implies that
\begin{align}
    \prob\big( X \geqslant z\big)\leqslant d\exp\big\{(b/2)-\sqrt{z+(b/2)^2} \big\},
    \qquad \forall z\ge 0.
\end{align}
This inequality yields
\begin{align}
    \Ex[X] & \leqslant d\int_0^\infty \exp\big\{(b/2)-\sqrt{z+(b/2)^2}\big\}\,dz\\
    & = d\int_0^\infty e^{-t}(2t+b)\,dt\\
    & =  bd + 2d.
\end{align}
This completes the proof. 
\end{proof}

We say that a zero-mean random variable $\zeta$ 
satisfies the $(a,b)$-Bernstein condition, if we have
\begin{align}
    \Ex[e^{\lambda \zeta}] \leqslant \exp\bigg\{\frac{ 
    \lambda^2 a}{2(1-b|\lambda|)}\bigg\} \qquad 
    \text{provided that}\quad |\lambda|\leqslant 1/b.
\end{align}
One can show that if $\xi$ satisfies the $(a,b)$-Bernstein
condition, then the variance of $\xi$ is bounded from 
above by $a$. Indeed, since $x^2\leqslant 2(e^x-1-x - x^3/6)$ 
for every $x\in\mathbb R$, replacing $x$ by $\lambda\xi$
for $\lambda$ small enough and taking the expectation, 
we get
\begin{align}
    \lambda^2\Ex[\xi^2] &\leqslant 2(\Ex[e^{\lambda \xi}]
    - 1 - \lambda\Ex[\xi] -(\lambda^3/6)\Ex[\xi^3])\\
    &\leqslant 2\bigg(\exp\bigg\{\frac{ 
    \lambda^2 a}{2(1-b|\lambda|)}\bigg\} - 1 \bigg) 
    -(\lambda^3/3)\Ex[\xi^3] = \frac{\lambda^2 a}{ 
    1- b|\lambda|} + o(\lambda^2)
\end{align}
as $\lambda\to 0$. Dividing the two sides of the last 
inequality by $\lambda^2$ and letting $\lambda$ go to
zero, we get $\var[\xi] = \Ex[\xi^2]\leqslant a$.

\begin{lemma}[Bernstein inequality]\label{Bernstein:ineq}
    Let $\bxi =(\xi_1,\ldots,\xi_n)$ be a zero-mean
    random vector with independent coordinates and let
    $\balpha\in\mathbb R^n$ be a deterministic
    vector. Assume that for some $a,b>0$, all $\xi_i$'s
    satisfy the $(a,b)$-Bernstein condition. Then, for
    every $\delta\in(0,1)$, we have
    \begin{align}
        \prob\Big(\balpha\!^\top\bxi\leqslant 
        \sqrt{2a\log(1/\delta)}\, \|\balpha\|_2 +  b\|\balpha\|_\infty\log(1/\delta)\Big)
        \geqslant 1-\delta,\\
        \prob\Big(|\balpha\!^\top\bxi|\leqslant 
        \sqrt{2a\log(2/\delta)}\, \|\balpha\|_2 +  b\|\balpha\|_\infty\log(2/\delta)\Big)
        \geqslant 1-\delta.
    \end{align}
\end{lemma}
\begin{proof}
    Without loss of generality, we assume that 
    $\|\balpha\|_\infty=1$. The Markov inequality 
    implies that for any $\lambda$ satisfying 
    $|\lambda|\leqslant 1/b$,
    \begin{align}
        \prob(\balpha\!^\top\bxi\geqslant t)\leqslant 
        e^{-\lambda t}\Ex[e^{\lambda\balpha\!^\top\bxi}] 
        = e^{-\lambda t} \prod_{i\in [n]} 
        \Ex[e^{\lambda\alpha_i\xi_i}] \leqslant 
        \exp\Big(-\lambda t + \frac{\lambda^2 a\|\balpha\|_2^2}{2(1-b|\lambda|)}\Big).
    \end{align}
    One can check that, for $h(u) = 1+u - \sqrt{1+2u}$, 
    $u>0$,
    \begin{align}
        \sup_{|\lambda|\leqslant 1/b} \Big(\lambda t - 
        \frac{\lambda^2 A}{2(1-b|\lambda|)}\Big) = 
        \frac{A}{b^2}h\bigg(\frac{bt}{A}\bigg)\geqslant 
        \frac{t^2}{2(A + bt)},\qquad \forall t>0.
    \end{align}
    Hence, we get
    \begin{align}
        \prob(\balpha\!^\top\bxi\geqslant t)\leqslant 
        \exp\bigg\{ - \frac{a\|\balpha\|_2^2}{b^2} h\bigg(\frac{bt}{a\|\balpha\|_2^2}\bigg) 
        \bigg\},\qquad \forall t\in\mathbb R.
    \end{align}
    Thus, if $bt/a\|\balpha\|_2^2 \geqslant h^{-1}
    (b^2z/a\|\balpha\|_2^2)$, we have $ \prob(
    \balpha\!^\top\bxi \geqslant  t)\leqslant e^{-z}$. 
    One can check that $h^{-1}(v) = v + \sqrt{2v}$, 
    which implies that the above condition of $t$ is 
    equivalent to
    \begin{align}
        t\geqslant \frac{a\|\balpha\|_2^2}{b} 
        \bigg(\frac{b^2z}{a\|\balpha\|_2^2} + 
        \frac{b\sqrt{2z}}{\sqrt{a}\|\balpha\|_2}\bigg) = 
        bz + \|\balpha\|_2\sqrt{2az}.
    \end{align}
    Replacing $z$ by $\log(1/\delta)$, we get the first
    inequality of the lemma. 
    
    For the second inequality, we simply remark that 
    the moment generating function of the random 
    variables $-\xi_i$ satisfy the same assumption 
    as the one of $\xi_i$'s. Therefore, we have 
    \begin{align}
        \prob\Big(\balpha\!^\top\bxi\leqslant 
        \sqrt{2a\log(1/\delta)}\, \|\balpha\|_2 +  
        b\log(1/\delta)\Big)\geqslant 1-\delta,\\
        \prob\Big(-\balpha\!^\top\bxi \leqslant 
        \sqrt{2a\log(1/\delta)}\, \|\balpha\|_2 +  
        b\log(1/\delta)\Big)\geqslant 1-\delta.
    \end{align}
    Using the union bound and replacing $\delta$ 
    by $\delta/2$, we get the second claim of the 
    lemma. 
\end{proof}

\begin{lemma}
    Let $B>0$ and $N\in\mathbb N$.
    \vspace{-7pt}
    \begin{enumerate} 
        \item If $\zeta$ is a zero mean random 
        variable such that $\prob(|\zeta|\leqslant 
        B) = 1$, then it satisfies the $(a,b)$-Bernstein 
        condition with $a = \var[\zeta]$ and $b=B/3$.
        
        \item If $\zeta$ is a the average of $N$ 
        independent zero-mean random variables each 
        of which takes values in $[-B,B]$, then it
        satisfies the $(a,b)$-Bernstein condition 
        with $a = \var[\zeta]$ and $b=B/(3N)$.
        
        \item If $\zeta = (\zeta'-\Ex[\zeta'])/N$ 
        with $\zeta'$ drawn from a Poisson distribution
        with intensity $N\theta$, then it satisfies 
        the $(a,b)$-Bernstein condition with $a =
        \var[\zeta]$ and $b=1/(3N)$. 
    \end{enumerate}
\end{lemma}

\begin{proof}
    The proof of these claims is quite standard and based
    on the inequality $e^z - z - 1 \leqslant 3z^2/(3-|z|)$
    provided that $|z| < 3$. 
\end{proof}

%%%%%%%%%%%%%%%%%%%%%%%%%%%%%%%%%%%%%%%%%%%%%%%%%%%%%%%%%%
\begin{lemma}\label{partition1}
Let $\xi_1, \dots ,\xi_N$ be centered and independent 
random variables satisfying $(\sigma_\xi^2, b_\xi) 
$-Bernstein condition. For every partition $G=\{ A_1,
\dots,A_k\}$ of $[N]$, let us define the projection 
matrix $\Pi_G$ by
\begin{align}
    (\Pi_G\bv)_i = \frac1{|A_j|}\sum_{\ell\in A_j} 
    v_{\ell} \, \quad \text{if}\quad i\in A_j,\quad
    \forall \bv\in\mathbb R^N.
\end{align}
We have that with probability 
at least $1-e^{-t}$ for all $t>0$
\begin{align}
    \forall{\bv}, \ \bv^\top \Pi_G\bxi \leqslant
    \frac12 \|\Pi_G\bxi\|_2\|\bv\|_2 + \sigma_\xi
    \sqrt{2 (t + \log M)}\| \bv\|_2 + b_\xi 
    (t + \log M) \|\bv\|_\infty.
\end{align}
where $\log M \leqslant k\log 12\leqslant 2.5k$. On 
the same event, we have
\begin{align}
    \|\Pi_G\bxi\|_2^2 \leqslant 4  (t + \log M)\big( 
    2\sigma_\xi^2 + b_\xi \| \Pi_G \bxi\|_\infty\big).
\end{align}
\end{lemma}
\begin{proof}
Let $\bw\in\mathbb R^N$ be a vector such that 
$\| \Pi_G \bw \|_2=1$. We have
$\bw^\top \Pi_G \bxi = \sum_{i\in[N]}(\Pi_G\bw)_i\xi_i$
where the terms of the last sum are independent. 
The Bernstein inequality yields
\begin{align}\label{lem5:eq1}
    \prob\big(\bw^\top \Pi_G \bxi \leqslant \sigma_\xi 
    \sqrt{2 t} + b_\xi t \|\Pi_G\bw\|_\infty 
    \big) \geqslant 1-e^{-t},\qquad t\geqslant 0.
\end{align}
Let $V_G$ be the image of the unit ball of $\mathbb R^N$ by 
$\Pi_G$ and let $\mathcal N_G= \{\bw_1,\ldots,\bw_M\}$ be an 
$\varepsilon$-net of $V_G$ for $\varepsilon = 1/4$. The set 
$V_G$ being included in the unit ball of a linear space of 
dimension $k$,  \citep[Lemma 5.13]{Proba} shows that 
$M\leqslant 12^k$. Define
\begin{align}
    \bu_m = \underset{u \in V_G}{\arg\min}\{ \|\bu\|_\infty : \|\bu-\bw_m\|_2\le 1/4\},
    \qquad m=1,\ldots,M.
\end{align}
Using \eqref{lem5:eq1} and the union bound, we get
\begin{align}
    \prob\Big(\forall m\in[M],\ \bu_m^\top \Pi_G \bxi
    \leqslant \sigma_\xi \sqrt{2 (t + \log M)} +
    b_\xi(t + \log M) \|\bu_m\|_\infty \Big)
    \geqslant 1-e^{-t},\qquad t\geqslant 0,
\end{align}
where we used the fact that $\Pi_G\bu_m =\bu_m$. 
Let $\bw$ be an arbitrary vector from $V_G$. Let $\bw_m$ 
be any point from the net $\mathcal N_G$ such that 
$\|\bw - \bw_m\|_2\le 1/4$. We have
\begin{align}
    \|\bu_m\|_\infty \le \|\bw\|_\infty\qquad \text{and}
    \qquad  \|\bw - \bu_m\|_2\le 1/2.
\end{align}
This implies that, with probability at least 
$1-e^{-t}$, for any vector $\bw\in V_G$,
\begin{align}
    \bw^\top \Pi_G\bxi &= (\bw - \bu_m)^\top \Pi_G \bxi
    + \bu_m^\top \Pi_G\bxi\\
    & \leqslant \|\bw - \bu_m\|_2 \|\Pi_G\bxi\|_2 
    +\bu_m^\top \Pi_G\bxi\\
    &\leqslant (\nicefrac12) \|\Pi_G\bxi\|_2 + 
    \sigma_\xi \sqrt{2 (t + \log M)} + b_\xi
    (t + \log M) \|\bu_m\|_\infty\\
    &\leqslant (\nicefrac12) \|\Pi_G\bxi\|_2 + 
    \sigma_\xi \sqrt{2 (t + \log M)} + b_\xi(t + 
    \log M) \|\bw\|_\infty.
\end{align}
Since this inequality is valid for any vector $\bw$ 
in the image of the unit ball ball by $\Pi_G$, it is
also valid for $\bw = \bv/\|\Pi_G\bv\|_2$. 
Replacing this into the last display, then multiplying
the two sides of the inequality by $\|\Pi_G\bv\|_2$, 
we get the first claim of the lemma. 

For the second claim, we take $v= \bxi$ to get
\begin{align}
    \|\Pi_G\bxi\|_2^2 &\leqslant 2\sigma_\xi\sqrt{2 
    (t + \log M)}\|\Pi_G\bxi\|_2 + 2b_\xi (t + \log M)
    \|\Pi_G\bxi\|_\infty\\
    &\leqslant (\nicefrac12)\|\Pi_G\bxi\|_2^2 + 4 
    \sigma_\xi^2 (t + \log M) + 2b_\xi (t + \log M)
    \|\Pi_G\bxi\|_\infty,
\end{align}
where we used the inequality $2uv \leqslant 
{(\nicefrac12)u^2 + 2v^2}$ for all $u,v \in\mathbb R$. 
Rearranging the terms of the last display, we obtain 
the claim of the lemma.
\end{proof}

\begin{lemma}\label{partition2}
Let $\xi_1,\dots,\xi_N$ be independent zero-mean
random variables satisfying the $(\sigma_\xi^2,b_\xi)
$-Bernstein condition with $\sigma^2_\xi,b_\xi > 0$. 
Let $\{\calA_j: j \in J\}$  be families of subsets 
of $\{1,\ldots, N\}$ such that $N_j = \min_{A\in 
\calA_j} |A| \geqslant N_0$. With probability 
at least $ 1- e^{-t}$, for all $j\in J$, we have 
\begin{align}
    \max_{A\in\calA_j} \frac1{|A|}\Big| 
    \sum_{\ell\in A} \xi_\ell\Big| \leqslant 
    \sigma_\xi\sqrt{\frac{2(t + \log(2|J|))}{N_0} + 
    \frac{2\log|\calA_j|}{N_j}} + \frac{b_\xi(t + 
    \log(2|J|))}{N_0} +\frac{b_\xi\log |\calA_j|
    }{N_j}.
\end{align}
\end{lemma}

\begin{proof}
Using the version of \Cref{Bernstein:ineq} of the 
Bernstein inequality, we find
\begin{align}
    \prob\Big(\Big|\sum_{\ell\in A} \xi_\ell\Big|
    \leqslant \sigma_\xi\sqrt{2 |A| t} + b_\xi t  
    \Big)\geqslant 1-2e^{-t},
\end{align}
which yields 
\begin{align}
    \prob\bigg(\frac1{|A|}\Big|\sum_{\ell\in A}
    \xi_\ell\Big| \leqslant \sigma_\xi \sqrt{
    \frac{2t}{|A|}} + \frac{b_\xi t}{|A|} \bigg)
    \geqslant 1-2e^{-t}.
\end{align} 
It follows from the last display and the union 
bound that
\begin{align}
    \prob\bigg(\max_{A\in\calA_j} \frac1{|A|}\Big| 
    \sum_{\ell\in A} \xi_\ell\Big| \leqslant 
    \sigma_\xi\sqrt{\frac{2(t+\log|\calA_j|)}{N_j}} 
    + \frac{b_\xi(t + \log |\calA_j|)}{N_j} \bigg)
    \geqslant 1-2e^{-t}.
\end{align}
Taking the union bound over $j \in J$,  we get
that with probability at least $ 1- e^{-t}$, 
for all $j\in J$, it holds that
\begin{align}
    \max_{A\in\calA_j} \frac1{|A|}\Big| 
    \sum_{\ell\in A} \xi_\ell\Big| \leqslant 
    \sigma_\xi\sqrt{\frac{2(t + \log(2|J|) + \log|
    \calA_j|)}{N_j}} + \frac{b_\xi(t + \log(2|J|) + 
    \log |\calA_j|)}{N_j}.
\end{align}
This completes the proof of the lemma.
\end{proof}
%\end{appendices}

\end{document}